\documentclass[10pt,a4paper]{article}
\usepackage{}
\usepackage{verbatim}
\usepackage{setspace}
\usepackage[left=2.5cm,right=2.5cm,top=3.0cm,bottom=3.0cm,includeheadfoot]{geometry}
\usepackage{latexsym}
\usepackage{amssymb}
\usepackage{amsthm}
\usepackage{amsmath,bm}
\usepackage{tikz}
\usetikzlibrary{matrix}
\usetikzlibrary{positioning}
\usetikzlibrary{arrows}
\usepackage{graphicx}
\usepackage[font=small,labelfont=bf]{caption}
\usepackage{fullpage}

\usepackage{setspace}
\usepackage{enumitem}
\usepackage{mathrsfs}
\usepackage{subcaption}
\usepackage[english]{babel}
\usepackage{xcolor}

\newtheorem{theorem}{Theorem}[section]
\newtheorem{lemma}[theorem]{Lemma}

\theoremstyle{definition}
\newtheorem{definition}[theorem]{Definition}
\newtheorem{example}[theorem]{Example}

\theoremstyle{proposition}
\newtheorem{proposition}[theorem]{Proposition}

\theoremstyle{definition}
\newtheorem{remark}[theorem]{Remark}

\theoremstyle{corollary}
\newtheorem{corollary}[theorem]{Corollary}

\theoremstyle{question}
\newtheorem{question}[theorem]{Question}

\theoremstyle{conjecture}
\newtheorem{conjecture}[theorem]{Conjecture}

\theoremstyle{goal}
\newtheorem{goal}[theorem]{Goal}

\theoremstyle{convention}

\numberwithin{equation}{section}

\newcommand{\field}[1]{\mathbb{#1}}
\newcommand{\C}{\field{C}}
\newcommand{\R}{\field{R}}
\newcommand{\N}{\field{N}}
\newcommand{\Z}{\field{Z}}
\newcommand{\Q}{\field{Q}}

\newcommand{\tham}{(M,\omega,T,\phi)}

\newcommand{\id}{\operatorname{Id}}
\newcommand{\GKM}{(\Gamma_{GKM}, \eta)}
\newcommand{\skeleton}{(\Gamma, \sigma, \operatorname{Ord}(E^\sigma), \mathbf{d})}
\newcommand{\w}{\operatorname{w}}

\usepackage[latin1]{inputenc}

\title{\textsc{Monotone Symplectic Six-Manifolds that admit a Hamiltonian GKM Action are 
diffeomorphic to Smooth Fano 
Threefolds }}

\date{$^1$University of Haifa, Department of Mathematics\\
Campus at Mount Carmel,  Mount Carmel 3498838, Haifa,  Israel\\
icharton@math.haifa.ac.il\\
$^2$University of Haifa,
Department of Mathematics, Physics and Computer Science \\
Campus at Oranim, 36006, Tivon, Israel\\
lkessler@math.haifa.ac.il}

\begin{document}
	
\author{Isabelle Charton$^1$  and Liat Kessler$^2$}	
\maketitle
\date

\begin{abstract}
Let $(M,\omega)$ be a compact symplectic manifold with a Hamiltonian GKM action of a compact torus.
We formulate a positive condition on the space; this condition is satisfied if the underlying
symplectic manifold is monotone.
The main result of this article is that the underlying manifold 
of a positive Hamiltonian GKM space of dimension six is diffeomorphic to a smooth Fano threefold.  
We prove the main result in two steps.\\
In the first step, we deduce from results of Goertsches, Konstantis, and Zoller
that if  the complexity of the action is zero or one then the equivariant and the ordinary cohomology 
with integer coefficients are determined by the GKM graph. 
This result,  in combination with a classification result by Jupp, Wall and \u Zubr for certain six-manifolds, implies 
that the diffeomorphism type of a compact symplectic six-manifold with a Hamiltonian GKM action is determined by the 
associated GKM graph.\\
In the second step, based on results by Godinho and Sabatini, we compute the complete list of the GKM graphs of 
positive Hamiltonian GKM spaces of dimension six. 
We deduce that any such GKM graph is isomorphic to a GKM graph of a smooth Fano threefold.
\end{abstract}

\tableofcontents

\begin{section}{Introduction}
	A symplectic manifold $(M,\omega)$ admits an almost complex structure $J \colon TM\rightarrow TM$, $J^2=-\id$, that 
	is 
	\textbf{compatible} with the symplectic form $\omega$, i.e., $\omega(\cdot, J\cdot)$ is a Riemannian 
	metric. By picking a compatible almost complex structure, we can consider the tangent bundle  $TM$ 
	as a complex vector bundle over $M$. We denote by $c_i(M)$ the $i$-th Chern 
	class of $(M,\omega)$; the Chern classes are well defined since the space of compatible almost complex structures 
	is 
	contractible.
	A compact symplectic manifold $(M,\omega)$ is called \textbf{monotone} if $c_1(M)=r\cdot [\omega]$ for some $r\in 
	\R$;  $(M,\omega)$ is called \textbf{positive monotone} if $r>0$.
	
	The algebraic counterparts of positive monotone symplectic manifolds are {smooth Fano varieties}.
	A \textbf{smooth Fano variety} $X$ is a compact complex manifold whose anticanonical 
	line bundle 
	$K_X^{-1}$ is ample. The ampleness of $K_X^{-1}$ implies that there exists a holomorphic embedding
	$i: X \hookrightarrow \C P^N$ such that $(K_X^{-1})^k=i^* \mathcal{O} (1)$ for some $N>0$ and $k>0$, known as 
	\textbf{polarisation} of $X$ by  $K_X^{-1}$.
	The form $i^*(\omega_{FS})$ is a symplectic form on $X$, where 
	$\omega_{FS}$ is the Fubini-Study form on $\C P^N$. 
	Moreover, the  almost complex structure $J \colon TM \to TM$ induced from the complex structure on $X$ is 
	compatible 
	with  $i^*(\omega_{FS})$.
	Since $c_1(X)=c_1(K_X^{-1})$ and $c_1(i^* \mathcal{O} 
	(1))=[i^*(\omega_{FS})]$,  the symplectic manifold $(X, i^*(\omega_{FS}))$ is positive monotone.

	Since smooth Fano varieties carry a lot of geometric and algebraic structures it is
	important to understand in which content positive monotone symplectic manifolds are similar to
	smooth Fano varieties.
	In real dimensions two and four, it is known that any positive monotone symplectic manifold is diffeomorphic to a 
	smooth Fano 
	variety. In dimension two, this fact follows from the work of Morse \cite{Morse}.
	In dimension four, this fact is a result of Ohta and Ono \cite{OhtaOno} based on works of Gromov \cite{Gromov},
	McDuff \cite{McDuff} and Taubes \cite{Taubes}. 
	In dimension greater or equal to twelve, Fine and Panov \cite{FinePanov} provide examples of positive 
	monotone symplectic manifolds that are not diffeomorphic to a smooth Fano variety. In dimensions six, eight, and 
	ten, 
	the question of whether any positive monotone symplectic manifold is diffeomorphic to a Fano variety is open. 
	
	If a positive monotone symplectic manifold $(M,\omega)$ admits an integrable 
	almost complex structure $J \colon TM 
	\rightarrow TM$ that is compatible with $\omega$,  then $M$ endowed with the complex atlas indicated by 
	$J$ is a smooth Fano variety. This is a consequence of the Kodaira Embedding Theorem, see \cite[Sect. 
	14.4]{McDuffSalamon}.
	In particular, if a compact symplectic manifold $(M,\omega)$ of dimension $2n$ is endowed with an effective and 
	Hamiltonian action of a compact torus $T={(S^1)}^n$ of half the dimension, then $(M,\omega)$ admits a $T$-invariant 
	and 
	integrable almost 
	complex structure $J$ that is compatible with $\omega$  $\cite{Delzant}$; if $(M,\omega)$ is positive monotone, 
	then 
	$M$  with the induced complex atlas is a smooth Fano variety.
	Moreover, it is enough then that $(M,\omega)$ is monotone;
	a monotone symplectic manifold that admits an effective and Hamiltonian action of a compact torus 
	is positive monotone, see \cite[Proposition 3.3]{css} and \cite[Lemma 5.2]{1224}.
	Recall that an action of a compact torus $T={(S^1)}^d$ on a symplectic manifold is \textbf{Hamiltonian} if there 
	exists a \textbf{moment 
		map} 
	$\phi: M \rightarrow \mathfrak{t}^*$, i.e., 
	a smooth and $T$-invariant map  whose 
	codomain is the dual of the Lie algebra $\mathfrak{t}$ of $T$
	such that
	\begin{align*}
	\iota_{X^\xi} \omega = -\operatorname{d} \left\langle \phi, \xi\right\rangle \quad \text{for all } \,\xi \in 
	\mathfrak{t},
	\end{align*} 
	where $X^{\xi}$ is the vector field on $M$ generated by $\xi$ and $\left\langle ,\right\rangle $ is the natural 
	pairing between $\mathfrak{t}^*$ and $\mathfrak{t}$.
	We note that on a simply connected manifold, (as is a smooth Fano variety), a smooth $T$-action is symplectic iff 
	it is Hamiltonian: the action is symplectic iff $\iota_{X^{\xi}}\omega$ is closed for all $\xi \in \mathfrak{t}$ 
	(by Cartan's formula) and Hamiltonian iff $\iota_{X^{\xi}}\omega$ is exact for all $\xi \in \mathfrak{t}$.
	Given a connected symplectic manifold $(M,\omega)$
	endowed with an effective 
	Hamiltonian $T$-action with moment map $\phi$, we call the quadruple $\tham$ a \textbf{Hamiltonian 
	$\mathbf{T}$-space}. 
	Since the
	orbits of a Hamiltonian $T$-action on $(M,\omega)$ are isotropic,  
	$\operatorname{dim}(T)\leq 
	\frac{1}{2}\operatorname{dim}(M)$; we call
	$k:=\frac{1}{2}\operatorname{dim}(M)-\operatorname{dim}(T)$ the \textbf{complexity} of $\tham$. 
	We also call $\tham$ a \textbf{complexity $\mathbf{k}$ space}.
	
	The algebraic counterparts of Hamiltonian $T$-actions on symplectic manifolds are holomorphic actions of algebraic 
	tori on 
	compact complex manifolds. 
	Recall that an algebraic torus $T_{\C}=(\C^*)^d$ is the complexification of a compact torus $T=(S^1)^d$;
	in particular, $T$ is contained in $T_{\C}$. If $X$ is a smooth Fano
	variety endowed with a holomorphic $T_{\C}$-action, then its anticanonical line bundle is $T_{\C}$-invariant.
	Hence, the polarisation of $X$ by $K_X^{-1}$ is $T_{\C}$-equivariant with respect to a $T_{\C}$-representation
	on $\C P^N$. The induced $T$-action on $\C P^N$ is Hamiltonian with respect to the Fubini-Study  symplectic form 
	$\omega_{FS}$. Hence, also the induced $T$-action on $X$ is Hamiltonian with respect to $i^*(\omega_{FS})$.
	
	Since the underlying manifold of a compact  monotone  Hamiltonian $T$-space of complexity zero 
	is diffeomorphic to a smooth Fano variety, it is natural to investigate the following question. 
	
	\begin{question}\label{Question: Positive Complexity}
		Let $\tham$ be a Hamiltonian $T$-space of positive complexity whose underlying symplectic manifold is  monotone.
		\begin{center}
			Is $M$ diffeomorphic to a smooth Fano variety?
		\end{center}
	\end{question}
	
	In dimension six, i.e., the lowest dimension in which it is not known if any positive monotone symplectic 
	manifold is diffeomorphic to a smooth Fano variety, Fine and Panov stated in \cite{FinePanovCon} 
	the following conjecture.
	
	\begin{conjecture}\label{ConjectureFinePanov} Let $(M,\omega)$ be a monotone symplectic manifold  of 
		dimension six that admits an effective and  Hamiltonian $S^1$-action. Then $M$ is diffeomorphic to a smooth 
		Fano 
		threefold. 
	\end{conjecture}
	A recent work by  Lindsay and Panov \cite{LindsayPanov} provides results that support this conjecture.
	In a series \cite{ChoSemifree1, ChoSemifree2, ChoSemifree3} of papers, Cho classifies monotone symplectic  
	manifolds  of 
	dimension six admitting an effective and  Hamiltonian $S^1$-action that is
	semifree. He shows that such a space admits an $S^1$-invariant integrable almost complex 
	structure that is compatible with 
	the symplectic form. Hence, Conjecture \ref{ConjectureFinePanov} is true if the action is in addition semifree.
	But it is still open if  Conjecture \ref{ConjectureFinePanov} is true in general. 
	
	In this paper, we look at Question \ref{Question: Positive Complexity} in case the complexity is one.\\
	Karshon \cite{Karshon4} classifies compact complexity one spaces of dimension four, and shows that such a space
	admits an invariant compatible integrable almost complex structure.
	In dimension six and above, there are complexity one spaces that do not admit an invariant compatible integrable 
	almost 
	complex structure \cite{Tolman}. We note that these spaces might admit a non-invariant compatible integrable almost 
	complex structure \cite{GKZ2}.
	In any dimension, a special class of complexity one spaces, namely the {tall} ones,  are classified by Karshon 
	and  Tolman \cite{Tall1, Tall2, Tall3}. \textbf{Tall} means that for any $x$ in the dual Lie algebra  
	$\mathfrak{t}^*$ the reduced space $\phi^{-1}(x)/T$ is not a point. In \cite{css}, Sabatini, Sepe and the first 
	author 
	study compact  monotone tall complexity one spaces, and show that in such a space the Hamiltonian action extends to 
	an 
	effective Hamiltonian action of 
	complexity zero, and hence the space admits an invariant compatible integrable almost complex structure. This 
	result is based on the classification of Karshon and Tolman.
	However, so far there is no classification result for non-tall   Hamiltonian $T$-spaces of complexity one 
	in dimension greater than four. Sabatini and Sepe \cite{SabatiniSepe} prove that a compact complexity one  
	monotone space
	shares topological properties with smooth Fano varieties; namely, it is simply connected and its Todd  genus is one.

	The following is a special case of Conjecture \ref{ConjectureFinePanov}.
	\begin{goal}\label{MainGoal}
		Let $\tham$ be a  complexity one space where  $(M, \omega )$ is a  monotone symplectic  manifold  of dimension 
		six. 
		Then $M$ is diffeomorphic to a smooth Fano threefold. 
	\end{goal}
	By \cite{css}, this is true if the space is tall.  In this article, we prove Goal \ref{MainGoal} 
	under the assumption that the $T$-action is  {GKM}. 
	A Hamiltonian $T$-space is \textbf{GKM} if the set $M^T$ of fixed points is finite, for each codimensional one 
	subtorus 
	$H\subset T$ any connected component of $M^H$ has at most dimension two, and the underlying manifold is compact. 
	Note that a Hamiltonian GKM space is not tall. 
	
	We present our results in the following subsection.

\begin{subsection}{Results}
		
Given a Hamiltonian GKM space, one can associate to it canonically a graph, the so-called 
\textbf{GKM graph}; see Subsection \ref{SubSec: Hamiltonian GKM Spaces}.
Moreover,  any  edge $e$ of the GKM graph is associated in a canonical way to a $2$-dimensional 
symplectic two-sphere $S_e$  in the underlying symplectic manifold $(M,\omega)$. We call a Hamiltonian GKM space 
\textbf{positive} if for each edge $e$ the evaluation of the first Chern class of $(M,\omega)$ on $S_e$ is positive. 
We introduce the notion of a positive Hamiltonian GKM space in a careful way at the beginning of Section 
\ref{sec:pos}. In particular, if the underlying symplectic manifold is monotone then the Hamiltonian GKM space is 
positive (see Lemma \ref{Lemma: symplectic fano c1 positive}). 
		
\begin{theorem}\label{ManiThm: result main goal}
Let $\tham$ be a positive and six-dimensional Hamiltonian GKM space of complexity one. Then $M$ is 
diffeomorphic to a smooth  Fano threefold.
\end{theorem}
		
This theorem supports Goal \ref{MainGoal}. Indeed a direct consequence of Theorem \ref{ManiThm: result main goal} 
and Lemma \ref{Lemma: symplectic fano c1 positive} is the following corollary.

\begin{corollary}\label{ManiCor: result main goal}
Let $\tham$ be a Hamiltonian GKM space of complexity one where  $(M, \omega )$ is a monotone symplectic 
manifold of dimension six. Then $M$ is diffeomorphic to a smooth Fano threefold.
\end{corollary}
We prove Theorem \ref{ManiThm: result main goal}  in two steps.\\
In the {\bf first step}, we relate the equivariant cohomology of a Hamiltonian GKM space of complexity one and its GKM 
graph, and deduce that the diffeotype of a Hamiltonian GKM space of dimension six is determined by its GKM graph. The 
positive condition is not needed in this step. 
		
\begin{theorem}\label{ManiThm: coprime weights GKM graph determines eq. cohomology}
Let  $(M_1,\omega_1, T, \phi_1)$ and $(M_2,\omega_2, T, \phi_2)$ be two Hamiltonian GKM spaces of 
complexity one  or zero. If 
there exists an isomorphism from the GKM graph of
$(M_1,\omega_1, T, \phi_1)$ to the one of $(M_2,\omega_2, T, \phi_2)$, then this isomorphism induces 
\begin{itemize}
\item[(a)] a ring isomorphism in equivariant cohomology
\begin{align*}
H_T^*(M_1;\Z) \rightarrow H_T^*(M_2;\Z)
\end{align*}
that maps the equivariant Chern classes of $(M_1,\omega_1, T, \phi_1)$ to the ones of  $(M_2,\omega_2, T, \phi_2)$ and 
\item[(b)] a ring isomorphism in cohomology
\begin{align*}
H^*(M_1;\Z) \rightarrow H^*(M_2;\Z)
\end{align*}
that maps the Chern classes of $(M_1,\omega_1)$ to the ones of  $(M_2,\omega_2)$.
\end{itemize}
\end{theorem}
		
\begin{theorem}\label{ManiThm: dim six gkm graph determines diffeomorphisms type}
Let  $(M_1,\omega_1, T, \phi_1)$ and $(M_2,\omega_2, T, \phi_2)$ be two Hamiltonian GKM spaces of dimension six. If 
there exists an isomorphism from the GKM graph of $(M_1,\omega_1, T, \phi_1)$ to the one of $(M_2,\omega_2, T,\phi_2)$, 
then there exists a (non-equivariant) diffeomorphism $M_2\rightarrow M_1$.
\end{theorem}
		
In \cite[Theorem 3.1]{Goertsches}, Goertsches, Konstantis, and
Zoller show that an isomorphism between the GKM graphs of GKM spaces,
not necessarily Hamiltonian, induces a ring isomorphism in
(equivariant) cohmology over $\Z$, assuming that the GKM spaces are
compact and connected, the odd cohomolgy rings
vanish, and the spaces satisfy the following technical condition:
For any point $q\in M$ that does not lie in the one skeleton
$$M_{(1)}= \left\lbrace p\in M \, \vert \, \dim (T \cdot p) \leq 1
\right\rbrace$$
of $M$, its stabilizer is contained in a proper subtorus of $T$.
If, in addition, the GKM spaces are equipped with invariant
almost complex structures and the graph-isomorphism is an isomorphism of
signed GKM graphs, then the induced ring isomorphism also maps the
(equivariant) Chern classes to (equivariant) Chern classes.
The proof of  \cite[Theorem 3.1]{Goertsches} relies on a result of
Franz and Puppe
\cite[Corollary 2.2]{franzPuppe2}  for (not necessarily Hamiltonian) $T$-spaces.
Moreover, using a classification result for compact and simply
connected six-dimensional manifolds by Jupp
\cite{Jupp}, Wall \cite{wall} and \u Zubr \cite{Zubr}, Goertsches,
Konstantis, and
Zoller show that the obtained isomorphism of cohomologies is induced
by a diffeomorphism between the manifolds, and hence the conclusion of
Theorem \ref{ManiThm: dim six gkm graph determines diffeomorphisms
	type}
holds, if the GKM spaces are, in addition, six-dimensional and simply connected. 		
	
In Section \ref{sec:equiv}, we show that the results  of Goertsches, Konstantis, and 
Zoller imply Theorems \ref{ManiThm: coprime weights GKM graph determines eq. cohomology} and 
\ref{ManiThm: dim six gkm graph determines diffeomorphisms type}.
We reformulate their technical condition
in terms of the weights at the fixed points, in compact Hamiltonian $T$-spaces. This formulation allows us to 
show that the condition holds if the complexity of the $T$-action is equal to zero or one. We deduce that 
Theorem \ref{ManiThm: coprime weights GKM graph determines eq. cohomology} follows from  \cite[Theorem 3.1]{Goertsches}.
Moreover, we show that if the Hamiltonian GKM spaces are of dimension six then the conclusion of the theorem 
holds. Since Hamiltonian GKM spaces are simply connected (see Remark \ref{rem: simply connected}), we also 	deduce 
Theorem \ref{ManiThm: dim six gkm graph determines diffeomorphisms type}.\\
In Appendix \ref{sec: alt proof}, we give an alternative proof of Theorem \ref{ManiThm: coprime weights GKM 
graph determines eq. cohomology}, relying on the existence of Kirwan classes that form a basis of the 
equivariant cohomology of a compact Hamiltonian $T$-space with isolated fixed points as a module over 
$H^{*}(BT;\Z)$. We use Kirwan classes in the second step of our proof of Theorem \ref{ManiThm: result main goal} as 
well.\\
				
In the {\bf second step}, we give a complete list of the GKM graphs that are coming from positive 
six-dimensional Hamiltonian GKM spaces. 
This result is based on a work by Godinho and Sabatini \cite{GodinhoSabatini}.
In \cite{GodinhoSabatini}, Godinho and Sabatini consider effective and  Hamiltonian $S^1$-actions with only isolated 
fixed points on compact symplectic manifolds. They show that any such space admits a graph that describes the action 
and that the weights (i.e., elements of $\ell_{S^1}^*\cong \Z$) of the $S^1$-representations on the tangent spaces at 
the fixed points satisfy certain linear relations with respect to the graph. 
Note that such a graph is not unique. 
If the action extends to an effective and Hamiltonian $T$-action that is GKM, then the GKM graph of the corresponding 
Hamiltonian GKM space naturally provides a graph that describes the $S^1$-action.
Note that the GKM graph is unique. 
In the GKM case, the weights (i.e., elements in $\ell_T^*\cong \Z^d$) of the $T$-representations on the tangent spaces 
at the fixed points also satisfy these linear relations.
Furthermore, the GKM condition imposes additional constraints on the weights.
In Subsection \ref{subsec:GKM skeletons}, we analyze these linear relations and constraints. 
For that we introduce GKM skeletons. 
A GKM skeleton is an $n$-valent simple graph together with an integer label for each edge. 
In particular, we show how an (abstract) GKM graph can be constructed from a GKM skeleton, and we discuss the 
uniqueness of such an (abstract) GKM graph.
Note that any (abstract) GKM graph can be constructed from a GKM skeleton.\\
Moreover, let $\tham$ be a positive Hamiltonian GKM space of dimension six and let $\GKM$ be its GKM graph.
So, $\Gamma_{GKM} = (V_{GKM}, E_{GKM})$ is a $3$-valent graph, and the weight map $\eta: E_{GKM} \rightarrow 
\ell_{T}^*$ assigns each edge in the set $E_{GKM}$ to an element in the dual lattice of the torus $T$.
By the work of  Godinho and Sabatini \cite{GodinhoSabatini}, the cardinality of the vertex set $V_{GKM}$ is at most 16.
In particular, there are only finitely many GKM skeletons from which a GKM graph of a positive Hamiltonian GKM space of 
dimension six can be constructed.
Furthermore, we use Kirwan classes to formulate a necessary condition for an abstract GKM graph to be realizable by the 
GKM graph of a positive  Hamiltonian GKM space of dimension six. \\
In Section 6, we sum the results of Sections 4 and  5 to create a computer program.
In this program, we utilize a classification result for $3$-valent graphs with a small number of vertices from 
\cite{DataBaseCubicGraphs}.
This computer program provides a finite list of (abstract) GKM graphs that must include all GKM graphs coming from 
positive six-dimensional Hamiltonian GKM spaces. 
It turns out that each (abstract) GKM graph in this list is a GKM graph that is 
\textbf{coming from a holomorphic  GKM  $T$-action on a smooth Fano threefold $X$}, i.e., it is the GKM graph of a  
Hamiltonian GKM $T$-action on $(X,i^{*}\omega_{FS})$ that is induced from a holomorphic $T_{\C}$-action on $X$. 
(We show this directly; see also \cite{suss}).\\
We prove the following theorem.
\newpage
\begin{theorem}\label{ManiThm: sympetic fano GKM graphs coming from fano threefolds.}
Let $\tham$ be a positive six-dimensional Hamiltonian GKM space of complexity one. Then its GKM graph is 
isomorphic to a GKM graph that is coming from a holomorphic GKM $T$-action on a Fano threefold.  
\end{theorem}
		
Combining Theorem \ref{ManiThm: dim six gkm graph determines diffeomorphisms type} and Theorem \ref{ManiThm: 
sympetic fano GKM graphs coming from fano threefolds.} gives Theorem \ref{ManiThm: result main goal}.
\end{subsection}
	
\begin{subsection}{Structure of the Article}
		
\begin{itemize}
\item[\textbf{Section 2}] We introduce notations and recall preliminary results about 
equivariant cohomology and Hamiltonian GKM spaces that are needed in this paper.
\item[\textbf{Section 3}] We reformulate the technical condition of  \cite{Goertsches} in terms  of weights 
at fixed points in a compact Hamiltonian $T$-space and deduce Theorems 
\ref{ManiThm: coprime weights GKM graph determines eq. cohomology} and \ref{ManiThm: dim six gkm graph determines 
diffeomorphisms type} from \cite[Theorem 3.1]{Goertsches}.
\item[\textbf{Section 4}] We introduce positive Hamiltonian GKM spaces and prove basic properties of six 
dimensional ones. These results are needed to compute the list of graphs associated to such spaces.
\item[\textbf{Section 5}] We analyze the linear relations between the first Chern class maps and the weight 
maps of (abstract) GKM graphs.
\item[\textbf{Section 6}] We sum up the results of Sections 4 and 5 and explain how we use computer programs to 
classify the GKM graphs of positive six-dimensional Hamiltonian GKM spaces. We conclude by a proof of Theorem 
\ref{ManiThm: sympetic fano GKM graphs coming from fano threefolds.}.
			
\item[\textbf{Appendix A}] We give the complete list of GKM graphs of positive six-dimensional Hamiltonian 
GKM spaces  that are not projections of GKM graphs coming from smooth reflexive polytopes.
			
\item[\textbf{Appendix B}] We give an alternative proof, using Kirwan classes, of Theorem \ref{ManiThm: 
coprime weights GKM graph determines eq. cohomology}. 
\end{itemize}
\end{subsection}         
\end{section}

\textbf{Acknowledgments.}  The first author would like to thank Silvia Sabatini  for drawing her attention to the 
question of in which 
content smooth Fano varieties and monotone symplectic manifolds with a Hamiltonian GKM  action are different. We thank 
Leopold Zoller and Nikolas Wardenski  for clarifying remarks.
The first author began working on this project during her Ph.D. studies, which were supported 
by the SFB-TRR 191 grant Symplectic Structures in Geometry, Algebra and Dynamics funded by the German Research 
Foundation.
The second author is supported by the Israel Science Foundation, grant no. 570/20.

\begin{section}{Preliminaries on Equivariant Cohomology  and Hamiltonian GKM Spaces}

\begin{subsection}{Equivariant Cohomology}

Let $M$ be a topological  space endowed with a continuous action of a torus $T=(S^1)^d$. In the Borel model, the 
$T$-equivariant cohomology 
of  $M$ is defined as follows. Let $ET$ be a contractible space on which $T$ acts freely and let 
$BT=ET/T$ be the classifying space of $T$. The diagonal action of \(T\) on \(M\times ET\) is free.    
\(M \times_{T} ET\)  denotes the orbit space. The \(T\)-equivariant cohomology ring of \(M\) is
\begin{align*}
H_{T}^*(M;R)\colon =H^*(M \times_{T} ET;R),
\end{align*}
where \(R\) is the coefficients ring. In particular, if  the \(T\)-action on \(M\) is  trivial, then 
\begin{align*}
H_T^*(M;R)=H^*(M;R)\otimes H^*(BT;R).
\end{align*}
If $T=S^1$, then $ES^1$ is the unit sphere \(S^\infty\) in 
\(\C^\infty\) and \(BS^1\) is \(\C P^\infty \).  Hence
\begin{align*}
H_{S^1}^*(\left\lbrace \operatorname{point}\right\rbrace ;R)= H^*(\C P^\infty ;R)=R[x],
\end{align*}
where \(x\) has degree \(2\). Moreover, if \(T\) is a  \(d\)-dimensional torus, then \(BT\) is the 
\(d\)-times product of \(\C P^\infty\), and so
\begin{align*}
H_{T}^*(\left\lbrace \operatorname{point}\right\rbrace ;R)=	H^*(BT;R)=R[x_1,\dots, x_d],
\end{align*}
where \(\left\lbrace x_1,\dots, x_d\right\rbrace \) is a basis of  the dual lattice $\ell_T^*$ of $T$ and 
$\operatorname{deg}(x_i)=2$ for all $i=1,...,d$.
If the coefficient ring is $\Z$ then $H^2(BT;\Z)$ is equal to $\ell_T^*$ and if  the coefficient ring is 
$\R$ then $H^2(BT;\R)$ is equal to $\mathfrak{t}^*$.
 The projection map \(M \times ET \rightarrow ET\) is \(T\)-equivariant, so we  obtain a map
\begin{align*}
\pi:\,M \times_T ET \rightarrow BT.
\end{align*}
This makes \(M \times_T ET\)  an \(M\)-bundle over \(BT\),
\begin{align*}
M      \overset{r}\longrightarrow M \times_T ET \overset{\pi} \longrightarrow BT.
\end{align*}

So it induces a sequence of ring homomorphisms
\begin{align*}
H^*(BT;R) \overset{\pi^*}\longrightarrow H_T^*(M;R) \overset{r^*}\longrightarrow H^*(M;R).
\end{align*}
The map \(\pi^*\)  gives \(H_T^*(M;R)\) an \(H^*(BT;R)\)-module structure by
\begin{align*}
\alpha \cdot \beta =\pi^*(\alpha ) \cup \beta
\end{align*}
for \(\alpha \in  H^*(BT;R)\) and \(\beta \in H_{T}^*(M;R)\).

\begin{subsubsection}{Equivariant Chern Classes} 
Let $V\rightarrow M$ be a $T$-invariant vector bundle. Its 
equivariant Euler class $e^T(V)$ is the Euler class of the 
vector bundle 
\begin{align}\label{EQ: extension of vector bundles}
V\times_{T} ET \rightarrow M \times_{T} ET.
\end{align}
If  $V\rightarrow M$ is a complex $T$-invariant 
vector bundle, then \eqref{EQ: extension of vector bundles} is a complex vector bundle, and its $i$-th 
equivariant Chern class $c_i^T(V)$ is defined in the same way. These equivariant classes are mapped to the
ordinary Euler class resp. ordinary  Chern classes of  $V\rightarrow M$ under the map 
\begin{align*}
r^*: H_T^*(M;\Z)\rightarrow H^*(M;\Z)
\end{align*}
induced from $r: M \rightarrow M \times_T ET$. 
In case $M$ is a point,
a complex $T$-invariant vector bundle 
$V\rightarrow M= \{p\}$ of rank $k$ is  just  a complex vector space  of complex dimension $k$ with a 
$T$-representation. Let 
$\alpha_1,..., \alpha_k \in \ell_T^*$ be the weights of this representation, then the total equivariant Chern class of  
the complex vector bundle $V\rightarrow \{p\}$  is
\begin{align*}
c^T(V)= \prod_{i=1}^{k} \left( 1+\alpha_i \right)   \in H_T^*(\{p\}; \Z) \cong H^*(BT;\Z).
\end{align*}   
Hence, the $i$-th equivariant Chern class is
\begin{align*}
c_i^T(V)=\sigma_{k,i}(\alpha_1, \dots \alpha_k)\in H^{2i}(BT;\Z),
\end{align*}
where $\sigma_{k,i}$  is the  elementary symmetric polynomial in $k$ variables of degree $i$. \footnote{This means 
$\sigma_{k,0}=1$, $\sigma_{k,i}=0$  for $i>k$ and for $i=1,...,k$  $$\sigma_{k,i}(X_1,...,X_k)=\sum_{1\leq j_1 <...< 
j_i \leq k}X_{j_1}\cdot... \cdot X_{j_k}.$$}

In particular, the first equivariant Chern class is $c_1^T(V)=\sum_{i=1}^{k}\alpha_i$ and the equivariant Euler class  
is $e^T(V)=c_k^T(V)=\prod_{i=1}^{k}\alpha_i$.

\end{subsubsection}			
			
\begin{subsubsection}{The ABBV Localization Formula}
Let $M^T$ be the set of fixed points and assume that it is not empty. Let \(F\) be one of its connected 
components. The inclusion map \(i_F\colon F \rightarrow M\) is a \(T\)-equivariant  map, so it induces a map
\begin{align*}
i_F^*\colon H_T^*(M;R) \rightarrow H_T^*(F;R).
\end{align*}
Moreover, the map \(\pi \colon M \times_{T} ET \rightarrow BT \) induces a push-forward map in equivariant cohomology
\begin{align*}
H_{T}^*(M;R)\rightarrow H^{*-\operatorname{dim}(M)}(BT;R),
\end{align*}
which can be seen as integration along the fibers. So we denote it by $\int_M$. The following 
theorem, due to Atiyah-Bott \cite{AtiyahBott} and Berline-Vergne  \cite{BerlineVerngne}, gives a 
formula for the map \(\int_M\) in terms of the fixed point set data.
\begin{theorem}\label{Thm:ABBV}(ABBV Localization formula) Let \(M\) be a compact oriented manifold 
endowed with a smooth \(T\)-action. For \(\mu^T\in H_T^* (M;\Q) \)
\begin{align*}
\int_M \mu^T \,=\, \sum_{F\subset M^T} \int_F \dfrac{i^*_F\mu^T}{\text{e}^{T}(N_F)},
\end{align*}
where the sum is over all the connected components \(F\) of \(M^T\) and \(\text{e}^{T}(N_F)\) is the equivariant Euler 
class of the normal bundle $N_F \to F$. 
\end{theorem}
\end{subsubsection}

		\end{subsection}	

\begin{subsection}{Hamiltonian GKM Spaces}\label{SubSec: Hamiltonian GKM Spaces}
Given a $T$-action on a compact connected symplectic manifold of dimension $2n$, 
there exists a $T$-invariant 
almost complex structure $J: TM \rightarrow TM$ that is compatible with the symplectic form $\omega$; moreover, the 
space of $T$-invariant almost complex structures on $(M,\omega)$ is contractible (see e.g.,  
\cite[Proposition 4.1.1]{McDuffSalamon}). For each fixed 
point $p\in M^T$, the $T$-action induces a $T$-representation on $(T_pM, J_p)\cong \C^n$. This $T$-representation 
splits into a direct sum of one-dimensional $T$-representations, i.e., $T_pM=\oplus_{i=1}^n L_i$, where $L_i$ is a one-dimensional $T$-representation with weight $\alpha_{p,i}\in \ell_T^*$ for $i=1,...,n$. The 
elements $\alpha_{p,1},...,\alpha_{p,n}$ are called the \textbf{weights of the  $T$-representation on $T_pM$}.
These weights do not depend on the choice of such an almost complex structure since the space of $T$-invariant almost complex structures on $(M,\omega)$ is contractible. Hence, the weights are well-defined.
Since the $T$-action on $M$ is effective, the following lemma holds.

\begin{lemma}\label{Lemma: Z-span weights}
Given an effective $T$-action on a  connected symplectic $(M,\omega)$, 
let $p\in M^T$ be a fixed point and let $\alpha_{p,1},...,\alpha_{p,n}\in \ell_T^*$ be 
the weights of the  $T$-representation on $T_pM$. Then the $\Z$-span 
\begin{align*}
\left\lbrace \kappa_1 \cdot\alpha_{p,1}+\dots+ \kappa_n\cdot \alpha_{p,n} \, \vert \, \kappa_1,...,\kappa_n \in 
\Z\right\rbrace 
\end{align*} 
of the weights  $\alpha_{p,1},...,\alpha_{p,n}$ is equal to $\ell_T^*$.
\end{lemma}

\begin{proof}Since $T$ is compact there exist $T$-invariant neighborhoods $U$ resp. $V$ of $\mathbf{0}$ in $T_pM$ 
resp. of $p$ in 
$M$ that are  $T$-equivariantly diffeomorphic. Since the $T$-action on $M$ is effective and $M$ is connected, by the 
Principal Orbit Theorem \cite[Theorem 2.8.5]{DK}, the $T$-action 
on $V$ is effective. Hence, the $T$-action on $U$ is also effective. The $T$-action on $U$ is determined by the 
$T$-representation on $T_pM$ and is effective if and only if the  $\Z$-span of the weights 
$\alpha_{p,1},...,\alpha_{p,n}$ is 
equal to $\ell_T^*$.
\end{proof}


\begin{definition}\label{Def: Hamiltonian GKM spaces}
A\textbf{ Hamiltonian GKM space} is a \textbf{compact} Hamiltonian $T$-space $(M,\omega, T, \phi)$ such that the 
following hold.
\begin{itemize}
\item[(i)] The set $M^T$ of fixed points of the $T$-action on $M$ is finite.
\item[(ii)] For each $p\in M^T$, the weights $\alpha_{p,1},...,\alpha_{p,n}\in \ell_T^*$ of the  $T$-representation on 
$T_pM$ are pairwise linearly independent.
\end{itemize}
We say that the Hamiltonian $T$-action is \textbf{GKM}.
\end{definition}

In the following lemma, we point out an obstruction for the dimension of the torus of a Hamiltonian GKM space.

\begin{lemma}\label{Lemma: GKM forces  dim(T) geq 2}
Let $\tham$ be a Hamiltonian GKM space of dimension $2n\geq 4$. Then the dimension of the torus $T$ is greater or equal 
to $2$.
\end{lemma}
\begin{proof}
Since $M$ is compact, the set of fixed points $M^T$ is 
not empty. Let $p\in M^T$ be a fixed point and let $\alpha_{p,1},...,\alpha_{p,n}\in \ell_T^*$ be the weights of 
the $T$-representation on $T_pM$. Since, by (ii) of Definition \ref{Def: Hamiltonian GKM spaces}, 
these weights are pairwise linearly independent, the dimension of the torus must be greater or equal to two. 
\end{proof}

Given a Hamiltonian GKM space $(M,\omega, T, \phi)$, the second condition of Definition \ref{Def: Hamiltonian GKM 
spaces} implies that for each codimensional one subtorus $H$ of $T$, any connected component of 
\begin{align*}
M^H:=\left\lbrace  p\in M \, \mid \, t\cdot p =p\, \text{ for all } t\in H\right\rbrace 
\end{align*}
has at most dimension two. Let $S$ be a two-dimensional component of $M^H$. Then $S$ is a symplectic embedded 
two-sphere and there exist exactly two points $p,q\in M^T$ such that $S \cap 
M^T=\{p,q\}$. Moreover, the inclusion  $H \hookrightarrow T$ induces an $H$-representation on $T_pM$ and the tangent 
$T_pS$ is equal to the subspace of $T_pM$ that is fixed by the $H$-representation. Hence, there exists an element  
$\alpha\in \ell_T^*$ that is a weight of the  $T$-representation on 
$T_pM$ such that 
\begin{align}\label{eq:hexp}
H=\exp (\operatorname{ker}\, \alpha),
\end{align}
where $\exp: \mathfrak{t} \rightarrow T$ is the exponential map and 
\begin{align*}
\operatorname{ker}\, \alpha = \left\lbrace \xi \in \mathfrak{t}\, \vert \, \left\langle \alpha, \xi \right\rangle=0
\right\rbrace,
\end{align*}
where $\left\langle ,\right\rangle$ is the natural pairing between $\mathfrak{t}^*$ and $\mathfrak{t}$.
Vice versa, $-\alpha$ is a weight of the  $T$-representation on $T_qM$.
Given $p,q \in M^T$ there exists at most one two-sphere $S$ that is fixed by a codimensional one subtorus such that $S \cap M^{T}=\{p,q\}$. This is since a point fixed by two different codimensional one subtori is in $M^{T}$. 
On the other hand, for $p \in M^T$ and a weight $\alpha$ of the $T$-representation on $T_{p}M$,
there is a codimensional one subtorus $H$ such that the connected component of $M^H$ that contains $p$ is
 a symplectic embedded two-sphere $S$ and the weight of the $T$-representation on $T_{p}S$ is $\alpha$; moreover, the subtorus $H$ must be as in (\ref{eq:hexp}), so this sphere is unique.
\newline 
The information about the $T$-representations on the tangent spaces of the fixed points and about the two-dimensional  
components fixed by a codimensional one subtorus can be stored in a graph, the so-called \textbf{GKM graph}, which is 
defined as follows.

\begin{definition}\label{Def: GKM Graph}
Let $(M,\omega, T, \phi)$ be a Hamiltonian GKM space. Its \textbf{GKM graph} $(\Gamma_{GKM}, \eta)$ is the graph 
$\Gamma_{GKM}=(V_{GKM},E_{GKM})$ with directed edges together with the map $\eta: E_{GKM}\rightarrow \ell_T^*$
defined as follows.
\begin{itemize}
\item[(i)] The set of vertices $V_{GKM}$ is equal to the set of fixed points $M^T$.
\item[(ii)] The set of edges is a subset of $V_{GKM}\times V_{GKM}$. There exists a directed edge 
$$e=(p,q)\in V_{GKM}\times V_{GKM}$$ if and 
only if there exists a two-sphere $S$ that is fixed by a codimensional one subtorus $H$ of $T$ such that $S\cap 
M^T=\{p,q\}$. 
\item[(iii)]  For each edge $e=(p,q)\in E_{GKM}$, we denote by $S_{(p,q)}$ the unique two-sphere as described in $(ii)$.
Then $\eta(p,q)$ is the weight of the $T$-representation on $T_pS_{(p,q)}$.
\end{itemize}
\end{definition} 
			
In the following remark, we sum up some properties of GKM graphs.

\begin{remark}\label{Rem: Properties GKM graphs}
Let $(\Gamma_{GKM}, \eta)$ be the GKM graph of a Hamiltonian GKM space $\tham$, then the 
following properties hold.
\begin{itemize}
\item[(i)] Let $2n$ be the dimension of $M$. Then the graph $\Gamma_{GKM}=(M^T, E_{GKM})$ is  \textbf{$n$-valent}, 
i.e., for each $p\in M^T$ there exist exactly $n$ points $p_1,\dots, p_n\in M^T \setminus \{p\}$ such that 
$(p,p_i)\in E_{GKM}$ for $i=1,...,n$. Moreover, $\eta(p,p_1),...,\eta(p,p_n)$ are the weights of the 
$T$-representation on $T_pM$.
\item[(ii)] Given $p,q\in M^T$, then $(p,q)$ belongs to $E_{GKM}$ if and only if $(q,p)$ belongs to $E_{GKM}$.
If $(p,q)\in E_{GKM}$, then $\eta(q,p)=-\eta(p,q)$.
\end{itemize}
\end{remark}


\begin{definition}\label{Def: isomophic GKM graphs}
Given two Hamiltonian GKM spaces $(M_1,\omega_1, T, \phi_1)$ and $(M_2,\omega_2, T, \phi_2)$ of the same dimension, let 
$(\Gamma_{1,GKM}, \eta_1)$  and $(\Gamma_{2,GKM}, \eta_2)$ be their GKM graphs.
An \textbf{isomorphism} between these GKM graphs  is a pair $(F,\theta)$ such that 
 
\begin{itemize}
\item $F$ is an isomorphism between the underlying graphs 
$$\Gamma_{1,GKM}=(M_1^T,E_{1,GKM})\quad \text{and} \quad \Gamma_{2,GKM}=(M_2^T,E_{2,GKM}),$$
i.e., $F: M_1^T \rightarrow M_2^T$ is a bijection such that for any two points $p,q\in M_1^T$
$$(p,q)\in E_{1,GMK}\quad \text{if and only if}\quad  (F(p), F(q)) \in E_{2,GKM},$$
\item $\theta: \ell_T^* \rightarrow \ell_T^*$ is a linear isomorphism such that

\begin{align*}
\theta \left(\eta_1 (p, q)\right) =  \eta_2(F(p),F(q)),
\end{align*}
for all $(p,q)\in E_{GKM}$.
\end{itemize}
\end{definition}

Given a Hamiltonian GKM space $\tham$, let $\GKM$ be its GKM graph. The \textbf{initial map} 
$i:E_{GKM}\rightarrow V_{GKM}$ and the \textbf{terminal map} $t:E_{GKM}\rightarrow V_{GKM}$ are given by
\begin{align*}
i(p,q)=p \quad \text{and} \quad t(p,q)=q \quad \text{for every edge } e=(p,q) \in E_{GKM}.
\end{align*}
We associate to each point $p\in M^T= V_{GKM}$ the following two subsets of $E_{GKM}$, 
\begin{align*}
E_{GKM}^{p,i}:= \{e\in E_{GKM} \mid i(e)=p\} \quad \quad \text{and} \quad \quad E_{GKM}^{p,t}:= \{e\in E_{GKM} \mid 
t(e)=p\}.
\end{align*}

A \textbf{connection} along an edge $e=(p,q)\in E_{GKM}$ is a bijection 
\begin{align*}
\nabla_e: E_{GKM}^{p,i} \longrightarrow E_{GKM}^{q,i}
\end{align*}
such that $\nabla_e(p,q)=(q,p)$. The \textbf{connection is compatible} if for  all $e' \in E_{GKM}^{p,i} $ there 
exists an integer $a_{e,e'}\in \Z$, such that
\begin{align*}
\eta(e')-\eta(\nabla_e(e'))=a_{e,e'}\cdot\eta(e).
\end{align*}

Due to the next lemma, there always exists a compatible connection.

\begin{lemma}\label{Lemma: Existence Connection}
Let $\tham$ be a Hamiltonian GKM space and let  $\GKM$ be its GKM graph. For each edge $e \in 
E_{GKM}$ of $\Gamma_{GKM}$, there exists a connection along $e$ that is compatible.
\end{lemma}

This lemma is by \cite[Theorem 1.1.2]{GuilleminZara}.

\begin{remark}\label{Rem: Independence of GKM graph}
Let $\tham$ be a Hamiltonian GKM space. Its GKM graph does not depend on the moment map $\phi$.
Moreover, if $\lambda\in \R$ is non-zero, then the $T$-action on $(M,\lambda\cdot\omega)$ is also Hamiltonian and GKM 
and  $\phi'=\lambda\cdot \phi$ is a moment map. If $\lambda>0$, then the GKM graphs of $\tham$ and 
$(M,\lambda\cdot \omega, T, \lambda\cdot \phi)$ are the same.   
\end{remark}

\begin{remark}\label{Rem: images weight relation}
Let $\tham$ be a Hamiltonian GKM space and let $\GKM$ be its GKM graph. Let $e=(p,q)\in E_{GKM}$ be an edge. 
The moment map images $\phi(p)$ and $\phi(q)$ determine $\eta(p,q)$ up to positive integer multiple. Namely,
let $S$ be the unique two-sphere in $M$ that is fixed by a codimensional one subtorus of $T$ such that 
$S\cap M^T=\{p,q\}$. The moment  map image $\phi(S)$ is given by the compact line segment through  $\phi(p)$
and $\phi(q)$ in the vector space $\mathfrak{t}^*$. This line segment is rational, i.e., there exists a unique
primitive vector $\alpha\in \ell_{T}^* \setminus \{0\}\subset \mathfrak{t}^*\setminus \{0\}$ such that 
$\phi(q)-\phi(p)=r \cdot \alpha $ for some real number $r>0$. In particular,
$$\eta(p,q)=m\cdot \alpha  \quad \text{for some positive integer } m.$$
\end{remark}

\begin{subsubsection}{GKM Graphs of Symplectic Toric Manifolds}

By the famous Convexity Theorem of Atiyah \cite{Atiyah} and Guillemin-Sternberg \cite{GSconvex}, the moment map 
image of a compact Hamiltonian $T$-space is a convex polytope. So the moment map image is called the \textbf{moment map 
polytope}. Moreover, compact symplectic toric manifolds, i.e., compact complexity zero spaces, are determined by their 
moment map polytopes (up to isomorphism) and the moment map polytope of such a space is smooth \cite{Delzant}. 
We recall this and the definition of smooth polytopes in the following.  

\begin{definition}\label{Def: smooth polytope}
Let $\Delta$ be a $d$-dimensional polytope in $\R^d$. A vertex $v$  of $\Delta$ is called \textbf{smooth} if 
there exist exactly $d$ edges of $\Delta$ that contain $v$ and a $\Z$-basis $\alpha_{v,1},...,\alpha_{v,d} $
of $\Z^d$ such that the edges that contain $v$ are given by 
$$\left( v+\R_{\geq0}\cdot \alpha_{v,j}\right)  \cap \Delta $$ 
for $j=1,...,d$. The polytope $\Delta$ is called \textbf{smooth} if all of its vertices are smooth. 
\end{definition}

\begin{remark}\label{Rem: smooth polytopes in the dual Lie algebra}
Let $T$ be a $d$-dimensional torus. By considering its dual Lie algebra $\mathfrak{t}^*\cong \R^d$ together with the 
dual lattice $\ell^*_T\cong \Z^d$, the definition of smooth polytopes in $\R^d$ naturally extends to polytopes in 
$\mathfrak{t}^*$. 
\end{remark}

\begin{theorem}\label{Thm: Delzant}\cite[Delzant]{Delzant}
Let $\tham$ be a compact symplectic toric manifold. The space admits an integrable almost complex structure 
that is T-invariant and compatible with $\omega$ and the moment map polytope is a smooth polytope.
Moreover, if $\Delta$ is a smooth polytope in $\mathfrak{t}^*$, then there exists a unique (up to isomorphism) compact 
symplectic toric manifold whose moment map image is $\Delta$.
\end{theorem}

Let $\tham$ be a compact symplectic toric manifold. Then the space is also GKM and its GKM graph $\GKM$ is determined 
by the moment map polytope $\Delta_M\colon =\phi(M)$, as follows. 
The moment map induces a bijection from the set of fixed points to the set of vertices of $\Delta_M$. Hence, the set 
$V_{GKM}$ of vertices of the  graph $\Gamma_{GKM}$ can be identified with the set of vertices of the polytope 
$\Delta_M$. There exists an edge $(p,q)\in E_{GKM}\subset M^T \times M^T$ if and  only if  the corresponding vertices 
$\phi(p)$ and $\phi(q)$ of $\Delta_M$ are connected  by an edge of the polytope $\Delta_M$. If $(p,q)\in E_{GKM}$, then 
$\eta(p,q)$ is the unique primitive element in $\ell_{T}^*$ such that $\phi(q)-\phi(p)=r\cdot \eta(p,q)$ for some 
$r>0$.  In the following, we say that 
\textbf{the GKM graph of a compact symplectic  toric manifold $\tham$ is coming from the smooth polytope 
$\Delta_M\colon =\phi(M)$}. We note that the GKM graph does not determine the moment map polytope. In particular, 
the GKM graph does not contain information about the symplectic form.

\begin{example}\label{Example: GKM do not give polytope}
In Figure \ref{Figure: Smooth and Reflexive}, the polytopes (A), (B) and (C) are smooth.
\begin{itemize}
\item The polytope (A) is the moment map polytope of the standard $(S^1)^2$-action on $(\C P^2, \omega_{FS})$.
\item The polytope (B) is the moment map polytope of the standard  $(S^1)^2$-action on $(S^2\times S^2, \omega\oplus 
\omega)$, where $\omega$ is volume form on $S^2$ with $\int_{S^2} \omega =2$. 
\item The polytope (C) is the moment map polytope of the standard  $(S^1)^2$-action on $(S^2\times S^2, \omega\oplus 
\frac{5}{4}\omega)$.
\end{itemize}
The polytopes (A) and (B) induce the same GKM graph.
\end{example}

\begin{definition}\label{Def: reflexive polytope}
Let $\Delta$ be a $d$-dimensional polytope in $\R^d$ that is \textbf{integral}, i.e., all of its vertices lie in $\Z^d$.
For each facet $F$ of $\Delta$, let $l_F\in \Z^d$ be the primitive outward normal vector to the hyperplane defining
$F$, i.e., 
$$F= \left\lbrace x \in \R^d \, \vert \, \left\langle l_F, x \right\rangle_{\R^d} =c_F \right\rbrace \cap \Delta
\quad \text{and}\quad \Delta\subseteq  \left\lbrace x \in \R^d \, \vert \, \left\langle l_F, x \right\rangle_{\R^d} 
\leq c_F 
\right\rbrace
$$
for some  $c_F \in \R$, where $\left\langle , \right\rangle _{\R^d}$ is the standard scalar product on $\R^d$.
The polytope $\Delta$ is \textbf{reflexive} if  $c_F=1$ for any facet $F$ of $\Delta$. 
\end{definition}

\begin{figure}[htbp]
\begin{center}
\includegraphics[width=12cm]{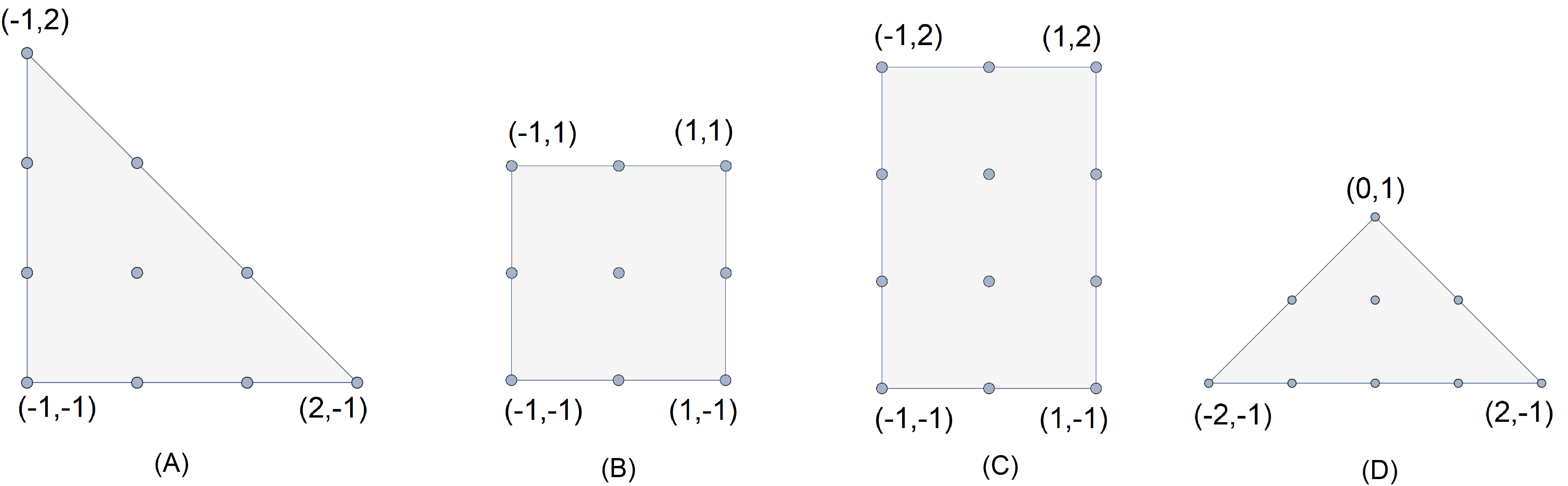}
\caption{The polytopes (A) and (B) are smooth and reflexive. The polytope (C) is smooth, but not reflexive. The 
polytope (D) is reflexive, but not smooth.}
\label{Figure: Smooth and Reflexive}
\end{center}
\end{figure}
Monotone symplectic toric manifolds and reflexive polytopes are related as follows. 
\begin{proposition}\label{Pro: toricReflMono}
Let $\Delta$ be a smooth polytope in $\mathfrak{t}^*$ and let $\tham$ the compact symplectic toric manifold whose 
moment map image is $\Delta$. Then the following two conditions are equivalent.
\begin{itemize}
\item There exists $\nu \in \mathfrak{t}^*$ such that $\Delta+\nu$ is reflexive.
\item $(M,\omega)$ is monotone with $c_1(M)=[\omega]$.
\end{itemize}
\end{proposition}
A proof of this proposition can be found in \cite[Proposition 1.8]{EntovPolterovich} and \cite[Sect.3]{McDuffMonotone}

\end{subsubsection}

\begin{subsubsection}{Projections of GKM Graphs}
Given a Hamiltonian GKM space $\tham$, let $T'$ be a  subtorus of $T$. 
The restriction of the $T$-action to $T'$ is again a Hamiltonian action on $(M,\omega)$ with moment map
\begin{align*}
i^*\circ \phi: M \rightarrow (\mathfrak{t'})^*,
\end{align*}
where $i^*: \mathfrak{t}^* \rightarrow (\mathfrak{t}')^*$ is the dual  of the inclusion $i: \mathfrak{t}' 
\hookrightarrow \mathfrak{t}$ from the Lie algebra of $T'$ to the one of $T$.
So the quadruple 
\begin{align*}
(M,\omega, T', i^*\circ \phi)
\end{align*}
is a compact Hamiltonian $T$-space. The map  $i^*$ maps the dual lattice $\ell_T^*$ of $T$ to the dual lattice
$\ell_{T'}^*$ of $T'$. Note that the quadruple is a Hamiltonian GKM space if and only if for each fixed point $p\in 
M^T$, the  weights  of the $T$-representation on $T_pM$ are mapped to pairwise linearly independent  weights. 
In this case the GKM graph of $(M,\omega, T', i^*\circ \phi)$ is $(\Gamma_{GKM}, i^* \circ \eta)$, where 
$\GKM$ is the GKM graph of $\tham$. We say that
the GKM graph $(\Gamma_{GKM}, i^* \circ \eta)$  is a \textbf{projection} of $(\Gamma_{GKM}, \eta)$.
\end{subsubsection}
\end{subsection}


\begin{subsection}{Compact Hamiltonian $T$-spaces with only Isolated Fixed Points}

\begin{subsubsection}{Generic Vectors and Morse Theory}\label{SubSec: Generic Vectors and Morse Theory}
Let $\tham$ be a compact Hamiltonian  $T$-space with only isolated fixed points. A vector $\xi \in \mathfrak{t}$ is 
called\textbf{ generic} if $\left\langle \alpha, \xi \right\rangle \neq 0$ for each weight $\alpha \in \ell_T^*$
of the $T$-representation of $T_pM$ for every $p\in M^T$. If $\xi \in \mathfrak{t}$ is generic, then the 
$\xi$-component 
of the moment map
\begin{align*}
\phi^\xi : M \rightarrow \R\, , \quad p \mapsto \left\langle \phi (p), \xi \right\rangle 
\end{align*}
is a Morse function and its set of critical points coincides with the set of fixed points $M^T$.

Given $p\in M^T$, the Morse index of $p$  with respect to $\phi^\xi$ is equal to twice the number of weights of the 
$T$-representation on $T_pM$ that satisfy $\left\langle \alpha, \xi\right\rangle <0$.

For any fixed  point $p\in M^T$ we define the \textbf{index} $\lambda (p)$  of $p$ with respect to $\xi$ as follows.
\begin{align*}
\lambda (p) \,\text{equals half of  the  Morse index of } p\, \text{with respect to  } \phi^\xi,
\end{align*}
or equivalently,
\begin{center}
	$\lambda(p)$ is the number of weights of the  $T$-representation on $T_pM$ that satisfy
	$\left\langle \alpha, \xi \right\rangle <0.$
\end{center}


\begin{remark}\label{rem: simply connected}
In a Hamiltonian $T$-space with only isolated fixed points, for each
$p \in M^T$  its  Morse index with respect to $\phi^\xi$ (for a generic vector $\xi$) is even.
Hence, if $M$ is also compact, it is homotopy equivalent to a CW
complex with only even cells. This implies that $M$ is simply
connected, $H^{\operatorname{odd}}(M;\Z)=0$ and the odd Betti numbers of $M$ are zero.
\end{remark}

\begin{remark}\label{Rem: Morse and GKM}
Assume that $\tham$ is a Hamiltonian GKM space and let $\GKM$ be its GKM graph.
Given a fixed point $p\in M^T$, let $p_1,\dots,p_{n}\in M^T\setminus\{p\}$ be the unique fixed points such that 
$(p,p_i)\in E_{GKM}$. Since $\eta(p,p_1),..., \eta(p,p_n)$ are the weights of the $T$-representation on $T_pM$, we have 
that  $\lambda(p)$ is equal to the number of edges of $(p,p_1),..., (p,p_n)$ that satisfy
\begin{align*}
	\phi^\xi(p_i)< \phi^\xi(p), \, \text{or  equivalently, }\, \left\langle \eta(p,p_i),\xi \right\rangle <0.
\end{align*}
\end{remark} 

\begin{remark}\label{Rem: Hamiltonian GKM Graphs are connected}
Note that since the underlying manifold of a Hamiltonian GKM space is connected, its GKM graph $\GKM$ is also 
connected. 
The latter means that for two different points $p,q\in M^T$, there exists a sequence 
$p_0,..., p_k \in M^T$ such that $p_0=p$, $p_k=q$ and $(p_i,p_{i+1})\in E_{GKM}$ for $i=0,\dots,k-1$. 
We say that $p$ and $q$ are \emph{connected by a path}. 
To see this fix a generic vector $\xi \in \mathfrak{t}$ and let $p_{min}$ be
the unique fixed point on which $\phi^\xi$ attains its minimum. Let $p\in M^T \setminus \{p_{min}\}$ be another fixed 
point.
We set $p_0=p$. By Remark \ref{Rem: Morse and GKM} there exists a fixed point $p_1\in M^T\setminus \{p_0\}$ with
$\phi^\xi(p_1)<\phi^\xi(p_0)$. It might be that $p_1=p_{min}$; otherwise, we repeat this step. Note that since $M^T$ is finite,
we conclude that there exist $p_0,..., p_k$ such that $p_0=p$, $p_k=p_{min}$ and $(p_i,p_{i+1})\in E_{GKM}$ for 
$i=0,\dots,k-1$. Hence, each $p\in M^T\setminus\{p_{min}\}$ is connected with $p_{min}$. Moreover, since $(p,q)\in 
E_{GKM}$ implies $(q,p)\in E_{GKM}$, it follows that any two fixed points are connected by a path.
\end{remark}
\end{subsubsection}

\begin{subsubsection}{Kirwan Classes}\label{sec: kirwan}
Let $\tham$ be a compact Hamiltonian $T$-space with only isolated fixed points. By the Kirwan Injectivity Theorem 
\cite{Kirwan}, the inclusion $i:M^T \hookrightarrow 
M$ induces an injective map $i^*: H_T^*(M;\Z) \rightarrow H_T^*(M^T; \Z)$. Therefore, we consider  $H_T^*(M;\Z)$ as a 
subring of $H_T^*(M^T; \Z)$.  Since $H^*_T(\{p\};\Z)\cong H^*(BT;\Z)$, we consider the ring $H_T^*(M^T; \Z)$ as the 
ring 
of 
maps from $M^T$ to $H^*(BT;\Z)$, denoted by
$$\operatorname{Maps}\left( M^T, H^*(BT;\Z)\right).$$
So any class $\alpha\in H_T^*(M;\Z)$ is 
completely determined by the restrictions
\begin{align*}
i^*_{\{p\}}(\alpha) \in H^*_T(\{p\}; \Z) \cong H^*(BT;\Z)
\end{align*}  
for all $p\in M^T$, where $i^*_{\{p\}}$ is the map induced by the inclusion $p\hookrightarrow M$. For simplicity, in 
the following we write $\alpha(p)$ instead of $i^*_{\{p\}}(\alpha)$. Since we can consider $H^*_T(M;\Z)$ as a 
subring of $\operatorname{Maps}\left( M^T, H^*(BT;\Z)\right)$, we have that 
\begin{align*}
H^{2i}_T(M;\Z)= H_T^*(M;\Z)	\cap \operatorname{Maps}\left( M^T, H^{2i}(BT;\Z)\right) .
\end{align*}

We use results by Kirwan \cite{Kirwan} on the existence of 
so-called Kirwan classes. We repeat the definition of Kirwan classes. Let $\tham$ be a compact Hamiltonian 
$T$-space with only isolated fixed points and let $\xi \in \mathfrak{t}$ be a generic vector. For each  $p\in M^T$ 
the \textbf{equivariant Euler class of negative tangent bundle} at $p$ with respect to $\phi^\xi$ is
\begin{align*}
\Lambda_p^-= \prod_{\substack{ i=1,...,n\\\left\langle a_{p_i}, \xi\right\rangle<0  }} \alpha_{p,i},
\end{align*}
where $\alpha_{p,1},...,\alpha_{p,n}$ are the weights of the $T$-representation on $T_pM$.
Here, the empty product is equal to the multiplicative identity of $H^*(BT;\Z)$.

\begin{lemma}\label{Lemma: Kirwan Classes}(Kirwan \cite{Kirwan})
	Let $\tham$ be a compact Hamiltonian $T$-space with only isolated fixed points and let $\xi \in \mathfrak{t}$ be a 
	generic vector. For every fixed point $p\in M^T$, there exists an equivariant  class $\gamma_p\in 
	H_T^{2\lambda(p)}(M;\Z)$ such that
	\begin{itemize}
		\item[(i)] $\gamma_p(p)=\Lambda_p^-$ and 
		\item[(ii)] $\gamma_p(q)=0$  for every fixed point $q\in M^T\setminus\{p\}$ with $\phi^\xi(q)\leq\phi^\xi(p)$. 
	\end{itemize}
	Moreover, for any choice of such classes,  the set $\{\gamma_p\}_{p\in M^T}$ is a basis for $H_T^*(M;\Z)$ as a 
	module over 
	$H^*(BT;\Z)$.
\end{lemma} 
A class that satisfies properties (i) and (ii) of Lemma \ref{Lemma: Kirwan Classes} is called a \textbf{Kirwan 
	class} at $p$. 

In general, Kirwan classes are not unique. Let  $p$ and $q$ be fixed points and let $\gamma_p$ and $\gamma_{q}$ be 
Kirwan classes at $p$ and $q$. If $\lambda(p)\leq\lambda(q)$ and $\phi^\xi(p)\leq \phi^\xi(q)$, then 
$\gamma_p+f\cdot\gamma_q$ is also a Kirwan class at $p$ for any $f\in H^{2(\lambda(p)-\lambda(q))}(BT;\Z)$. 
In order to recover the multiplicative structure of $H_T^*(M;\Z)$  it is enough to compute the 
restrictions $\gamma_p(q)$ for all $p,q\in M^T$ for some set of Kirwan classes $\{\gamma_p\}_{p\in M^T}$. 
In this case one can easily compute the\textbf{ equivariant structure constants} with respect to this basis; these are, 
for all $q,p,r\in M^T$,  the unique elements $c^r_{p,q}\in H^*(BT;\Z)$ such that $\gamma_p\cdot \gamma_q= \sum_{r\in 
	M^T}c^r_{p,q} \gamma_r$. Furthermore, if one knows these equivariant structure constants, then the ordinary 
cohomology ring is also known. Indeed, by the Kirwan Surjectivity Theorem \cite{Kirwan}, the restriction map
\begin{align*}
r^*:H_T^*(M;\Z) \rightarrow H^*(M;\Z)
\end{align*}
is surjective and its kernel is the ideal generated by $H^2(BT;\Z)$. If we denote by $\tau_p$  the image of $\gamma_p$ 
for $p\in M^T$, then $H^{2i}(M;\Z)\cong \Z^{b_{2i}(M)}$ is the free abelian group generated by the elements $\tau_p$ 
with 
$\lambda(p)=i$ and $H^{odd}(M;\Z)=0$. The multiplicative structure of $H^*(M;\Z)$ is given by
\begin{align*}
\tau_p\cdot \tau_q = \sum_{\substack{r\in M^T \\ \lambda(r)=\lambda(p)+\lambda(q)}} c_{p,q}^r \tau_r\, ,
\end{align*}
where the empty sum is equal to zero.
\begin{remark}\label{Rem: Compute classes with respect to basis}
	Let $\alpha\in H_T^*(M;\Z)$ be a class and let $\{\gamma_p\}_{p\in M^T}$ be a set of Kirwan classes.
	For each $p\in M^T$, there exists a unique element $c_p\in H^*(BT;\Z)$ such that  $\alpha= \sum_{p \in M^T} 
	c_p\cdot 
	\gamma_p$. These 
	coefficients can be computed as follows. Given $p\in M^T$ and suppose that we already know $c_q$ for all 
	$q\in 
	M^T$ with $\phi^\xi(q)< \phi^\xi(p)$. We have
	\begin{align}\label{EQ1:Rem: Compute classes with respect to basis}
	\alpha(p)= c_p \cdot \Lambda_p^- + \sum_{\substack{q\in M^T \\\phi^\xi(q)< \phi^\xi(p) }} c_q \cdot \gamma_{q}(p).
	\end{align}
	Since $\Lambda_p^-\neq 0$, $c_p$ can be simply recovered from equation  \eqref{EQ1:Rem: Compute classes 
		with respect to basis}.
\end{remark} 
A simple consequence of \eqref{EQ1:Rem: Compute classes with respect to basis} is the following.
\begin{lemma}\label{Lemma: Properties of classes}
	Let $\tham$ be a Hamiltonian $T$-space with only isolated fixed points, let $\xi \in \mathfrak{t}$ be a generic 
	vector 
	and let $\alpha \in H_T^*(M;\Z)$.
	\begin{itemize}
		\item[(i)] If $p$ is a fixed point such that $\alpha(q)=0$ for all $q\in M^T$ with $\phi^\xi(q)<\phi^\xi(p)$,
		then $\alpha(p)=f\cdot \Lambda_p^-$ for some $f\in H^*(BT;\Z)$.
		\item[(ii)] If $\alpha \in H^{2i}_T(M;\Z)$ and $\alpha(q)=0$ for all $q\in M^T$ with $\lambda(q)\leq i$, then 
		$\alpha=0$.
	\end{itemize}
\end{lemma}

\end{subsubsection}
 
 \end{subsection}
\end{section}
	
\begin{section}{The Equivariant Cohomology of a Hamiltonian GKM Space is determined by its GKM Graph}\label{sec:equiv}
	
In this section, we prove the following proposition that is needed for the proof of Theorems \ref{ManiThm: coprime 
weights GKM graph determines eq. cohomology} and \ref{ManiThm: dim six gkm graph determines diffeomorphisms type}
that we give at the end of this section.

\begin{proposition}\label{Pro: complexity one proper subtorus}
Let $\tham$ be a compact Hamiltonian $T$-space of complexity zero or one. Let $p\in M$ be a point that does not lie in 
the one 
skeleton
$$M_{(1)}= \left\lbrace q \in M \, \vert \, \dim(T \cdot q) \leq 1 \right\rbrace.$$
Then the stabilizer of $p$ is contained in a proper subtorus of $T$.
\end{proposition}

\begin{subsection}{Subgroups of $T $ defined by subsets of $\ell_{T}^{*}$}
Let $T$ be a compact $d$-dimensional torus. We identify its Lie algebra 
$\mathfrak{t}$ with $\R^d$. Let $ \exp: \mathfrak t \rightarrow T$
be the exponential map. The lattice of $T$ is 
$$ \ell_T= \ker \left( \exp \right) \cong \Z^d.$$
Let $\mathfrak{t}^*$ be the dual Lie algebra. The dual lattice is
$$\ell_T^* = \left\lbrace \alpha \in \mathfrak{t}^* \, \vert \, \alpha(\xi)\in \Z \, \text{for all }\xi \in 
\ell_T\right\rbrace \cong \Z^d. $$ 
A subset $A \subset \ell_T^*$  defines a  subgroup of $\mathfrak t$, namely
\begin{align*}
\mathcal{K}(A):=\left\lbrace  \xi \in \mathfrak t \, \vert \, \alpha(\xi)\in \Z \,\, \text{for all } 
\alpha \in A\right\rbrace .
\end{align*}
If $A$ is empty, we set $\mathcal K (A)=\mathfrak t$. Through the exponential map $A$ defines also a subgroup 
of $T$, namely $\exp \left(  \mathcal K(A)\right)$ .

\begin{example}\label{Ex: expK1}
If the $\Z$-span of the elements of $A$ is equal to $\ell_T ^*$, then $\mathcal K (A)$ is $\ell_T $ and 
$\exp \left( \mathcal K (A)\right) $ is trivial.
\end{example}

A vector in $\ell_T^*$ is called \textbf{primitive} if  there exists no integer $p\geq2$
that divides the vector in $\ell_T^*\cong \Z^d$. In particular, a primitive vector is non-zero. 
Let $\alpha\in \ell_T^*$ be a non-zero vector. There exist a unique primitive vector $v$ and a unique integer $k\geq1$ 
such that $\alpha=k \cdot v$. Namely, $k$ is the largest integer that divides $\alpha$ in $\ell_T^*$. In particular,
$\alpha$ is primitive if and only if $k=1$.

\begin{example}\label{Ex: expK2}
Let $\alpha \in \ell_T^*$ be a non-zero vector. Let $k\geq 1$ be the unique integer such that $\alpha= k \cdot v_1$,
where $v_1\in \ell_{T}^*$ is primitive. Since $v_1$ is primitive, there exist vectors $v_2,...,v_d$ such 
that $v_1,...,v_d$ is a $\Z$-basis of $\ell_{T}^*$. Let $\xi_1,...,\xi_d\in \ell_T$ be the dual basis, i.e.,
$v_i(\xi_i)=1$ and $v_i(\xi_j)=0$ if $i\neq j$. We have that
\begin{align*}
&\left\lbrace  \xi \in \mathfrak t\, \vert \, \alpha(\xi) =0\right\rbrace = 
\left\lbrace \lambda_2\cdot \xi_2+...+ \lambda_d\cdot \xi_d \, \vert \, \lambda_2,..., \lambda_d \in 
\R\right\rbrace \quad \text{and} \\
&\mathcal K (\{\alpha\})= \frac{1}{k}\Z \cdot \xi_1 \oplus \left\lbrace  \xi \in \mathfrak t\, \vert \, 
\alpha(\xi) 
=0\right\rbrace.
\end{align*} 
If $\alpha$ is not primitive, i.e, $k\geq2$, then 
$\exp \left( \mathcal K (\{\alpha\})\right) =Z \otimes H,$ 
where $Z\cong \Z_k$ is the cyclic subgroup of $T$ that is generated by $\exp \left( \frac{1}{k}\cdot \xi_1\right)$ and 
$H$ is the codimensional one subtorus of $T$ whose Lie algebra is 
$\left\lbrace \xi \in \mathfrak t \, \vert \, \alpha(\xi)=0 \right\rbrace.$
If $\alpha$ is primitive, i.e., $k=1$, then $\exp\left(  \mathcal K (\{\alpha\})\right) =H$.
\end{example}
		
\begin{definition}\label{Def: coprime}
Given two vectors $\alpha_1,\alpha_2\in \ell_{T}^*$, we say that the vectors are \textbf{coprime} if there 
exists no integer $p\in \Z_{\geq 2}$ that divides both vectors in $\ell_T^*$. 
\end{definition}

\begin{remark}\label{Rem: Coprime and primitive vectors}
Let $\alpha_1,\alpha_2\in \ell_{T}^*$ be two vectors. If $\alpha_1$ is primitive, then the vectors are coprime.
On the other side, if $\alpha_1$ is zero, then the vectors are coprime if and only if $\alpha_2$ is primitive.
Moreover, if the vectors are linearly dependent, then there exist a primitive vector $v$ and integers $m_1,m_2\in \Z$
such that $\alpha_1=m_1 \cdot v$ and $\alpha_2=m_2 \cdot v$. In this case, $\alpha_1$ and $\alpha_2$ are coprime
if and only if  $\gcd(m_1,m_2)=1$ \footnote{If $m_1=m_2=0$, then $\gcd(m_1,m_2)=0$.}. 
\end{remark}
		
\begin{lemma}\label{Lem: two weights cotained in subtours}
Let $\alpha_1$ and $\alpha_2$ be two vectors in $\ell_T^* \cong \Z^d$. \hspace{-0.2cm}The vectors are coprime if and 
only if  $\exp \left( \mathcal K (\{\alpha_1, \alpha_2\})\right)$ is contained in a proper subtorus of $T$. 
\end{lemma}
\begin{proof}
If 	$\alpha_1$ and $\alpha_2$ are coprime, then by Lemma \ref{Lemma: existence primitive vector} below, there exists a 
primitive vector $\alpha \in \ell_T^*$ that lies in the $\Z$-span of $\alpha_1$ and 
$\alpha_2$. Hence, for any $\xi \in \mathcal K (\{\alpha_1, \alpha_2\}) $, we have $\alpha(\xi)\in \Z$. This 
implies
$$ \mathcal K (\{\alpha_1, \alpha_2\}) \subset \mathcal K (\{\alpha\}) \quad \text{and} \quad 
\exp\left( \mathcal K (\{\alpha_1, \alpha_2\}) \right)  \subset \exp\left(  \mathcal K (\{\alpha\})\right).$$
Since $\alpha$ is primitive, $\exp\left( \mathcal K (\{\alpha\})\right)$ is a codimensional one subtorus  of $T$ (see 
Example \ref{Ex: expK2}).\\
\noindent On the other side, assume that $\alpha_1$ and $\alpha_2$ are not coprime. There exists an integer $p\geq 2$
that divides $\alpha_1$ and $\alpha_2$ in $\ell_T^*$. Hence, $\mathcal K (\{\alpha_1, \alpha_2\})$ contains 
$\frac{1}{p}\ell_T$ and $\exp\left( \mathcal K (\{\alpha_1, \alpha_2\}) \right)$ contains all elements of $T$ of order 
$p$. Hence, $\exp\left( \mathcal K (\{\alpha_1, \alpha_2\}) \right)$ does not lie in a proper subtorus of $T$.
\end{proof}

\begin{lemma}\label{Lemma: existence primitive vector}
Let $\alpha_1$ and $\alpha_2$ be two vectors in $\ell_T^* \cong \Z^d$ that are 
coprime. Then there exists a primitive vector in the $\Z$-span of $\alpha_1$ and $\alpha_2$.
\end{lemma}
\begin{proof}
First assume that $\alpha_1$ and $\alpha_2$ are linearly dependent. Then there exist a primitive vector $v$ and 
integers $m_1,m_2\in \Z$ such that $\alpha_1=m_1 \cdot v$, $\alpha_2=m_2 \cdot v$ and $\gcd(m_1,m_2)=1$ (see Remark
\ref{Rem: Coprime and primitive vectors}). Since $\gcd(m_1,m_2)=1$, there exist integers $k_1,k_2 \in \Z$ such that 
$k_1\cdot m_1+k_2\cdot m_2=1$. Hence, $k_1\cdot \alpha_1+ k_2 \cdot \alpha_2=v.$\\ 
\hspace{0.3cm}Now assume that $\alpha_1$ and $\alpha_2$ are linearly independent. In particular, both vectors are 
non-zero. Let $k\geq1$ be the unique integer and let $v_1\in \ell_{T}^*$ be the unique primitive vector such that 
$\alpha_1=k\cdot v_1$. Since $v_1$ is primitive, it can be extended to a $\Z$-basis $v_1,v'_2...,v'_d$ of $\ell_T^*$.
So we have $\alpha_2=m\cdot v_1+ w$, where $m\in \Z$ and $w$ lies in the $\Z$-span of $v'_2...,v'_d$. Since $\alpha_1$
and $\alpha_2$ are linearly independent, the vector $w$ is non-zero. So let $l>0$ be the unique integer and $v_2$ be 
the unique primitive vector such that $w=l\cdot v_2$. In particular, $v_2$ lies in the $\Z$-span of $v'_2...,v'_d$. 
Hence, $v_2$ can be extended to a $\Z$-basis $v_2,...,v_d$ of the $\Z$-span of $v'_2...,v'_d$. In particular, 
$v_1,...,v_d$ is a $\Z$-basis of $\ell_T^*$ and 
\begin{align*}
\alpha_1=k \cdot v_1 \quad \text{and} \quad \alpha_2=m\cdot v_1 + l\cdot v_2.
\end{align*}
Since $\alpha_1$ and $\alpha_2$  are linearly independent and  coprime, we have that 
$k,l\neq 0$ and $\gcd(k,m,l)=1$.
Consider the following set 
$$S=\left\lbrace  q\in \Z_{\geq2}\, \text{is prime}\, \vert\,  q\mid l, \, q\nmid k \, \text{and} \,q\nmid 
m\right\rbrace. $$
If $S$ is empty, we set $p=1$, otherwise we set $p=\prod_{q\in S} q$. By construction $p$ is an integer such 
that for any prime integer $q$ the following holds
\begin{align*}
q\mid p \quad \text{if and only if }\quad \left(  q\mid l, \, q\nmid k \,\, \text{and}\,\, q\nmid m\right). 
\end{align*}

The vector 
\begin{align}\label{EQ:Lemma: existence primitive vector}
p\cdot\alpha_1+ \alpha_2= (pk+m)\cdot v_1+ l\cdot v_2
\end{align}
lies in the $\Z$-span of $\alpha_1$ and $\alpha_2$. We show that \eqref{EQ:Lemma: existence primitive vector} is 
primitive.\\ 
Since $\ell_T^*\cong \Z^d$ is a lattice in the $\R$-vector space  $\mathfrak t^* \cong \R^d$ of full 
rank
and $v_1,\dots,v_d$ is a $\Z$-basis of $\ell_T^*$ , $v_1,\dots,v_d$ is an $\R$-basis of  $\mathfrak t^*$, i.e., for a 
vector 
$v\in \mathfrak t^*$ there exist unique coefficients 
$\lambda_1,...,\lambda_d\in \R$ such that 
$$v=\lambda_1\cdot v_1+...+ \lambda_d \cdot v_d.$$
In particular, $v\in \ell_{T}^*$ if and only if $\lambda_1,...,\lambda_d\in \Z$.\\
Hence, in order to prove that the vector \eqref{EQ:Lemma: existence primitive vector} is primitive,  we need to show 
that any prime integer that divides $l$, does not divide $pk+m$.

So let $q$ be a prime integer that divides $l$. There are three case, namely $\left( q \mid k\right) $ or 
$\left( q\mid m\right) $ or both do not hold.\\
\noindent\textbf{1.case ($q \mid k$):} If $q$ divides $k$, then $q$ does not divide $m$, because $\gcd(k,m,l)=1$.
Hence, $q$ does not divide  $pk+m$.\\
\noindent\textbf{2.case ($q \mid m$):} If $q$ divides $m$, then
\begin{itemize}
\item $q$ does not divide $k$, because $\gcd(k,m,l)=1$,
\item $q$ does not divide $p$, by the construction of $p$.
\end{itemize}
Since $q$ is prime, it does not divide $pk$ and $pk+m$.\\
\noindent\textbf{3.case ($q \nmid k$ and $q \nmid m$):} By construction $q$ divides $p$.
Hence, $q$ does not divide  $pk+m$.
\end{proof}
\end{subsection}

\begin{subsection}{Isotropy Groups of Compact Hamiltonian $T$-Spaces}
In the following we relate the isotropy groups of a (compact) Hamiltonian $T$-space and the weights of the 
$T$-representations on tangent spaces of fixed points. Let $\tham$ be a Hamiltonian $T$-space of dimension 
$2n$ and let $p\in M^T$ be a fixed point. There exist complex coordinates $z_1,\dots,z_n$ centered at $p$ such that the 
$T$-action is given by
$$\exp (\xi)\cdot (z_1,...,z_n)= (\operatorname{e}^{2\pi i \alpha_{p,1}(\xi)}\cdot z_1,...,\operatorname{e}^{2\pi i 
\alpha_{p,n}(\xi)}\cdot z_n),$$ 
where $\alpha_{p,1},\dots, \alpha_{p,n}\in \ell_{T}^*$ are the weights of the $T$-representation on $T_pM$.
For a point $q=(z_1,...,z_n)$, its stabilizer is given by
$$ \exp \left( \mathcal K (\{a_{p,j}\, \vert z_j \neq 0\})\right).$$
We obtain the following lemma.
\begin{lemma}\label{Lem: isogroups near p}
Let $\tham$ be a Hamiltonian $T$-space and let $p\in M^T$ be a fixed point. There exists an open neighborhood $U$ of 
$p$ such that the following hold.
\begin{itemize}
\item For any $q\in U$ there exists a subset $A$ of the weights of the $T$-representation on $T_pM$ such that the 
stabilizer of $q$ is $\exp \left( \mathcal K (A)\right) $. 
\item For any subset  $A$ of the weights of the $T$-representation on $T_pM$, there exists a point $q\in U$ such that 
the stabilizer of $q$ is $\exp \left( \mathcal K (A)\right)$.
\end{itemize}
\end{lemma}

Moreover, if $\tham$ is also compact, then the following holds.

\begin{lemma}\label{Lemma: iso Groups Hamm}
Let $\tham$ be a compact Hamiltonian $T$-space and let $H$ be an isotropy  subgroup of $T$. Then there exists 
a  fixed point  $p$ and a subset $A$ of the set of the weights of the $T$-representation on $T_pM$ such that 
$$H= \exp \left( \mathcal K (A)\right).$$
\end{lemma}
\begin{proof}
Since $H$ is an isotropy  subgroup of $T$, there exists a point $q\in M$ such that the stabilizer of $q$ is $H$.
Let $X$ be the connected component of $M^H$ that contains $q$. So $X$ is a compact and symplectic 
submanifold of $M$. Since $T$ is abelian, $X$ is also $T$-invariant. The $T$-action on $(X,\omega\vert_X)$
is also Hamiltonian. Hence, $X\cap M^T$ is not empty, because $X$ is compact. Moreover, since $T$ is abelian, by the 
Principal Orbit Theorem \cite[Theorem 2.8.5]{DK},
the set of points of $X$ whose stabilizer is $H$ is dense in $X$. 
Let $p\in X \cap M^T$ be a fixed point and let $U$ be an open neighborhood of $p$ in $M$ as in Lemma \ref{Lem: 
isogroups near p}. Now there 
exists a point $q'\in U$  whose stabilizer  is $H$ and the lemma follows.
\end{proof}

\begin{proposition}\label{Pro: proper subtours = coprime}
Let $\tham$ be a compact Hamiltonian $T$-space. The following two conditions are equivalent.
\begin{itemize}
\item[(1)] Let  $p\in M^T$ be a fixed point and let $\alpha_1, \alpha_2 \in \ell_T^*$ two linearly 
independent vectors that occur as weights of the $T$-representation on $T_pM$. Then $\alpha_1$ and $\alpha_2$ are 
coprime.
\item[(2)] For any point $q\in M$ that does not lie in the one skeleton 
$$M_{(1)}= \left\lbrace p\in M \, \vert \, \dim (T\cdot p) \leq 1 \right\rbrace $$
of $M$, its stabilizer is contained in a proper subtorus of $T$.
\end{itemize}
\end{proposition}
\begin{proof}
$\mathbf{(1)\implies (2)}:$ Assume that $(1)$ holds. Let $q\in M$ be a point and let $H$ be its stabilizer. 
By Lemma \ref{Lemma: iso Groups Hamm} there exist a fixed point $p\in M^T$ and a subset $A$ of the weights 
$\alpha_{p,1},..., \alpha_{p,n}$ of the 
$T$-representation on $T_pM$ such that $H=\exp \left( \mathcal K (A)\right).$ If the dimension of the $\R$-span
of the elements in $A$ is equal to zero resp. one, then the dimension of $H$ is $d$ resp. $d-1$, where $d=\dim (T)$. 
Hence, $q \in M_{(1)}$.
So let us assume that the dimension of the $\R$-span
of the elements in $A$ is greater or equal to $2$. Hence, $A$ contains two weights that are linearly 
independent, say $\alpha_{p,1}, \alpha_{p,2}\in A$. Because $\{\alpha_{p,1}, \alpha_{p,2}\} \subset A$,
we have $$ \mathcal K (A) \subset \mathcal K (\{\alpha_{p,1},\alpha_{p,2} \})\quad \text{and}
		\quad \exp \left( \mathcal K (A) \right) \subset \exp  \left( \mathcal K (\{\alpha_{p,1},\alpha_{p,2} 
		\})\right) . $$
Since the condition $(1)$ holds,   $\alpha_{p,1}$ and $ \alpha_{p,2}$ are  coprime. Hence, by Lemma \ref{Lem: two 
weights cotained in subtours}, 
$\exp  \left( \mathcal K (\{\alpha_{p,1},\alpha_{p,2} \})\right)$ is contained in a proper subtorus 
of $T$ and so is $H=\exp \left( \mathcal K (A)\right).$\\
\noindent { $\mathbf{(2)\implies (1)}:$} Assume that $(2)$ holds. Let $\alpha_1, \alpha_2\in \ell_T^*$ be two 
linearly independent vectors that occur as weights of the $T$-representation on $T_pM$ for some fixed point $p$. We 
need to show that these weights are coprime. By Lemma \ref{Lem: isogroups near p}, there exists a point $q$ whose 
stabilizer is  $H=\exp  \left( \mathcal K 
(\{\alpha_{p,1},\alpha_{p,2} \})\right)$. Since the weights are linearly independent, the dimension of $H$ is 
$d-2$ and $q\notin M_{(1)}$. Since $(2)$ holds, $H$ is contained in a proper subtorus of $T$. By Lemma \ref{Lem: two 
weights cotained in subtours}, we have that $\alpha_1$ and $\alpha_2$ are coprime. 
\end{proof}
\end{subsection}

\begin{subsection} {Proof of Proposition \ref{Pro: complexity one proper subtorus}}
The missing ingredient for the proof of Proposition \ref{Pro: complexity one proper subtorus} is the following lemma.

\begin{lemma}\label{Lemma: complexity one, zero dim six implies weights pairwise coprime}
Let $\tham$ be a Hamiltonian $T$-space. Assume that the complexity of the $T$-action is equal to zero or one.
Then for each fixed point $p\in M^T$ the weights of the $T$-representation on $T_pM$ are pairwise coprime.  
\end{lemma}
\begin{proof}
Let $2n$ be the dimension of the manifold $M$ and let $d$ be the dimension of the torus $T$. So the complexity 
of the $T$-action is $k=n-d$. Let $p\in M^T$ be a fixed point and let $\alpha_{p,1},\dots, \alpha_{p,n}$ be the 
weights of the $T$-representation on $T_pM$. By Lemma \ref{Lemma: Z-span weights}, the $\Z$-span of the weights 
$\alpha_{p,1},\dots, \alpha_{p,n}$ is equal to $\ell_T^*\cong \Z^d$. If the complexity is equal to zero, i.e., $d=n$, 
then the $n$ weights form a $\Z$-basis of $\ell_T^*$. In particular, all the weights are primitive and so they are 
pairwise coprime. Now assume that the complexity is equal to one, i.e., 
\vspace{-0.1cm}
$d=n-1$. Let $\beta_1,...,\beta_{n-1}$ be a $\Z$-basis of $\ell_T^*\cong \Z^{n-1}$ and let 
		\begin{align*}
		\operatorname{det}\colon \left( \ell_{T}^*\right)^{n-1}\rightarrow \Z
		\end{align*}
		be the determinant map such that
		\begin{align*}
		\operatorname{det}\left( \beta_1,...,\beta_{n-1}\right) =1.
		\end{align*}
		For each $i=1,...,n-1$ there exist integers $A_{i,1},...,A_{i,n}$ such that $\beta_i=\sum_{j=1}^{n} 
		A_{i,j}\cdot 
		\alpha_{p,j}$. Therefore, we have
		\begin{align}\label{EQ:Lemma: complexity one, zero dim six implies weights pairwise coprime}
		\operatorname{det}\left( \beta_1,...,\beta_{n-1}\right) \,=\, \sum_{j=1}^{n} C_j \cdot 
		\operatorname{det}\left( \alpha_{p,1},...,\widehat{\alpha_{p,j}},...,\alpha_{n}\right) ,
		\end{align}
		where $C_1,...,C_n$ are integers. Now let $m\in\Z \setminus \{0\}$ be an integer that divides at least 
		two of the weights $\alpha_{p,1},\dots, \alpha_{p,n}$ in $H^2(BT;\Z)$. Then each of the $n$ terms in the 
		right-hand 
		side of 
		\eqref{EQ:Lemma: complexity one, zero dim six implies weights pairwise coprime}
		is an integer multiple of $m$. Since the left-hand side of \eqref{EQ:Lemma: complexity one, zero dim six 
		implies 
			weights pairwise coprime} is equal to $1$, we have $m=\pm 1$. 
Hence, the weights are indeed pairwise coprime.
\end{proof}
	
\begin{proof}[Proof of Proposition \ref{Pro: complexity one proper subtorus}]
The proposition follows directly from Proposition \ref{Pro: proper subtours = coprime} and Lemma \ref{Lemma: complexity 
one, zero dim six implies weights pairwise coprime}.
\end{proof}

\end{subsection}

\begin{subsection}{Proof of Theorems \ref{ManiThm: coprime weights GKM graph determines eq. cohomology} and 
\ref{ManiThm: dim six gkm graph determines diffeomorphisms type}}
In this subsection  we deduce Theorems \ref{ManiThm: coprime weights GKM graph determines eq. cohomology} and 
\ref{ManiThm: dim six gkm graph determines diffeomorphisms type}  from results of Goertsches, Konstantis, 
and Zoller \cite{Goertsches}  and 
Proposition \ref{Pro: complexity one proper subtorus}. In order to do so, we need to clarify some notations.
In this article we consider Hamiltonian $T$-actions on compact symplectic manifolds that are GKM, i.e,  
Hamiltonian GKM spaces.  The GKM condition can be also defined for smooth $T$-actions on compact (non 
symplectic) manifolds. In the literature there are various definitions of GKM actions, e.g., sometimes it is assumed 
that the odd cohomology  of the underlying manifold vanishes and sometimes not. Here we refer to the GKM condition as 
in \cite{Goertsches}. Namely, a smooth torus action on a compact and orientable manifold is GKM if 
$M^T$ is a finite set of points, and the one skeleton $M_{(1)}$ is a finite union of T-invariant two-spheres.
In this case there exists a graph associated to the action, namely the unsigned
GKM graph. In the GKM graph, a vertex corresponds to a fixed point, an
(unoriented) edge between two vertices corresponds to an invariant
two-sphere $S$ connecting the corresponding two fixed points in $S$,
and any edge is labeled by the weight mod $\pm 1$ of the corresponding two-sphere.
A Hamiltonian GKM space is a GKM
space; indeed, the underlying manifold is symplectic hence orientable,
and, by definition, the manifold is compact, $M^T$ is a finite set of points and $M_{(1)}$
is a finite union of T-invariant two-spheres.
If a compact and orientable manifold with a GKM action admits 
a $T$-invariant almost complex structure, then the action admits a signed GKM graph (see \cite{Goertsches}). A compact 
Hamiltonian 
$T$-space admits a $T$-invariant almost complex structure and  the signed GKM graph coincides with our definition of 
a GKM  graph of a Hamiltonian GKM space. 
In particular, our definition of isomorphisms between Hamiltonian GKM graphs coincides 
with the definition of isomorphisms between signed GKM graphs in \cite{Goertsches}.       

\begin{proof}[Proof of Theorems \ref{ManiThm: coprime weights GKM graph determines eq. cohomology} and
\ref{ManiThm: dim six gkm graph determines diffeomorphisms type}]
Let  $(M_1, \omega_1, T, \phi_1)$ and $(M_2, \omega_2, T, \phi_2)$ be  Hamiltonian GKM spaces of complexity one or 
zero. By definition, a Hamiltonian GKM space is a compact, connected GKM space.
By Remark \ref{rem: simply connected}, the manifolds $M_1$ and $M_2$ are simply connected and $H^{odd}(M_1;\Z)=0$ 
and $H^{odd}(M_2;\Z)=0$. By Proposition \ref{Pro: complexity one proper subtorus},
for any $p\in M_1$ that does not lie in the one skeleton of $M_1$, its stabilizer lies in a proper subtorus of 
$T$ and the same holds for $M_2$.\\
 Now assume that there exists an isomorphism from the GKM graph of $(M_1, \omega_1, T, 
\phi_1)$ to the one of $(M_2, \omega_2, T, \phi_2)$. 
In the terms of \cite{Goertsches}, this is an isomorphism of signed GKM graph.
Then, by the proof of \cite[Proposition 3.4]{Goertsches}, this isomorphism
induces a ring isomorphism $H_T^*(M_1;\Z) \rightarrow H_T^*(M_2;\Z)$ that maps the 
equivariant Chern classes of $M_1$ to the ones of $M_2$.  Moreover, by \cite[Theorem 3.1 (a), (c)]{Goertsches}, the 
graph isomorphism induces a ring isomorphism $H^*(M_1;\Z) \rightarrow H^*(M_2;\Z)$ that maps the  Chern classes of 
$M_1$ to the ones of $M_2$. This proves our Theorem \ref{ManiThm: coprime weights GKM graph determines eq. 
cohomology}.\\
Now assume, in addition, that the dimension of $M_1$ and $M_2$ is six. 
By \cite[Theorem 3.1 (b)]{Goertsches} the ring isomorphism $H^*(M_1;\Z) \rightarrow H^*(M_2;\Z)$ is induced by a 
(non-equivariant) diffeomorphism $M_2\rightarrow M_1$.
This proves our Theorem \ref{ManiThm: dim six gkm graph determines diffeomorphisms type}.
Indeed, since by  Lemma \ref{Lemma: GKM forces  dim(T) geq 2}, the complexity of  a  six-dimensional 
Hamiltonian GKM space is one or zero.
\end{proof}
\end{subsection}
	
\end{section}

\begin{section}{Positive Hamiltonian GKM Spaces}\label{sec:pos}

In this section, we introduce the notion of being \textbf{positive} for Hamiltonian GKM spaces. In 
particular, we prove some properties of six-dimensional positive Hamiltonian GKM spaces. 
First, we define the first Chern Class map of a  Hamiltonian GKM space.

\begin{definition}\label{Definition: First Chern Class Map }
Let $\tham$ be a Hamiltonian GKM space, $(\Gamma_{GKM},\eta)$ its GKM graph, and  $c_1(M)$ 
the first Chern class of $(M,\omega)$. The \textbf{first Chern class map} of $\tham$ is the map $\mathcal{C}_1: 
E_{GKM} 
\rightarrow 
\Z$ given by 
\begin{align*}
\mathcal{C}_1(e)= \int_{S_{(p,q)}} i^*_{S_{(p,q)}} \left( c_1(M)\right) 
\end{align*}
for each edge $e=(p,q) \in E_{GKM}$, where $S_{(p,q)}$ is the unique $T$-invariant two-sphere fixed by a
codimensional one subtorus of $T$ that contains $p$ and $q$, and $i^*_{S_{(p,q)}}$ is the map induced on 
$H^2(\,\cdot\,; 
\Z)$ by the inclusion 
$i: S_{(p,q)}\hookrightarrow M$.
\end{definition}

\begin{definition}\label{Def: positive Hamiltonian GKM space}
A Hamiltonian GKM space is called \textbf{positive} if its first Chern class map  $\mathcal{C}_1: 
E_{GKM} \rightarrow \Z$ is positive, i.e., $\mathcal{C}_1(e)>0$ for all $e\in E_{GKM}$.
\end{definition}

Due to the ABBV localization formula, the first Chern Class map can be computed from the GKM graph.
This is the content of the following lemma. In particular, the GKM graph contains the information of whenever  the 
Hamiltonian GKM space is positive.

\begin{lemma}\label{Lemma: Compute first Chern Class map}
Consider a Hamiltonian GKM space $\tham$ with GKM graph $\GKM$.  For each $e=(p,q)\in E_{GKM}$ the following holds,
\begin{align}\label{EQ1:Lemma: Compute first Chern Class map}
\mathcal{C}_1(e)=\frac{\sum_{e'\in E_{GKM}^{p,i}} \eta (e')-\sum_{e'\in E_{GKM}^{q,i}} \eta (e')}{\eta(e)}.
\end{align}
\end{lemma}
\begin{proof}
Let $c_1^T(M)\in H^2_T(M,\Z)$ be the equivariant extension of the first Chern class $c_1(M)$ of $(M,\omega)$. Since 
$S_{(p,q)}$ is a $T$-invariant two-sphere, we have
\begin{align}\label{EQ2:Lemma: Compute first Chern Class map}
	\mathcal{C}_1(e)= \int_{S_{(p,q)}} i^*_{S_{(p,q)}} \left( c_1(M)\right) = \int_{S_{(p,q)}} i^*_{S_{(p,q)}} \left( 
	c_1^T(M)\right).	
\end{align}
	The torus $T$ acts on $S_{(p,q)}$ with isolated fixed points $p$ and $q$. The weight of the
	$T$-representation 
	on $T_pS_{(p,q)}$ resp. $T_qS_{(p,q)}$  is $\eta(e)$ resp. $-\eta(e)$. Hence, the equivariant Euler class of the 
	normal 
	bundle of $p$ resp. $q$ is   $\eta(e)$ resp. $-\eta(e)$. Moreover, the restriction $i_{\{p\}}^*(c_1^T(M))$ is 
	equal 
	to the sum of the weights of the $T$-representation on $T_pM$. Due to Remark \ref{Rem: Properties GKM graphs} 
	$(i)$, the 
	latter 
	is 
	equal to $\sum_{e'\in E_{GKM}^{p,i}} \eta (e')$. For the same reason $i_{\{q\}}^*(c_1^T(M))=\sum_{e'\in 
	E_{GKM}^{q,i}} 
	\eta (e')$. Therefore, the ABBV localization formula (Theorem \ref{Thm:ABBV}) implies that the right-hand sides of 
	equations \eqref{EQ1:Lemma: Compute first Chern Class map} and \eqref{EQ2:Lemma: Compute first Chern Class map} 
	coincide.  
\end{proof} 

\begin{remark}
Since $\mathcal{C}_1$ is completely determined by the GKM graph  $\GKM$, we call $\mathcal{C}_1$  
also the first Chern class map of  $\GKM$. Moreover, we say that $\GKM$ is \textbf{positive} if 
$\mathcal{C}_1(e)>0$ for all  $e\in E_{GKM}$. 
\end{remark}

\begin{remark}\label{Rem: Compute first Chern class map}
Let $e=(p,q)\in E_{GKM}$ be an edge. 
By Lemma \ref{Lemma: Existence Connection} there exists a compatible connection 
$$\nabla_{e}: E_{GKM}^{p,i}\longrightarrow E_{GKM}^{q,i}$$ 
along $e$. 
Let $2n$ be the dimension of $M$ and let 
$p_1,\dots, p_{n-1}$ resp. $q_1,\dots, q_{n-1}$ 
be the unique $n-1$ fixed points in $M^T\setminus\{p,q\}$ such that 
$$(p,p_i)\in E_{GKM}^{p,i}\quad \text{and} \quad (q,q_i)\in E_{GKM}^{q,i},$$
ordered so that
$$\nabla_e(p,p_i)=(q,q_i)$$ 
for all $i=1,...,n-1.$ So there exist integers $a_1,...,a_{n-1}$ such that 
\begin{align*}
\eta(p,p_i)-\eta(q,q_i)=a_i\cdot \eta(p,q) \quad \text{for all } i=1,...,n-1.
\end{align*}
By the formula in Lemma \ref{Lemma: Compute first Chern Class map} we have
\begin{align*}
\mathcal{C}_1(e)&=\frac{\sum_{e'\in E_{GKM}^{p,i}} \eta (e')-\sum_{e'\in E_{GKM}^{q,i}} \eta (e')}{\eta(e)}\\
&=\frac{\eta(p,q)-\eta(q,p)}{\eta(p,q)}+\sum_{i=1}^{n-1}\frac{\eta(p,p_i)-\eta(q,q_i)}{\eta(p,q)}\\
&=2+a_1+ \dots a_{n-1}.
\end{align*}

\end{remark}

\begin{lemma}\label{Lemma: symplectic fano c1 positive}
Let $\tham$ be a Hamiltonian GKM space, where $(M,\omega)$ is a monotone symplectic manifold. Then this 
space is positive.
\end{lemma}
\begin{proof}
Since $\tham$ is GKM the fixed point set $M^T$ is finite. Since $(M,\omega)$ is monotone
 and the torus action is effective and Hamiltonian, by \cite[Proposition 
5.2]{1224},  $(M,\omega)$ is positive monotone, i.e., there exists an $r\in \R_{>0}$ such that $c_1(M)=r\cdot [\omega]$. Let $e=(p,q)\in E_{GKM}$ be an edge. We 
have that
\begin{align}\label{EQ1:Lemma: symplectic fano c1 positive}
\mathcal{C}_1(e)= \int_{S_{(p,q)}} i^*_{S_{(p,q)}} \left( c_1(M)\right) = r\cdot \int_{S_{(p,q)}} i^*_{S_{(p,q)}}  
\omega.	
\end{align}
Since $S_{(p,q)}$ is an embedded symplectic submanifold of $(M,\omega)$, the right hand side of Equation \eqref{EQ1:Lemma: 
symplectic fano c1 positive} is positive.   
\end{proof}

\begin{subsection}{An Upper Bound for the Number of Fixed Points in Dimension Six}

In this subsection we prove that the number of fixed points of a positive six-dimensional Hamiltonian GKM space is at 
most $16$ and we give an obstruction for the first Chern class map. This is the content of Corollary \ref{Cor: 
Ingerdients for proof}. This corollary is a direct consequence of  results by Godinho-Sabatini \cite{GodinhoSabatini} 
and Godinho-von Heymann-Sabatini \cite{1224}, which we  recall here.  
 
\begin{lemma}\label{Lemma: GodinhoSabatini c1cn-1 Betti numbers}\cite[Corollary 3.1]{GodinhoSabatini}
Let $\tham$ be a compact Hamiltonian $T$-space of dimension $2n$ with only isolated fixed points. For 
$p=0,...,n$ let $b_{2p}$ be the  $2p$-th Betti number of $M$. Then 
\begin{align*}
\int_M c_1(M)c_{n-1}(M)= \sum_{p=0}^{n} b_{2p} \left[ 6p(p-1)+ \frac{5n-3n^2}{2}\right], 
\end{align*}
where $c_1(M)$ resp. $c_{n-1}(M)$ is the first resp. $(n-1)$-th Chern class of $(M,\omega)$. 
\end{lemma}	

\begin{remark}\label{Rem: Lemma: GodinhoSabatini c1cn-1 Betti numbers}
	That Lemma \ref{Lemma: GodinhoSabatini c1cn-1 Betti numbers} holds in case  the torus $T$ has 
	dimension one, i.e.,  $T\cong S^1$, is the content of \cite[Corollary 3.1]{GodinhoSabatini}. Since each 
	Hamiltonian 
	$T$-space admits a subcircle $S^1\subset T$ such that $M^{S^1}=M^T$ holds, it follows that Lemma \ref{Lemma: 
		GodinhoSabatini c1cn-1 Betti numbers}  is also true whenever the dimension of the torus is greater than one.
\end{remark}

Before we state the next lemma, we need to introduce a notation.

\begin{definition}\label{Def: orientation edge set}
Let $\tham$ be a Hamiltonian GKM space and let $\GKM$ be its GKM graph. An \textbf{orientation} $\sigma$ for the edge 
set $E_{GKM}$ is a subset $E^\sigma_{GKM}$ of $E_{GKM}$ such that for each $(p,q)\in E_{GKM}$ exactly one of the 
following two conditions 
is true.
\begin{itemize}
\item $(p,q)\in E^\sigma_{GKM}$ and  $(q,p)\notin E^\sigma_{GKM}$
\item $(q,p)\in E^\sigma_{GKM}$ and  $(p,q)\notin E^\sigma_{GKM}$
\end{itemize} 
\end{definition}

\begin{remark}\label{Remark: cardinality edge set}
Let $2n$ be the dimension of $M$. The graph $\Gamma_{GKM}=(M^T, 
E_{GKM})$ is $n$-valent, i.e., for each fixed point $p$ there exist exactly $n$ edges whose initial point is $p$.
Therefore, the cardinalities of $M^T$ and $E_{GKM}$ are related by
\begin{align*}
\left| E_{GKM} \right| = n \cdot \left| M^T \right|. 
\end{align*}
Whenever $\sigma$ is an orientation of the edge set, then the cardinality of $E_{GKM}^\sigma$ is equal to half of 
the one of $E_{GKM}$. Hence,
\begin{align*}
\left| E_{GKM}^\sigma \right| = \frac{n}{2} \cdot \left| M^T \right|. 
\end{align*}

\end{remark}

\begin{lemma}\label{Lemma: Poincare dual}(c.f.  \cite[Lemma 4.13]{1224})
 Choose an orientation $\sigma$ of the edge set $E_{GKM}$. Then 
\begin{align*}
\sum_{e\in E^\sigma_{GKM}} \mathcal{C}_1(e)=\int_M c_1(M)c_{n-1}(M).
\end{align*}
\end{lemma}

Lemma \ref{Lemma: Poincare dual}  follows directly from \cite[Lemma 4.13]{1224}. Since our setting is slightly  
different from the one in \cite{1224}, we give  the proof of Lemma \ref{Lemma: Poincare dual}.

\begin{proof}[Proof of Lemma \ref{Lemma: Poincare dual}] 
The proof is a simple application of the ABBV Localization Formula (Theorem \ref{Thm:ABBV}).
For each fixed point $p\in M^T$, let $\alpha_{p,1},\dots,\alpha_{p,n}$ be the weights of the $T$-representation on 
$T_pM$. Note that the set of these weights is equal to the set $\{\eta(e')\}_{e'\in E_{GKM}^{p,i}}$. 
Let $e=(p,q)\in E_{GKM}$ be an edge. By Lemma \ref{Lemma: Compute first Chern Class map} we have
	
\begin{align*}
\mathcal{C}_1(e)=\dfrac{\alpha_{p,1}+\dots +\alpha_{p,n}}{\eta(e)}-\dfrac{\alpha_{q,n}+\dots +\alpha_{q,n}}{\eta(e)}.
\end{align*}
	
For each fixed point $p$ and each weight $\alpha_{p,j}$ there exists exactly one edge $e\in E_{GKM}$ such that 
$i(e)=p$ and $\eta(e)=\alpha_{p,j}$. Note that $\eta(\bar{e})=-\alpha_{p,j}$ and either 
$e\in E^\sigma_{GKM}$ or $\bar{e}\in E^\sigma_{GKM}$, where $\bar{e}$ is the unique edge with $(i(\bar{e}), t(\bar{e}))=(t(e), 
i(e))$. 
We conclude that 
\begin{align*}
	\sum_{e\in E^\sigma_{GKM}} \mathcal{C}_1(e)=
	\sum_{p\in M^T}\left[ \sum_{j=1}^{n}\dfrac{\alpha_{p,1}+...+\alpha_{p,n}}{\alpha_{p,j}} \right]. 
	\end{align*}
	For the equivariant extensions $c_1^T(M)$ and $c_{n-1}^T(M)$  of $c_1(M)$ and $c_{n-1}(M)$ we have
	\begin{align*}
	\int_M c_1(M)c_{n-1}(M)=\int_M c_1^T(M)c_{n-1}^T(M).
	\end{align*}
	For each fixed point $p\in M^T$ we have
	\begin{align*}
	i^*_{p}\left( c_1^T(M)c_{n-1}^T(M)\right) =\left( \sum_{j=1}^{n}\alpha_{p,j}\right) \cdot \left( \sum_{k=1}^{n} 
	\prod_{l=1,l\neq k}^{n} \alpha_{p,l}\right) 
	\end{align*}
	and the equivariant Euler class of $T_pM$ is $\prod_{j=1}^{n} \alpha_{p,j}$. Therefore, the ABBV Localization 
	Formula
	gives
	\begin{align*}
	\int_M c_1^T(M)c_{n-1}^T(M)= \sum_{p\in M^T}\left[ \frac{i^*_{p}\left( c_1^T(M)c_{n-1}^T(M)\right)}{\prod_{j=1}^{n} 
		\alpha_{p,j}}\right]= \sum_{p\in M^T}\left[ \sum_{j=1}^{n}\dfrac{\alpha_{p,1}+...+\alpha_{p,n}}{\alpha_{p,j}} 
		\right].
	\end{align*}
The lemma follows.
\end{proof} 

The following corollary  is a direct consequence of  Lemma \ref{Lemma: GodinhoSabatini c1cn-1 Betti numbers} and
Lemma  \ref{Lemma: Poincare dual}.

\begin{corollary}\label{Cor: Ingerdients for proof}
Let $\tham$ be a Hamiltonian GKM space of dimension $2n$ and let $\GKM$ be its GKM graph. 
Choose an orientation $\sigma$ of the edge set $E_{GKM}$. Then the following hold.
\begin{itemize}
\item[(i)] 
\begin{align*}
\sum_{e\in E^\sigma_{GKM}} \mathcal{C}_1(e)=\sum_{p=0}^{n} b_{2p} \left[ 6p(p-1)+ \frac{5n-3n^2}{2}\right],
\end{align*}
where $b_0,\dots,b_{2n}$ are the even Betti numbers of $M$.
\item[(ii)] If the dimension of $M$ is equal to six, then
\begin{align*}
\sum_{e\in E^\sigma_{GKM}} \mathcal{C}_1(e)=24.
\end{align*} 
\item[(iii)] If $\tham$ is positive and  the dimension of $M$ is equal to six, then the number of 
fixed points is at most  $16$. 
\end{itemize}
\end{corollary}

\begin{proof}
\begin{itemize}
\item[(i)] 
This statement follows directly from Lemma \ref{Lemma: GodinhoSabatini c1cn-1 Betti numbers} and 
Lemma \ref{Lemma: Poincare dual}.
\item[(ii)] 
Assume that the dimension of $M$ is equal to six, i.e., $n$ is equal to $3$. Since $\tham$ is a Hamiltonian 
GKM space, the underlying symplectic manifold  is compact and connected. Therefore, $b_0=b_6=1$ holds. Moreover, by the 
Poincar\'e Duality Theorem we have that $b_2=b_4$ holds. Hence, the statement follows from $(i)$.
\item [(iii)] 
Assume that the dimension of $M$ is equal to six and that the space is positive, i.e., 
$\mathcal{C}_1(e)$  is a positive integer for all $e\in E_{GKM}$.
Then it follows from $(ii)$ that
\begin{align*}
	\left| E_{GKM}^\sigma\right| \leq 24.
\end{align*} 
By Remark \ref{Remark: cardinality edge set}  we have $\left| E^\sigma_{GKM}\right| =\frac{3}{2}\left| M^T\right| $. 
Hence, the number of fixed points is at most $16$.
\end{itemize}
\end{proof}

\end{subsection} 

\begin{subsection}{Special Kirwan Classes}

In this subsection we consider Hamiltonian GKM spaces that are weak index increasing with respect to 
a generic vector $\xi$ and we prove the existence of special Kirwan classes for such a space. We also show that  
six-dimensional positive Hamiltonian GKM spaces are weak index increasing with respect to any generic vector.

\begin{definition}\label{Definition: weak indexing increasing}
Let $\tham$ be a Hamiltonian GKM space and let $\xi\in \mathfrak{t}$ be a generic vector. 
Its GKM graph $\GKM$ is called \textbf{index increasing} resp. \textbf{weak index increasing} with respect to 
$\xi$ if
\begin{align*}
\lambda(p)< \lambda(q) \quad \text{resp.} \quad \lambda(p)\leq \lambda(q)
\end{align*} 
holds for any edge $(p,q)\in E_{GKM}$ with $\phi^\xi(p)< \phi(q)^\xi$.
\end{definition}

\begin{lemma}\label{Lemma: c1positive implies weak index increasing}
Let $\tham$ be a six-dimensional positive Hamiltonian GKM space. Then its GKM graph $\GKM$
is weak index increasing with respect to each generic vector $\xi\in \mathfrak{t}$.	
\end{lemma}
\begin{proof}
Let $\xi\in \mathfrak{t}$ be a generic vector and assume that $\GKM$ is not weak index increasing with respect to 
$\xi$. We show that  $\tham$ is not positive. Indeed by this assumption there exists an edge 
$e=(p,q)\in E_{GKM}$ such that $\lambda(p)>\lambda(q)$ and $\phi^\xi(p)< \phi^\xi(q)$. By Remark 
\ref{Rem: images 
weight relation} there exists an $\tilde{r}>0$ such that 
$\eta(p,q)=\tilde{r}\cdot\left( \phi(q)-\phi(p)\right) $. So we have 

\begin{align}\label{EQ1:Lemma: c1positive implies weak index increasing}
\left\langle \eta(p,q)\,,\, \xi  \right\rangle= \tilde{r}\cdot \left(\left\langle \phi(q), \xi \right\rangle - 
\left\langle 
\phi(p), \xi \right\rangle\right)= \tilde{r}\cdot \left( \phi^\xi(q)-\phi^\xi(p)\right)    >0.
\end{align}
Since $M$ is six-dimensional, we have that 
$$\lambda(p),\lambda(q)\in \{0,1,2,3\}.$$
Note that $\lambda(p)=3$ can not happen, because in this case $\phi^\xi$ attains its maximum at $p$,
which contradicts $\phi^\xi(p)< \phi^\xi(q)$. For the same reason $\lambda(q)=0$ can not happen, because in this case 
$\phi^\xi$ attains its minimum at $q$. Hence, we have  $\lambda(p)=2$ and $\lambda(q)=1$. 
Let 
$$\nabla_e: E_{GKM}^{p,i} \rightarrow E_{GKM}^{q,i}$$
be a compatible connection along the edge $e=(p,q)$ and let $p_1,p_2$ and 
$q_1,q_2$ be the fixed points in $M^T\setminus\{p,q\}$ with $(p,p_i),(q,q_i)\in E_{GKM}$ for $i=1,2$, 
ordered so that 
\begin{align*}
\nabla_e(p,p_i)=(q, q_i)
\end{align*}
for $i=1,2$. So there exist integers $a_1$ and $a_2$ such that 
\begin{align}\label{EQ2:Lemma: c1positive implies weak index increasing}
\eta(p,p_i)-\eta(q,q_i)=a_i \cdot  \eta(p,q)
\end{align}
for $i=1,2$. Recall that $\lambda(p)$ resp. $\lambda(q)$ is the number of weights of the 
$T$-representation on $T_pM$ resp. $T_qM$ such that $\left\langle \cdot\, , \xi \right\rangle $ is negative. Since 
$\lambda(p)=2$, $\lambda(q)=1$ and $\left\langle \eta(p,q),\xi\right\rangle >0$, we have 
\begin{align}\label{EQ3:Lemma: c1positive implies weak index increasing}
\left\langle \eta(p,p_i)\,,\, \xi  \right\rangle  <0 \quad \text{and }\quad\left\langle \eta(q,q_i)\, ,\, \xi  
\right\rangle  >0
\end{align}
for $i=1,2$. By combining \eqref{EQ1:Lemma: c1positive implies weak index increasing},
\eqref{EQ2:Lemma: c1positive implies weak index increasing} and 
\eqref{EQ3:Lemma: c1positive implies weak index increasing}, we conclude that $a_1$ and $a_2$ are negative integers. 
So by Lemma \ref{Lemma: Compute first Chern Class map} and its Remark \ref{Rem: Compute first Chern class map} we have 
that
\begin{align*}
	\mathcal{C}_1(e) = 2 + a_1 + a_2 \leq 0.
\end{align*} 
Hence, the space is indeed not positive.
\end{proof}

\begin{definition}\label{Def: paths}
Given a Hamiltonian GKM space $\tham$, let $\GKM$ be its GKM graph 
and let $\xi \in \mathfrak{t}$ be a generic vector. An \textbf{ascending path} from 
a fixed point $p$ to another fixed point $q$ is a 
$(k+1)$-tuple $\nu=(p_0,...,p_k)$ of points in $M^T$ such that $p_0=p$, $p_k=q$ and 

$$(p_{i-1},p_i)\in E_{GKM} \,\,\text{and}\,\, \phi^\xi(p_{i-1})< \phi^\xi (p_i)$$
for $i=1,\dots,k$. Moreover, for each fixed point $p\in 
M^T$, the \textbf{stable set of $p$}, denoted by $\Xi_p$ is the set of points $q\in M^T$ such that there 
exists an ascending path from $p$ to $q$, including $p$ itself.
\end{definition}

In the following lemma we sum up some properties about ascending paths and the stable sets. These properties follow 
directly from the definitions. 

\begin{lemma}\label{Lemma: simple properties about paths}
Let $p,q\in M^T$ be two different fixed 
points. Then the following hold.
\begin{itemize}
\item[(i)] If $q\in \Xi_p$ then  $\phi^\xi(p) < \phi^\xi(q)$.
\item[(ii)] If $\lambda(q)=0$, then  $q\notin \Xi_p$. If $\lambda(q)>0$, then  $q\notin \Xi_p$ 
if and only if for all $r\in M^T$ with $(r,q)\in E_{GKM}$ and $\phi^\xi(r)<\phi^\xi(q)$ we have $r \notin \Xi_p$.
\item[(iii)] If the GKM graph of $\tham$ is  index increasing resp. weak index increasing  with respect to $\xi$ and 
 $q\in \Xi_p$ , then  
$\lambda(p) <  \lambda(q)$ resp.  $\lambda(p) \leq  \lambda(q)$.
\end{itemize}
\end{lemma}

\begin{proposition}\label{Proposition: special Kirwan class lambda=1}
Let $\tham$ be a Hamiltonian GKM space and let $\xi \in \mathfrak{t}$ be a generic vector such that 
the GKM graph $\GKM$ is  weak index increasing with respect to $\xi$.
Let $p\in M^T$ with $\lambda(p)=1$. Then there exists a unique Kirwan class $\gamma_p\in H_T^{2}(M;\Z)$ at $p$ such 
that for $q\in M^T$ the following hold. 
\begin{itemize}
 \item[(i)] $\gamma_p(q)=\Lambda_p^-$  if $\lambda(q)=1$ and $q\in \Xi_p$. 
\item[(ii)] $\gamma_p(q)=0$ if $q\notin \Xi_p$.  
\end{itemize}
\end{proposition}
\begin{proof}
Note that $p\in \Xi_p$. By Lemma \ref{Lemma: simple properties about paths} $(i)$  for all $q\in M^T\setminus 
\{p\}$ 
with 
$\phi^\xi(q)\leq \phi^\xi(p)$ we have $q 
\notin \Xi_p$. Hence, 
a class $\gamma_{p}$ with   properties $(i)$ and $(ii)$ is indeed a Kirwan class at $p$. Moreover,  
that such a class is unique is easy to see. Namely, let $\gamma_{p}$ and $ \widetilde{\gamma}_p$ be two such Kirwan 
classes at $p$. For all $q\in M^T$  with $\lambda(q)\leq 1$, $\gamma_{p}(q) =\widetilde{\gamma}_p(q)$ holds and the 
degree of each of these classes is two. Hence, by Lemma \ref{Lemma: Properties of classes} (ii), we have $\gamma_{p} 
=\widetilde{\gamma}_p$.\newline
Now we prove the existence of such a class. Consider the set 
\begin{align*}
\mathcal{L}&:= \left\lbrace q\in M^T \,\,\vert\,\,  
\phi^\xi(p)\leq\phi^\xi(q)\right\rbrace. 
\end{align*} 
Note that by Lemma \ref{Lemma: simple properties about paths} $(i)$ we have $\Xi_p\subset \mathcal{L}$. Let 
\begin{align*}
p=q_1,q_2,...,q_{\left| \mathcal{L}\right|-1 },q_{\left| \mathcal{L}\right| }
\end{align*}
be the points of $\mathcal{L}$ ordered so that
\begin{align*}
\phi^\xi (q_1)\leq\phi^\xi(q_2)\leq...\leq\phi^\xi(q_{\left| \mathcal{L}\right|-1 }) \leq\phi^\xi(q_{\left| 
\mathcal{L}\right| }).
\end{align*}
By induction over $i=1,...,\left| \mathcal{L}\right|$ we show that for all such $i$
there exists a class $\beta_i\in H_T^2(M;\Z)$ that satisfies the following properties (a.$i$) and (b).
\begin{itemize}
\item[(a.$i$)] For all $j=1,...,i$
\begin{align*}
&\beta_i(q_j)=\Lambda_q^- \quad \text{if}\,\, \lambda(q_j)=1 \,\, \text{and}\,\, q_j\in \Xi_p\\
&\beta_i(q_j)=0 \quad\,\,\,\,\, \text{if} \,\, q_j\notin \Xi_p.
\end{align*}
\item[(b)] $\beta_i(r)=0$ if $r\in M^T$ and $\phi^\xi(r)<\phi^\xi(p)$.
\end{itemize}
The induction base is true. Any Kirwan class at $p$ satisfies the 
properties (a.$1$) and (b). Now assume that for a fixed $i=1,...,\left| \mathcal{L}\right| -1$ there exists  a class 
$\beta_i\in H_T^2(M;\Z)$ that satisfies (a.$i$) and (b). Consider the fixed point $q_{i+1}$. We have four cases.
\begin{itemize}
\item[] \textsc{1.case :} $\lambda(q_{i+1})\geq2$ and $q_{i+1}\in \Xi_p$
\item[]\textsc{2.case :} $\lambda(q_{i+1})\geq2$ and $q_{i+1}\notin \Xi_p$
\item[]\textsc{3.case :} $\lambda(q_{i+1})=1$ and $q_{i+1}\in \Xi_p$
\item[]\textsc{4.case :} $\lambda(q_{i+1})=1$ and $q_{i+1}\notin \Xi_p$
\end{itemize}
\textsc{1.case :} Since (a.${i+1}$) and (b) do not force any obstruction for $\beta_{i+1}(q_{i+1})$, we can choose
$\beta_{i+1}=\beta_i$.\newline
\textsc{2.case :} We show that $\beta_i(q_{i+1})=0$. Hence, we can choose $\beta_{i+1}=\beta_i$.  Since 
$\lambda(q_{i+1})\geq2$
there exist two different fixed points $r_1,r_2\in M^T\setminus\{q_{i+1}\}$ such that 
$$(r_1, q_{i+1}),(r_2, q_{i+1})\in E_{GKM}\quad \text{and}\quad \phi^\xi(r_1),\phi^\xi(r_2)<\phi^\xi(q_{i+1}).$$
If $\phi^\xi(r_1)<\phi^\xi(p)$, then (b) implies $\beta_i(r_1)=0$. If $\phi^\xi(p)\leq\phi^\xi(r_1)$, then 
$r_1\in \mathcal{L}$ and $\phi^\xi(r_1)<\phi^\xi(q_{i+1})$. This implies that $r_1=q_j$ for 
some $j=1,...,i$. Moreover, since $q_{i+1}\notin \Xi_{p}$, Lemma \ref{Lemma: simple properties about paths} $(ii)$ 
implies that $r_1=q_j \notin \Xi_p$. Therefore, (a.$i$)
implies $\beta_i(r_i)=0$.  Hence, in both cases $\phi^\xi(r_1)<\phi^\xi(p)$ or $\phi^\xi(p)\leq\phi^\xi(r_1)$, 
we have $\beta_i(r_1)=0$. For the same reason we have also $\beta_i(r_2)=0$. Since $(r_1, q_{i+1}),(r_2, q_{i+1})\in 
E_{GKM}$, by Lemma \ref{Lemma: divisble condition} there exist integers $A_1$ and $A_2$ such that 
\begin{align*}
&\beta_i(q_{i+1})=\beta_i(q_{i+1})-\beta_i(r_1)=A_1\cdot \eta(q_{i+1},r_1)\quad \text{and}\\
&\beta_i(q_{i+1})=\beta_i(q_{i+1})-\beta_i(r_2)=A_2\cdot \eta(q_{i+1},r_2).
\end{align*}
Since $\eta(q_{i+1},r_1)$ and $\eta(q_{i+1},r_2)$ are linearly independent, we conclude that $A_1=A_2=0$ 
and $\beta_i(q_{i+1})=0$.\newline
\textsc{3.case :} Since $\lambda(q_{i+1})=1$ there exists a unique fixed point $r\in M^T$ such that 
$$(r, q_{i+1})\in E_{GKM}\quad \text{and}\quad \phi^\xi(r)<\phi^\xi(q_{i+1}).$$
Since $q_{i+1}\in \Xi_p$, by Lemma \ref{Lemma: simple properties about paths} $(ii)$ we have $r\in \Xi_p$. Moreover, 
since $\Xi_p \subset \mathcal{L}$
we have $r=q_j$ for some $j=1,...,i$. Note that also  $q_{i+1}\in \Xi_r$. Since the GKM graph of $\tham$ is weak index 
increasing
with respect to $\xi$, by Lemma \ref{Lemma: simple properties about paths} $(iii)$ we have
$$1=\lambda(p)\leq \lambda(r)\leq \lambda(q_{i+1})=1.$$
So we have that $\lambda(r)=1$. Since (a.$i$) holds for the class $\beta_i$, we conclude that $\beta_i(r)=\Lambda_p^-$.
Since $(r,q_{i+1})\in E_{GKM}$, by Lemma \ref{Lemma: divisble condition} there exists an integer $A$ such that
$$\beta_i(q_{i+1})-\Lambda_p^-= \beta_i(q_{i+1})-\beta_i(r)=A\cdot \eta(q_{i+1},r).$$
Note that since $\lambda(q_{i+1})=1$  holds, we have $\eta(q_{i+1},r)=\Lambda_{q_{i+1}}^-$.
So we have 
$$\beta_i(q_{i+1})=\Lambda_p^-+ A\cdot \Lambda_{q_{i+1}}^-.$$
Let $\alpha_{q_{i+1}}\in H_T^2(M;\Z)$ be a Kirwan class at $q_{i+1}$. So the class
$$\beta_{i+1}=\beta_i-A\cdot \alpha_{q_{i+1}}$$
satisfies the properties (a.$i+1$) and (b).\newline
\textsc{4.case :} Since $\lambda(q_{i+1})=1$, there exists a unique fixed point $r\in M^T$ such that  
$$(r, q_{i+1})\in E_{GKM}\quad \text{and}\quad \phi^\xi(r)<\phi^\xi(q_{i+1}).$$
Since $q_{i+1}\notin \Xi_p$, by Lemma \ref{Lemma: simple properties about paths} $(ii)$ we have $r\notin \Xi_p$. So if 
$r\in \mathcal{L}$ then we have $r=q_j$
for some $j=1,...,i$. So $(a.$i$)$ and $r\notin \Xi_p$ implies that $\beta_i(r)=0$. If $r\notin \mathcal{L}$
then (b) also implies $\beta_i(r)=0$. Since $(r,q_{i+1})\in E_{GKM}$, there exists an integer $A$ such that
$$\beta_i(q_{i+1})= \beta_i(q_{i+1})-\beta_i(r)=A\cdot \eta(q_{i+1,r}).$$
Note that since $\lambda(q_{i+1})=1$  holds, we have $\eta(q_{i+1},r)=\Lambda_{q_{i+1}}^-$.
So we have 
$$\beta_i(q_{i+1})= A\cdot \Lambda_{q_{i+1}}^-.$$
Let $\alpha_{q_{i+1}}\in H_T^2(M;\Z)$ be a Kirwan class at $q_{i+1}$. So the class
$$\beta_{i+1}=\beta_i-A\cdot \alpha_{q_{i+1}}$$
satisfies the properties (a.$i+1$) and (b).\newline
We conclude that there exists a class $\beta_{\left| \mathcal{L}\right| }\in H^2_T(M;\Z)$ that
satisfies the properties (a.$\left| \mathcal{L}\right|$) and (b). In fact, such  a class $\beta_{\left| 
\mathcal{L}\right| }$  satisfies the desired properties $(i)$ and $(ii)$. 
\end{proof}

\end{subsection}

\end{section}

\begin{section}{Constructing (Abstract) GKM Graphs}

In this section, we prove further statements
that are needed for the classification of
GKM graphs 
of 
positive Hamiltonian GKM spaces of dimension six. Let $\GKM$ be the GKM graph of such a space. 
Then the graph $\Gamma_{GKM}$ is a simple, connected and $3$-valent (see Definition  \ref{Def:simple n-valent graph}). 
By Corollary \ref{Cor: Ingerdients for proof}, the graph has at most $16$ vertices and we have only finitely many 
possibilities for the first Chern class map $\mathcal{C}_1: E_{GKM}\rightarrow \Z$. Since simple and connected 
$3$-valent graphs with at most $16$ vertices are classified \cite{DataBaseCubicGraphs}, the classification problem is 
strongly related to 
the following question.  

\begin{question}\label{Question: Does there exists an (abstract) GKM graph?}
Let $T$ be a torus with dual lattice $\ell_T^*$. Let $\Gamma=(V,E)$ be a simple, connected and $n$-valent 
graph and let $\mathcal{D}\colon E \rightarrow \Z$ be a map.
\begin{itemize}
\item Does the pair $(\Gamma, \mathcal{D})$ \textbf{support} a Hamiltonian GKM graph, i.e., does there exist a map 
$\eta: E \rightarrow \ell_T^* $ such that the pair $(\Gamma, \eta)$ is the GKM graph of a Hamiltonian GKM space
$\tham$ of dimension $2n$ and for all edges $e\in E$ 
\begin{align*}
\mathcal{C}_1(e)= \mathcal{D}(e), 
\end{align*}
where $\mathcal{C}_1\colon E \rightarrow \Z$ is the first Chern class map?
\item And if so, how many  such maps $\eta: E \rightarrow \ell_T^* $  exist (up to 
isomorphisms and projections of GKM graphs), and can we compute such maps?
\end{itemize}
\end{question}

In the first part of this section, we introduce abstract GKM graphs and formulate Question \ref{Question: 
Does there exists an (abstract) GKM graph?}  for abstract GKM graphs. In the second part of this 
section, we show that  Question \ref{Question: Does there exists an (abstract) GKM graph?} for abstract GKM graphs 
can be solved by methods of linear algebra. In the last part of this section, we consider 
positive (abstract) GKM graphs that are coming from six-dimensional Hamiltonian GKM spaces.

\begin{subsection}{Abstract GKM Graphs}

Let  $\Gamma=(V,E)$ be a graph with directed edges. This means that there exist an \textbf{initial map} $i:E\rightarrow 
V$ and a 
\textbf{terminal map} $t:E\rightarrow V$. We associate to each vertex $v\in V$ the following two sets
\begin{align*}
E_v^i:= \{e\in E \mid i(e)=v\} \quad \quad \text{and} \quad \quad E_v^t:= \{e\in E \mid t(e)=v\}.
\end{align*}

\begin{definition}\label{Def:simple n-valent graph}	
A graph $\Gamma=(V,E)$ with directed edges is called \textbf{simple} if the following three conditions are true.
\begin{itemize}
\item The graph has no loops, i.e., $i(e)\neq t(e)$ for all $e\in E$. 
\item The graph has no double edges, i.e., if  $i(e)=i(e')$ and $t(e)=t(e')$ for $e,e'\in E$ then $e=e'$.
\item For each edge $e\in E$ there exists a unique edge $\bar{e}$  such that $i(\bar{e})=t(e)$ and $t(\bar{e})=i(e)$.
\end{itemize}
Moreover, such a graph $\Gamma=(V,E)$ is called \textbf{$n$-valent} if for each vertex $v\in V$ the cardinality 
of $E_v^i$ (or equivalently, the cardinality of $E_v^t$) is equal to $n$.
\end{definition}

Note that by the second item in Definition \ref{Def:simple n-valent graph} we can consider the edge set of a simple 
graph $\Gamma=(V,E)$ as a subset of $V\times V$. Hence, we write an edge $e\in E$ also as $(v,w)$, where $i(e)=v$ and 
$t(e)=w$. Such a graph is called \textbf{connected} if for any two different vertices $v,w \in V$, there 
exists a sequence $v_0,...,v_k$ in $V$ such that $v_0=v$, $v_k=w$ and $(v_i,v_{i+1})\in E$ for $i=0,\dots,k-1$. 
 A \textbf{connection along} an edge $e\in E$  is a bijection  
\begin{align*}
\nabla_e: E_{i(e)}^i \longrightarrow E_{t(e)}^i
\end{align*}
such that $\nabla_e(e)=\bar{e}$, where $\bar{e}$ is the edge with $(i(\bar{e}), t(\bar{e}))=(t(e),i(e))$.

\begin{definition}\label{Def: abstract GKM graph}
Let $n$ and $d$ be positive integers. An \textbf{abstract $(n,d)$-GKM graph} $(\Gamma, \operatorname{w})$ is a 
connected simple and $n$-valent graph $\Gamma=(V,E)$ together with a \textbf{weight map} 
$\operatorname{w}:E\rightarrow \Z^d\setminus\{0\}$ that is \textbf{antisymmetric}, i.e., for all $e\in E$ we have
$\operatorname{w}(\bar{e})=-\operatorname{w}(e)$,  such that the following hold.
\begin{itemize}
\item[(i)] For each vertex $v\in V$, the $\Z$-span of vectors $\operatorname{w}(e)$ for $ e\in E^i_v$  
 is equal to $\Z^d$. 
\item[(ii)] For each vertex $v\in V$, the vectors $\operatorname{w}(e)$ for $ e\in E^i_v$ are 
pairwise linearly independent in $\Z^d$ over $\Z$. 
\item[(iii)]  For each  $e\in E$, there exists a connection 
\begin{align*}
\nabla_e: E_{i(e)}^i \longrightarrow E_{t(e)}^i
\end{align*}
that is \textbf{compatible} with the weight map $\operatorname{w}$. The latter means that
 for each $e'\in E_{i(e)}^i$ 
there exists an integer $a_{e,e'}$ such that 
\begin{align*}
\operatorname{w}(e')-\operatorname{w}(\nabla_e(e'))= a_{e,e'}\cdot \w(e);
\end{align*}
the former means that $\nabla_e$ is a bijection and $\nabla_{e}(e)=\bar{e}$.
\end{itemize}
\end{definition}

\begin{remark}\label{Rem: relation n, d abstract GKM graph}
If  $n\geq2$, then for an abstract $(n,d)$-GKM graph we have $2\leq d\leq n$, as
follows from items $(i)$ and $(ii)$ of Definition \ref{Def: abstract GKM graph}.
If $n = 1$ then $d = 1$ as well.
\end{remark}

\begin{definition}\label{Def: isomorphism simple graphs}
	Let $\Gamma_1=(E_1,V_1)$ and  $\Gamma_2=(E_2,V_2)$ be two simple graphs. An \textbf{isomorphism} between $\Gamma_1$ and $\Gamma_2$ is a bijection $F:V_1 \rightarrow V_2$ such that for each two vertices $v,w\in V_1$, we have that $(v,w)\in 
	E_1$  if and only if $(F(v),F(w))\in E_2$.
\end{definition}

\begin{remark}\label{Rem: bijection edges}
An isomorphism $F:V_1 \rightarrow V_2$ between simple graphs $\Gamma_1=(E_1,V_1)$ and  $\Gamma_2=(E_2,V_2)$ induces a 
bijection  $E_1\rightarrow E_2$, $(v,w)\mapsto (F(v),F(w)).$
\end{remark}

\begin{definition}\label{Def: isomorphisms abstract GKM graphs}
	Let $(\Gamma_1=(V_1,E_1), \operatorname{w_1})$ and $(\Gamma_2=(V_2,E_2), \operatorname{w_2})$ be two abstract
	$(n,d)$- GKM graphs. An \textbf{isomorphism} $(F, \theta)$ between $(\Gamma_1,\operatorname{w_1})$ and $(\Gamma_2,\operatorname{w_2})$ is an isomorphism $F$ between the 
	simple graphs $\Gamma_1$ and $\Gamma_2$ together with a linear isomorphism $\theta:\Z^d \rightarrow \Z^d$ such that 
	for 
	all $(v,w)\in E_1$ 
	\begin{align*}
	\theta(\operatorname{w_1(v,w)})=\operatorname{w_2}(F(v),F(w)).
	\end{align*} 
\end{definition}

\begin{subsubsection}{Relations between Hamiltonian and Abstract GKM Graphs}

The GKM graphs of Hamiltonian GKM spaces give in a natural way examples of abstract GKM graphs. Let $\tham$ be a
Hamiltonian GKM space of dimension $2n$ and let $\GKM$ be its GKM graph. The graph $\Gamma_{GKM}$ is a simple and 
connected $n$-valent graph. Let $\chi: \ell_T^* \rightarrow \Z^d$ be a linear isomorphism, where $d$ is the dimension 
of the torus $T$. Then the graph $\Gamma_{GKM}$ together with the composition $\chi\circ \eta: E_{GKM}\rightarrow \Z^d$ 
is an abstract $(n,d)$-GKM graph. In particular, $\chi$ induces a map

\begin{align*}
\mathcal{L}_\chi\colon 
\left\lbrace  \begin{array}{r}
\text{GKM graphs of Hamiltonian }  \\ 
\text{ GKM $T$-spaces of dimension $2n$}  \\
\end{array}\right\rbrace 
\longrightarrow
\left\lbrace  \begin{array}{r}
\text{abstract $(n,d)$ }  \\ 
\text{ -GKM graphs}  \\
\end{array}\right\rbrace.
\end{align*}

Note that two GKM graphs of Hamiltonian GKM $T$-spaces are isomorphic 
(in the sense of Definition \ref{Def: isomophic GKM graphs}) if and only if their images under $\mathcal{L}_\chi$ 
are isomorphic abstract GKM graphs (in the sense of Definition \ref{Def: isomorphisms abstract GKM graphs}). Hence,
$\mathcal{L}_\chi$ induces an injective map 

\begin{align}\label{EQ: Lnd}
\mathcal{L}_{n,d}\colon 
\left\lbrace  \begin{array}{r}
\text{isomorphism classes of}  \\ 
\text{GKM graphs of Hamiltonian }  \\ 
\text{ GKM $T$-spaces of dimension $2n$}  \\
\end{array}\right\rbrace 
\longrightarrow
\left\lbrace  \begin{array}{r}
\text{isomorphism classes of}\\
\text{abstract $(n,d)$ }  \\ 
\text{ -GKM graphs}  \\
\end{array}\right\rbrace.
\end{align}

The map $\mathcal{L}_{n,d}$ is canonical. Namely, it does not depend on the choice of the linear isomorphism $\chi: 
\ell_T^* \rightarrow \Z^d$. Since $\mathcal{L}_{n,d}$ is injective, we can consider the set of  
isomorphism classes of the GKM graphs of Hamiltonian GKM  $T$-space of dimension $2n$ as a subset of the isomorphism 
classes of  abstract 
$(n,d)$-GKM graphs. An abstract $(n,d)$-GKM  graph resp. its isomorphism 
class is called  \textbf{Hamiltonian} if its isomorphism class lies in the image of  $\mathcal{L}_{n,d}$.\newline

Moreover, the concepts of projections and of the first Chern class maps defined for GKM graphs of Hamiltonian GKM 
spaces 
generalize to abstract GKM graphs.

\begin{definition}\label{Definition: Projections abstract GKM graphs}
Let $n$, $d$ and $d'$ be positive integers with $d'<d\leq n$. Let $\Gamma=(V,E)$ be a simple and connected $n$-valent 
graph
and let $\operatorname{w}:E \rightarrow \Z^d\setminus\{0\}$ and $\operatorname{w}':E \rightarrow 
\Z^{d'}\setminus\{0\}$ maps such that $(\Gamma,\operatorname{w})$ and $(\Gamma,\operatorname{w}')$ are abstract GKM 
graphs. Then $(\Gamma,\operatorname{w}')$ is a \textbf{projection} of $(\Gamma,\operatorname{w})$ if there exists a 
linear surjection $\theta: \Z^d\rightarrow \Z^{d'}$ such that 
\begin{align*}
\operatorname{w}'(e)= \theta \left( \operatorname{w}(e)\right) 
\end{align*}
for all $e\in E$.
\end{definition} 

The concept of projection of an abstract GKM graph can be considered
as a generalization of projection of a Hamiltonian GKM graph. Let $\tham$ be a Hamiltonian GKM space with GKM graph 
$\GKM$ and let $T'$ be a subtorus of $T$ such that the $T'$-action on $(M,\omega)$ is also GKM.
Note that $(M, \omega, T', i^*\circ \phi)$ is   a Hamiltonian GKM space, where $i^*:\mathfrak{t}^*\rightarrow 
(\mathfrak{t}')^*$ is the dual map of the inclusion $i: \mathfrak{t}'\rightarrow \mathfrak{t}$ from the Lie algebra of 
$T'$ to the one of $T$.
The GKM graph of this Hamiltonian GKM space is $(\Gamma_{GKM},  \eta')$, where $\eta'=i^* \circ \eta$.
So the Hamiltonian GKM graph $(\Gamma_{GKM}, \eta')$ is a projection of $\GKM$.
Let 
\begin{align*}
\chi: \ell_T^* \rightarrow \Z^d \quad \text{and}\quad \chi': \ell_{T'}^* \rightarrow \Z^{d'}
\end{align*}
be linear isomorphisms, where $d$ resp. $d'$ is the dimension of $T$ resp. $T'$.\\
Then the abstract GKM graph $(\Gamma_{GKM}, \chi' \circ \eta')$ is a projection of $(\Gamma_{GKM}, \chi \circ \eta)$.
Indeed,
\begin{align*}
\theta= \chi' \circ i^* \circ \chi^{-1}\colon \Z^d \rightarrow \Z^{d'}
\end{align*}
is a linear surjection such that
\begin{align*}
\theta\left( \chi \circ \eta (e)\right) =\chi'\circ \eta' (e)
\end{align*}
for all $e\in E_{GKM}$. That the converse is true is the statement of the following lemma.

\begin{lemma}\label{Lemma: Projections abstract GKM graphs}
Let $(\Gamma, \w)$ be an abstract $(n,d)$-GKM graph that is Hamiltonian and let $(\Gamma, \w')$ be an abstract 
$(n,d')$-GKM graph that is a  projection of $(\Gamma, \w)$. Then $(\Gamma, \w')$ is Hamiltonian. Explicitly, if $\tham$ 
is a Hamiltonian GKM space with GKM graph $\GKM$ such that
\begin{itemize}
\item $\Gamma_{GKM}=\Gamma$, and 
\item $\w=\chi \circ \eta$, where $\chi: \ell_T^*\rightarrow \Z^d$ is a linear isomorphism,
\end{itemize}
then there exists a subtorus $T'$ of $T$ such that 
\begin{itemize}
\item $(M, \omega, T', i^*\circ \phi)$ is a Hamiltonian GKM space, and
\item $\w'=\chi' \circ \eta'$, where $(\Gamma, \eta')$ is the GKM graph of $(M, \omega, T', i^*\circ \phi)$ and $\chi': 
\ell_{T'}^* \rightarrow \Z^{d'}$ is a linear isomorphism.
\end{itemize}
\end{lemma}
\begin{proof}
Since $(\Gamma, \w')$ is a projection of $(\Gamma, \w)$, there exists a linear surjection $\theta: \Z^d 
\rightarrow \Z^{d'}$ such that $\w'=\theta \circ \w$. 
The kernel of  $\theta \circ \chi$ is a subgroup of
$\ell_{T}^*$ of rank $d-d'$. Let $T'$ be the subtorus of $T$ whose Lie algebra is  
\begin{align*}
\left\lbrace \xi \in \mathfrak{t} \, \vert \, \left\langle x, \xi\right\rangle =0 \text{ for all }
x\in \ker(\theta\circ\chi)\right\rbrace. 
\end{align*}
The dimension of $T'$ is $d'$. Let $i:\mathfrak{t}'\rightarrow \mathfrak{t}$ be the inclusion from the Lie 
algebra of $T'$ to the one of $T$ and let $i^*:\mathfrak{t}^*\rightarrow (\mathfrak{t}')^*$ be its dual map.
We have that 
\begin{align*}
\ker(i^*) \cap \ell_{T}^* = \ker(\theta\circ\chi).
\end{align*}
In particular, there exists a unique linear map $\chi': \ell_{T'}^* \rightarrow \Z^{d'}$ such that the following 
diagram commutes.
\begin{center}
\begin{tikzpicture}
\node (A) {$\ell_{T}^*$};
\node (B) [below=of A] {$\ell_{T'}^*$};
\node (C) [right=of A] {$\,\,\Z^d$};
\node (D) [right=of B] {$\Z^{d'}$};
\draw[-stealth] (A)-- node[left] {\small $i^*$} (B);
\draw[-stealth] (B)-- node [below] {\small $\chi'$} (D);
\draw[-stealth] (A)-- node [above] {\small $\chi$} (C);
\draw[-stealth] (C)-- node [right] {\small $\theta$} (D);
\end{tikzpicture}
\end{center}
Since $\theta\circ \chi$ is surjective, we have that $\chi'$ is a linear isomorphism. 
Moreover, $i^*\circ \phi$ is a moment map for the $T'$-action on $(M,\omega)$. So $(M, \omega, T', i^*\circ \phi)$ is a 
Hamiltonian $T'$-space. We show that this space is GKM. Let $p\in M^T$ be a fixed point of the $T$-action 
and let $\alpha_{p,1},...,\alpha_{p,n}\in \ell_{T}^*$ be the weights of the $T$-representation on $T_pM$.
We need to show that $i^*\left( \alpha_{p,1}\right) ,\dots,i^*\left( \alpha_{p,n}\right) \in \ell_{T'}^*$
are pairwise linearly independent. By Remark \ref{Rem: Properties GKM graphs} $(i)$ the weights 
$\alpha_{p,1},...,\alpha_{p,n}$ are equal to $\eta(p,p_1),...,\eta(p,p_n)$, where $p_1,...,p_n \in M^T\setminus\{p\}$
are the $n$ fixed points such that $(p,p_i)\in E$ for $i=1,...,n$. Since $\w'=\theta\circ\w$, $\w=\chi\circ \eta$ and
$\theta\circ \chi = \chi' \circ i^*$, we have that 
$$\w'(p,p_i)=\theta\left( \w(p,p_i)\right) =\theta \circ \chi \left( \eta(p,p_i)\right) =\chi' \circ i^* \left( 
\eta(p,p_i)\right)$$
for all $i$. Since $(\Gamma, \w')$ is an abstract GKM graph, we have that $\w'(p,p_i)$ and $\w'(p,p_j)$ are linearly 
independent for $i\neq j$. So $i^* \left( \eta(p,p_i)\right)$ and $i^* \left( \eta(p,p_j)\right)$ are linearly 
independent for $i\neq j$. We  
conclude that $(M, \omega, T', i^*\circ \phi)$ is GKM. Its GKM graph is $(\Gamma, \eta')$, where $\eta' 
= i^* \circ \eta$ and $\operatorname{w}'=\chi' \circ \eta'$. Hence, $(\Gamma, \w')$ is Hamiltonian.
\end{proof}

The first Chern class  map of a Hamiltonian GKM space is the map $\mathcal{C}_1: E_{GKM}\rightarrow \Z$ that maps an 
edge $e\in E_{GKM}$ of the GKM graph to the evaluation of the first Chern class on the symplectic two-sphere that 
belongs to the edge $e$. By Lemma \ref{Lemma: Compute first Chern Class map} this map can be computed from the 
GKM graph. Therefore, we can extend the definition of the first Chern class map to abstract 
GKM graphs.

\begin{definition}\label{Def: first Chern class map for abstract GKM graphs}
	Let $(\Gamma=(V,E), \operatorname{w})$ be an abstract $(n,d)$-GKM graph. Its \textbf{first Chern class map} 
	$\mathcal{C}_1:E \rightarrow \Z$ is the map such that 
	\begin{align}\label{EQ1:Def: first Chern class map}
	\sum_{e'\in E_{i(e)}^i} \operatorname{w}(e')\, - \, \sum_{e'\in E_{t(e)}^i} 
	\operatorname{w}(e')=\mathcal{C}_1(e)\cdot
	\operatorname{w}(e)
	\end{align}
	holds for all $e\in E$. The first Chern class map is well defined, because $\operatorname{w}(e)\neq0$ for all $e 
	\in E$ 
	and by Definition \ref{Def: abstract GKM graph} $(iii)$ the left hand side of Equation \eqref{EQ1:Def: 
	first Chern class map}  is an integer multiple of $\operatorname{w}(e)$. 
\end{definition}

\begin{remark}\label{Rem: Ham vs abstract first Chern class map}
Let $\tham$ be a Hamiltonian GKM space with GKM graph $\GKM$ and let $\chi: \ell_{T}^* \rightarrow \Z^{d}$ be a linear 
isomorphism. Then by Lemma \ref{Lemma: Compute first Chern Class map} the first Chern class map $\mathcal{C}_1^{Ham}: 
E_{GKM}\rightarrow \Z$ of $\GKM$ and the first Chern class map $\mathcal{C}_1^{Abs}: E_{GKM}\rightarrow \Z$ of the 
abstract GKM graph $(\Gamma_{GKM}, \chi \circ \eta)$ coincide. Hence, Definition \ref{Def: first Chern class map for 
abstract GKM graphs} extends the definition of the first Chern  class (Definition \ref{Definition: First Chern Class 
Map }) 
for abstract GKM graphs.
\end{remark}
We close this subsection with two lemmas about  properties of the first Chern class map.

\begin{lemma}\label{Lemma: simple properties first chern map invariant under iso and projection.}
Let $(\Gamma_1, \operatorname{w}_1)$ and $(\Gamma_2, \operatorname{w}_2)$ be two abstract GKM graphs and let 
$\mathcal{C}_1^{1}: E_{1}\rightarrow \Z$ and $\mathcal{C}_1^{2}: E_{2}\rightarrow \Z$ be their first Chern class 
maps. Then the following hold.
\begin{itemize}
\item The first Chern class map is \textbf{invariant under isomorphism}, i.e., if there exists an isomorphism $(F, 
\theta)$ between $(\Gamma_1, \operatorname{w}_1)$ and $(\Gamma_2, 
\operatorname{w}_2)$, then 
$$\mathcal{C}_1^{1}(v,w)=\mathcal{C}_1^{2}((F(v), F(w)))$$
for all $(v,w)\in E_1$.
\item The first Chern class map is \textbf{invariant under projections}, i.e., if $(\Gamma_2, 
\operatorname{w}_2)$ is a 
projection of $(\Gamma_1, \operatorname{w}_1)$, so 
$\Gamma_1=\Gamma_2$ and $E_1=E_2$, then  
$$\mathcal{C}_1^{1}(e)=\mathcal{C}_1^{2}(e)$$
for all $e\in E_1=E_2$.
\end{itemize}
\end{lemma}
\begin{proof}
The proof of this lemma follows directly from the definitions.
\end{proof}

\begin{lemma}\label{Lemma: first Chern map is symmetric}
Let $(\Gamma, \operatorname{w})$ be an abstract GKM graph. Then its first Chern class map $\mathcal{C}_1:E \rightarrow 
\Z$ is \textbf{symmetric},  i.e., $\mathcal{C}_1(\bar{e})=\mathcal{C}_1(e)$ for all $e\in E$. 
\end{lemma}
\begin{proof}
By the definition of the first Chern class map we have
\begin{align*}
\sum_{e'\in E_{i(\bar{e})}^i} \operatorname{w}(e')\, - \, \sum_{e'\in E_{t(\bar{e})}^i} 
\operatorname{w}(e')=\mathcal{C}_1(\bar{e})\cdot
\operatorname{w}(\bar{e}).
\end{align*}
The left hand side of this equation is equal  to
\begin{align*}
- \left( \sum_{e'\in E_{i({e})}^i} \operatorname{w}(e')\, - \, \sum_{e'\in E_{t({e})}^i} 
\operatorname{w}(e')\right), 
\end{align*}
because $i(\bar{e})=t(e)$ and $t(\bar{e})=i(e)$. Again by the definition of first the Chern class map this latter term 
is equal to $- \mathcal{C}_1(e)\cdot \operatorname{w}(e)$. We obtain $\mathcal{C}_1(\bar{e})\cdot 
\operatorname{w}(\bar{e})=-\mathcal{C}_1(e)\cdot \operatorname{w}(e)$. Since $\operatorname{w}$ is antisymmetric, i.e.,
$\operatorname{w}(\bar{e})=-\operatorname{w}(e)$ and $\operatorname{w}(e)\neq 0$, the claim of the lemma follows.
\end{proof}

\end{subsubsection}
\end{subsection}

\begin{subsection}{GKM Skeletons}\label{subsec:GKM skeletons}

We return  to Question \ref{Question: Does there exists an (abstract) GKM graph?} for abstract GKM graphs. 
First we generalize the notation of orientation as in Definition \ref{Def: orientation edge set} to simple graphs.

\begin{definition}\label{Def: orientation edge set abstract case}
Let $\Gamma=(V,E)$ be a simple $n$-valent graph. An \textbf{orientation} $\sigma$ of the edge set $E$ is a 
subset $E^\sigma$ of $E$ such that for each $e\in E$ exactly one of the following two conditions is true.
\begin{itemize}
\item $e\in E^\sigma$ and  $\bar{e}\notin E^\sigma$
\item $\bar{e}\in E^\sigma$ and  $e\notin E^\sigma$
\end{itemize}
Here, $\bar{e}$ is the edge such that $i(\bar{e})=t(e)$ and $t(\bar{e})=i(e)$.
\end{definition}

\begin{remark}\label{Remark: orientation abs GKM graphs cardinality}
Let $\Gamma=(V,E)$ be a simple $n$-valent graph and let   $\sigma$ be an orientation for the edge set $E$. The 
cardinalities of the vertex set $V$, the edge set $E$ and the set $E^\sigma$ are related by the equations
\begin{align*}
\left| E\right| = n \cdot \left| V \right|  \quad \text{and} \quad \left| E^\sigma\right|=\frac{1}{2}\left| E\right|.
\end{align*}
 \end{remark}

Let $\Gamma=(V,E)$ be a connected, simple and $n$-valent graph and let $\mathcal{D}:E \rightarrow \Z$ be a map.
The analogue of Question \ref{Question: Does there exists an (abstract) GKM graph?}  for abstract GKM graphs is to ask 
whenever there exists a map  $\operatorname{w}: E \rightarrow \Z^d\setminus\{0\}$ 
such that $(\Gamma , \operatorname{w}) $ is an abstract $(n,d)$-GKM graph and
$\mathcal{C}_1(e)=\mathcal{D}(e)$
for all $e\in E$. Note that, by definition of the first Chern class map, such a map $\operatorname{w}$ must satisfy 
the linear equation
\begin{align}\label{EQ: linear relations D}
\sum_{e'\in E_{i(e)}^i} \operatorname{w}(e')\, - \, \sum_{e'\in E_{t(e)}^i} 
\operatorname{w}(e')=\mathcal{D}(e)\cdot
\operatorname{w}(e)
\end{align}
for all $e\in E$. So this question can be considered as a task of linear algebra. Note that if such a map 
$\operatorname{w}$ exists, then it is antisymmetric, i.e., $\operatorname{w}(\bar{e})=-\operatorname{w}(e)$ for all 
$e\in E$. Moreover, since the first Chern class map is symmetric 
(Lemma \ref{Lemma: first Chern map is symmetric}), we have that $\mathcal{D}$ is symmetric. In particular,
\eqref{EQ: linear relations D} holds for $e\in E$  if and only if it holds for $\bar{e}$. Let us fix an orientation 
$\sigma$ of the edge set $E$. So we have that \eqref{EQ: linear relations D} holds for all $e\in E$  if and only if it 
holds for all $e\in E^\sigma$. 
In order to analyze  equations \eqref{EQ: linear relations D} for all  $e\in E^\sigma$, it makes sense to fix an 
order for the 
elements of $E^\sigma$. 
For this reason we introduce GKM skeletons.

\begin{definition}\label{Def: n-GKM skeleton, supporting abstract GKM graphs}
An \textbf{$n$-GKM skeleton} $(\Gamma, \sigma, \operatorname{Ord}(E^\sigma),\mathbf{d} )$ is a connected simple 
$n$-valent graph $\Gamma=(V,E)$ together with an orientation $\sigma$ of the edge set, an ordering\footnote{By 
Remark \ref{Remark: orientation abs GKM graphs cardinality} the cardinality of $E^\sigma$ is equal to half of the one 
of 
$E$} 
$\operatorname{Ord}(E^\sigma)$ 
\begin{align*}
 e_1,\dots, e_{\frac{1}{2}\left| E\right| }
\end{align*} 
of $E^\sigma$ and a vector  $\mathbf{d}=(d_1,...,d_{\frac{1}{2}\left| E\right| })$  with integers entries.\newline
We say that  $(\Gamma, \sigma, \operatorname{Ord}(E^\sigma),\mathbf{d} )$ \textbf{supports an abstract 
$(n,d)$-GKM graph} if there exists a map $\operatorname{w}: E \rightarrow \Z^d\setminus \{0\}$ such that 
$(\Gamma, \operatorname{w})$ is an abstract $(n,d)$-GKM graph and such that $\mathcal{C}_1(e_i)= d_i$
for all $i=1,...,\frac{1}{2}\left| E\right|$, where $\mathcal{C}_1: E \rightarrow \Z$ is the first Chern class map of 
$(\Gamma, \operatorname{w})$.
\end{definition}

By the discussion above it is clear that the analog of Question \ref{Question: Does there exists an (abstract) GKM 
graph?} for abstract GKM graphs is equivalent to the question of whenever a given GKM skeleton supports an abstract 
GKM graph. 

Next we rewrite  equations  \eqref{EQ: linear relations D} so that the terms $\operatorname{w}(e')$ for $e'\in 
E_{i(e)}^i, E_{t(e)}^i $ are replaced by the  terms $\operatorname{w}(e_i)$ for $i=1,\dots. \frac{1}{2}\left| E\right| 
$. In order to do so we introduce the structure 
matrix.
\begin{definition}\label{Def: Structure Matrix}
	The \textbf{structure 
	matrix} 
	of an $n$-GKM skeleton $(\Gamma, \sigma, \operatorname{Ord}(E^\sigma), \mathbf{d})$ is the $(\frac{1}{2}\left| 
	E\right|\times \frac{1}{2}\left| E\right|)$-matrix $\mathbf{A}=(a_{j,k})_{j,k}$  given by
	\begin{align*}
	a_{j,k}= \begin{cases}
	2 \quad \quad\,\,\,\, \text{if}\, j=k, \\ 
	1 \quad \quad\,\,\,\, \text{if}\, j\neq k \,\text{and}\, i(e_j)=i(e_k)\, \text{or}\,\, t(e_j)=t(e_k),\\
	-1 \quad \quad \text{if}\, j\neq k \,\text{and}\ i(e_j)=t(e_k)\, \text{or}\,\, t(e_j)=i(e_k),\\
	0 \quad \quad\,\,\,\, \text{otherwise}.
	\end{cases}
	\end{align*}	
\end{definition}

\begin{remark}\label{Rem: Def: structure matrix}
The structure matrix $\mathbf{A}$ of an $n$-GKM skeleton $\skeleton$  does not depend on the vector $\mathbf{d}$. 
Since the graph $\Gamma$ is simple, for two indices with $j\neq k$ only the condition $ i(e_j)=i(e_k)\, \text{or}\, 
t(e_j)=t(e_k)$ or only  the condition $i(e_j)=t(e_k)\, \text{or}\, t(e_j)=i(e_k)$ can be satisfied. 
Hence, the structure matrix $\mathbf{A}$ is well defined. 
It is clear that $\mathbf{A}$ is symmetric. 
Moreover, since the graph $\Gamma$ is $n$-valent, for each $j$ there exists exactly $n-1$ indices $k\neq j$ such that 
$i(e_j)=i(e_k)\, \text{or}\, i(e_j)=t(e_k)$ and there exists exactly $n-1$ indices $k\neq j$ such that $ 
t(e_j)=i(e_k)\, \text{or}\, t(e_j)=t(e_k)$.
Hence, each row vector resp. column vector of $\mathbf{A}$ has one entry that is equal to $2$, $2(n-1)$ entries 
that are equal to $1$ or $-1$ and the remaining other $\frac{1}{2}\left| E\right| -2(n-1)-1$ entries are equal to 
zero.
\end{remark}

\begin{lemma}\label{Lemma: structure matrix eq}
Let $\skeleton$ be a GKM skeleton and let $\w: E \rightarrow \Z^d$ be an antisymmetric map, i.e., $\w(\bar{e})=-\w(e)$.
Then for each $j=1,\dots, \frac{1}{2}\left| E\right|$
\begin{align*}
\sum_{e\in E_{i(e_j)}^i} \operatorname{w}(e)\, - \, \sum_{e\in E_{t(e_j)}^i} 
\operatorname{w}(e)=\sum_{k=1}^{\frac{1}{2}\left| E\right|} a_{j,k} \operatorname{w}(e_k),
\end{align*}
where $\mathbf{A}=(a_{j,k})_{j,k=1,...,\frac{1}{2}\left| E\right|}$ is the structure matrix of the GKM skeleton.
\end{lemma}
\begin{proof}
Note that for each $e\in E$ there exists a unique index $k\in \{1,\dots, \frac{1}{2}\left| E\right|\} $ with either
$e=e_k$	or $\bar{e}=e_k$. Since $\Gamma$ is a simple graph we have 
\begin{align*}
e\in E_{i(e_j)}^i \,\, \text{and}\,\,\bar{e} \in E_{t(e_j)}^t\quad \implies \quad e=e_j.
\end{align*}
Moreover, since the map $\w$ is antisymmetric, by the definition of the structure matrix it is clear that the lemma is 
true.
\end{proof}

\begin{corollary}\label{Cor: linera equation structure matrix}
Let $(\Gamma, \sigma, \operatorname{Ord}(E^\sigma), \mathbf{d})$ be  a GKM skeleton
and let $\w: E\rightarrow \Z^d \setminus\{0\}$ be a map such that $(\Gamma, \w)$ is an abstract GKM graph.
Then the following conditions are equivalent.
\begin{itemize}
\item The abstract GKM graph $(\Gamma, \w)$ is supported by the GKM skeleton $\skeleton$.
\item For each $j=1,\dots, \frac{1}{2}\left| E\right|$
\begin{align}\label{EQ1:Cor: linera equation structure matrix}
\sum_{k=1}^{\frac{1}{2}\left| E\right|} a_{j,k} \operatorname{w}(e_k)\,=\, d_j\cdot\operatorname{w}(e_j),
\end{align}
where $\mathbf{A}=(a_{j,k})_{j,k=1,...,\frac{1}{2}\left| E\right|}$ is the structure matrix of the GKM skeleton. 
\end{itemize}
\end{corollary}
\begin{proof}
If $(\Gamma, \w)$ is supported by $\skeleton$, then $d_j=\mathcal{C}_1(e_j)$ for all $j=1,...,\frac{1}{2}\left| 
E\right|$. By the definition of the first Chern class map, we have
\begin{align}\label{EQ2:Cor: linera equation structure matrix}
\sum_{e\in E_{i(e_j)}^i} \operatorname{w}(e)\, - \, \sum_{e\in E_{t(e_j)}^i} 
\operatorname{w}(e)=d_j\cdot\operatorname{w}(e_j),
\end{align}
for all $j=1,...,\frac{1}{2}\left| E\right|$. Since the weight map $\w$ is antisymmetric, by Lemma \ref{Lemma: 
structure matrix eq} the left-hand sides of \eqref{EQ1:Cor: linera equation structure matrix}
and \eqref{EQ2:Cor: linera equation structure matrix} coincide for all $j=1,...,\frac{1}{2}\left| E\right|$.
Hence, \eqref{EQ1:Cor: linera equation structure matrix} holds for all $j=1,...,\frac{1}{2}\left| E\right|$.
\newline
On the other hand, if \eqref{EQ1:Cor: linera equation structure matrix} holds for all $j=1,...,\frac{1}{2}\left| 
E\right|$,
then by Lemma \ref{Lemma: structure matrix eq}  also \eqref{EQ2:Cor: linera equation structure matrix} holds for
all $j=1,...,\frac{1}{2}\left| E\right|$. So by the definition of the first Chern class map $d_j=\mathcal{C}_1(e_j)$ 
for all $j=1,...,\frac{1}{2}\left| E\right|$, i.e., $(\Gamma, \w)$ is supported by $\skeleton$.
\end{proof}

\begin{definition}\label{Def: defect, fundamental system}
Let  $(\Gamma, \sigma, \operatorname{Ord}(E^\sigma),\mathbf{d} )$  be an $n$-GKM skeleton. The \textbf{defect} 
$\delta$ of the $n$-GKM skeleton is the dimension of the kernel of $\mathbf{A-D}$, where $\mathbf{A}$ is the structure 
matrix and $\mathbf{D}$ is the diagonal matrix given by
\begin{align*}
\mathbf{D}=\operatorname{Diag}(d_1,\dots, d_{\frac{1}{2}\left| E\right| }).
\end{align*}
If the defect $\delta$ is positive, a \textbf{fundamental system} of 
	$(\Gamma,\sigma,\operatorname{Ord}(E^\sigma),\mathbf{d} )$ is a collection of  $\frac{1}{2}\left| E\right|$ vectors
	$f_1,\dots, f_{\frac{1}{2}\left| E\right|}$ in $\Q^\delta$ such that the transposes of the row vectors of the 
	$\left( \delta \times\frac{1}{2}\left| 
	E\right|\right) $-matrix
	\begin{align*}
	\mathbf{F}= \left( \begin{matrix}
	\vert & &\vert \\
	f_1   &\dots &f_{\frac{1}{2}\left| E\right|}   \\
	\vert & &\vert
	\end{matrix}\right) 
	\end{align*}
	form a basis of the kernel of $\mathbf{A-D}$.  
\end{definition}

\begin{remark}\label{Rem: Def: defect, fundamental system}
	Let $(\Gamma, \sigma, \operatorname{Ord}(E^\sigma),\mathbf{d} )$  be an $n$-GKM skeleton with a positive defect 
	$\delta$.
	Since the matrices  $\mathbf{A}$ and $\mathbf{D}$ have only integer entries, there exist vectors with integer 
	entries
	that form a basis of the kernel of $\mathbf{A-D}$. Hence, an  $n$-GKM skeleton with a positive defect always admits
	a fundamental system. Furthermore, a fundamental system is unique up to 
	$\operatorname{GL}(\Q,\delta)$-transformations.
	This means that whenever $f_1,\dots, f_{\frac{1}{2}\left| E\right|}$ and $f'_1,\dots, f'_{\frac{1}{2}\left| 
	E\right|}$ 
	are 
	two fundamental systems for $(\Gamma, \sigma, \operatorname{Ord}(E^\sigma),\mathbf{d} )$, then there exists an  
	invertible linear map $\Q^\delta \rightarrow \Q^\delta$ that maps $f_i$ onto $f'_i$ for all $i=1,..., 
	\frac{1}{2}\left| E\right|$ .
\end{remark}

\begin{proposition}\label{Pro: defect and FS}
Let $(\Gamma, \sigma, \operatorname{Ord}(E^\sigma),\mathbf{d} )$  be an $n$-GKM skeleton and let $\w: E\rightarrow \Z^d \setminus\{0\}$
be a map such that $(\Gamma, \w)$  is an abstract $(n,d)$-GKM graph. The following two conditions are equivalent.
\begin{itemize}
\item The abstract GKM graph $(\Gamma, \w)$ is supported by the GKM skeleton $\skeleton$.
\item The defect $\delta$ of $\skeleton$ is greater or equal to $d$ and for any fundamental system 
$f_1,\dots, f_{\frac{1}{2}\left| E\right|}$ of $\skeleton$, 
there exists a $(d \times \delta)$- matrix $\mathbf{M}$ of rank $d$ such that 
\begin{align*}
\mathbf{M}\cdot f_i = \operatorname{w}(e_i)
\end{align*}
for all $i=1,..., \frac{1}{2}\left| E\right|$ .	 
\end{itemize}
\end{proposition}

\begin{proof}
Let $\mathbf{W}$  be the $\left( d \times\frac{1}{2}\left| E\right|\right) $-matrix whose $i$-th column vector 
is $\operatorname{w}(e_i)$. 
Note that by $(i)$ of Definition \ref{Def: abstract GKM graph}, we have that the dimension of the span of the vectors
$\w(e)$ for $e\in E$ is  $d$. Since $\w(\bar{e})=-\w(e)$ and for each $e\in E$ there exists a unique index  $k$ such 
that either $e=e_k$ or $\bar{e}=e_k$, the rank of $\mathbf{W}$ is $d$. \\
Assume that $(\Gamma,\w)$ is supported by  $\skeleton$. 
By Corollary \ref{Cor: linera equation structure matrix} for each $j=1,\dots, \frac{1}{2}\left| E\right|$ we have
\begin{align}\label{EQ1:Pro: defect and FS}
\sum_{k=1}^{\frac{1}{2}\left| E\right|} a_{j,k} \operatorname{w}(e_k)\,=\, d_j\cdot\operatorname{w}(e_j),
\end{align}
where $\mathbf{A}=(a_{j,k})_{j,k=1,...,\frac{1}{2}\left| E\right|}$ is the structure matrix.
This is equivalent to the matrix equation
\begin{align*}
\mathbf{A}\cdot \mathbf{W}^\intercal\, = \,\mathbf{D} \cdot \mathbf{W}^\intercal,
\end{align*}
where $\mathbf{D}=\operatorname{Diag}(d_1,\dots, d_{\frac{1}{2}\left| E\right| })$ and $\mathbf{W}^\intercal$ is 
the 
transpose of $\mathbf{W}$. We deduce that the column vectors of $\mathbf{W}^\intercal$ are elements of the kernel of 
$\mathbf{A-D}$. Since the matrix rank of $\mathbf{W}^\intercal$ is also equal to $d$, we have that the dimension of 
the kernel of $\mathbf{A-D}$ is greater or equal to $d$. Hence, the defect $\delta$ must be greater or equal to $d$.
Moreover, let $f_1,\dots, f_{\frac{1}{2}\left| E\right| }$ be a fundamental system. Since transposes of the row 
	vectors 
	of the matrix
	\begin{align}\label{EQ2:Pro: defect and FS}
	\mathbf{F}= \left( \begin{matrix}
	\vert & &\vert \\
	f_1   &\dots &f_{\frac{1}{2}\left| E\right|}   \\
	\vert & &\vert
	\end{matrix}\right) 
	\end{align}
	form a basis of the kernel of $\mathbf{A-D}$, there exists a $(d \times \delta)$- matrix $\mathbf{M}$ such 
	that 
	\begin{align*}
	\mathbf{M}\cdot \mathbf{F} = \mathbf{W}.
	\end{align*}
	This is equivalent to $\mathbf{M}\cdot f_i = \operatorname{w}(e_i)$  for all $i=1,..., \frac{1}{2}\left| E\right|$. 
	Moreover, since the matrix rank of $\mathbf{W}$ is $d$, the matrix rank of $\mathbf{M}$ is also $d$.\\
	On the other hand, assume that the defect $\delta$ is greater or equal to $d$ and that there exists a   $(d \times 
	\delta)$- matrix $\mathbf{M}$ such that 
	 $\mathbf{M}\cdot f_i = \operatorname{w}(e_i)$  for all $i=1,..., \frac{1}{2}\left| E\right|$, where  $f_1,\dots, 
	 f_{\frac{1}{2}\left| E\right| }$ is a fundamental system. \\
	 So we have $\mathbf{M}\cdot \mathbf{F} = \mathbf{W}$, where $\mathbf{F}$ is the $\left( \delta 
	 \times\frac{1}{2}\left| E\right|\right) $ -matrix as in \eqref{EQ2:Pro: defect and FS}. Note that 
	 $\mathbf{A}\cdot \mathbf{F}^\intercal = \mathbf{D}\cdot \mathbf{F}^\intercal$. Therefore, we have
\begin{align*}
\mathbf{A}\cdot \mathbf{W}^\intercal= 
\mathbf{A} \cdot \left( \mathbf{M}\cdot\mathbf{F}\right)^\intercal=
\left( \mathbf{A} \cdot \mathbf{F}^\intercal \right)  \cdot \mathbf{M}^\intercal= 
\left( \mathbf{D} \cdot \mathbf{F}^\intercal \right)  \cdot \mathbf{M}^\intercal=
\mathbf{D} \cdot \left( \mathbf{M}\cdot\mathbf{F}\right)^\intercal=
\mathbf{D}\cdot \mathbf{W}^\intercal.
\end{align*}
Note that $\mathbf{A}\cdot \mathbf{W}^\intercal=\mathbf{D}\cdot \mathbf{W}^\intercal$ implies that \eqref{EQ1:Pro: 
defect and FS} holds for all
$j=1,\dots, \frac{1}{2}\left| E\right|$. Hence, by Corollary \ref{Cor: linera equation structure matrix} the abstract 
GKM graph is supported by the GKM 
skeleton.
	
\end{proof}

\begin{remark}\label{Rem1:Pro: defect and FS }
	Let $(\Gamma, \sigma, \operatorname{Ord}(E^\sigma),\mathbf{d} )$  be an $n$-GKM skeleton that supports an abstract
	$(n,d)$-GKM graph $(\Gamma, \operatorname{w})$ and let $f_1,\dots, f_{\frac{1}{2}\left|  E\right| }$ be a 
	fundamental 
	system. By Proposition \ref{Pro: defect and FS}, there exists a matrix $\mathbf{M}$ such that $\mathbf{M}\cdot f_i = 
	\operatorname{w}(e_i)$ for $i=1,\dots ,\frac{1}{2}\left|  E\right|$. Note that this matrix $\mathbf{M}$ is unique,
	so we call it the \textbf{weight transformation matrix}.
\end{remark}

Proposition \ref{Pro: defect and FS} gives a necessary condition for an $n$-GKM skeleton to support an abstract 
$(n,d)$-GKM 
graph, namely its defect must be  greater or equal to $d$, and an approach to construct a weight map $\operatorname{w}$ 
such that $(\Gamma, \operatorname{w})$ is an abstract GKM graph supported by the $n$-GKM skeleton. In the following 
example we apply this approach.

\begin{example}\label{Example: tori with four vertices}
Let $\Gamma =(V,E)$ be the complete graph with exactly four vertices $V=\{v_1,v_2,v_3,v_4\}$, i.e., the edge set is 
$E=\left\lbrace (v_i,v_j)\, \vert \, i,j=1,...4\,\, \text{and}\,\, i\neq j \right\rbrace .$
	
\begin{figure}[htbp]
\begin{center}
\includegraphics[width=8cm]{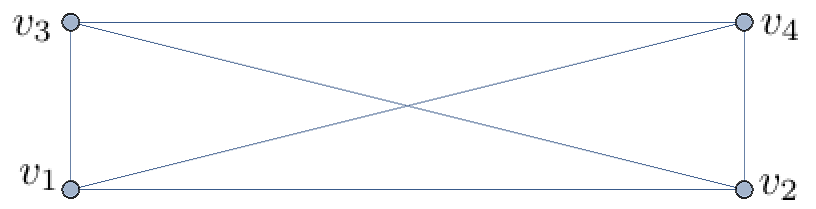}
\caption{The complete graph with  four vertices}
\label{Figure: The Complete Graph with  four Vertices}
\end{center}
\end{figure}

This graph is connected, simple and $3$-valent. Let $\sigma$ be the orientation of the edge set $E$ given by 
$(v_i,v_j)\in E^\sigma $ if and only if $i<j$. We fix the following ordering $\operatorname{Ord(E^\sigma)}$
\begin{align*}
e_1:=(v_1,v_2), \quad  e_2:=(v_1,v_3), \quad e_3:=(v_1,v_4),\\
e_4:=(v_2,v_3), \quad  e_5:=(v_2,v_4), \quad e_6:=(v_3,v_4).
\end{align*}
	
Let $\mathbf{d}$ be equal to $(4,4,4,4,4,4)$. So $\skeleton$  is a $3$-GKM skeleton. The structure matrix is 
	
\begin{align*}
\mathbf{A}=\left( \begin{array}{rrrrrr}
2 & 1 & 1 &$-1$ &$-1$&0 \\ 
1 & 2 & 1 &1  &0 &$-1$ \\
1 & 1 & 2 &0  &1 &1 \\
$-1$ & 1 & 0 &2  &1 &$-1$ \\
$-1$ & 0 & 1 &1  &2 &1 \\
0 & $-1$ & 1 &$-1$  &1 &2 \\
\end{array}\right).
\end{align*}
	
Let $\mathbf{D}$ be the diagonal matrix given by $\operatorname{Diag}(4,4,4,4,4,4)$. The dimension of the kernel of 
$\mathbf{A-D}$ is $3$. So the defect of the $3$-GKM skeleton is $3$. Consider the matrix
	
\begin{align*}
\mathbf{F}=\left( \begin{array}{rrrrrr}
0 & $-1$ & 0 &$-1$ &0&1 \\ 
$-1$& 0 & 0 &1  &1 &0 \\
1 & 1 & 1&0  &0 &0 \\
\end{array}\right).
\end{align*}

The transposes of its row vectors form a basis of the kernel of $\mathbf{A-D}$. So the column vectors $f_1,...,f_6$	of 
$\mathbf{F}$ form a fundamental system of the $3$-GKM skeleton. 

By Proposition \ref{Pro: defect and FS}, the $3$-GKM skeleton supports an 
abstract $(3,d)$-GKM graph if and only if there
exists a $(d\times 3)$-matrix $\mathbf{M}$ such that the antisymmetric map $\operatorname{w}:E\rightarrow \Z^d$ given by

\begin{align*}
\operatorname{w}(e_i)=\mathbf{M}\cdot f_i\quad \text{for} \quad i=1,...,6 
\end{align*}
forms together with $\Gamma$ an abstract $(3,d)$-GKM graph. 

Let $\mathbf{M}$ be the identity matrix of size $(3\times 3)$, then $(\Gamma, \operatorname{w})$ is an abstract 
$(3,3)$-GKM graph. Indeed, the pair $(\Gamma, \operatorname{w})$ satisfies condition $(i)$ and $(ii)$ of Definition 
\ref{Def: abstract GKM graph}.  Moreover,  $(iii)$ of Definition \ref{Def: abstract GKM graph} is also satisfied.  For 
example, consider the edge $e_1=(v_1,v_2)$. We have 
	
	\begin{align*}
	E_{v_1}^i=\{(v_1,v_2),(v_1,v_3),(v_1,v_4)\} \quad  \text{and}\quad E_{v_2}^i=\{(v_2,v_1),(v_2,v_3),(v_2,v_4)\}.
	\end{align*}
	Let $\nabla_{(v_1,v_2)}\colon E_{v_1}^i \rightarrow E_{v_2}^i$ be the connection  given by
	\begin{align*}
	(v_1,v_2)\mapsto (v_2,v_1)\,,\quad (v_1,v_3)\mapsto (v_2,v_4)\quad \text{and}\quad (v_1,v_3)\mapsto (v_2,v_4).
	\end{align*}
	Then 
	\begin{align*}
	&\operatorname{w}(v_1,v_2)\,-\, \operatorname{w}(\nabla_{(v_1,v_2)}(v_1,v_2))\,=\, 2\cdot 
	\operatorname{w}(v_1,v_2),\\
	&\operatorname{w}(v_1,v_3)\,-\, \operatorname{w}(\nabla_{(v_1,v_2)}(v_1,v_3))\,=\, 1\cdot 
	\operatorname{w}(v_1,v_2)\,\,
	\quad \text{and}\\
	&\operatorname{w}(v_1,v_4)\,-\, \operatorname{w}(\nabla_{(v_1,v_2)}(v_1,v_4))\,=\, 1\cdot \operatorname{w}(v_1,v_2).
	\end{align*}
	So $\nabla_{(v_1,v_2)}$ is a compatible connection along the edge $(v_1,v_2)$. Similarly, for each edge in $E$ there exists a compatible connection.  Hence, $(\Gamma, \w)$ is an abstract $(3,3)$-GKM 
	graph that is supported by the $3$-GKM skeleton.
	 Note that this abstract $(3,3)$-GKM graph is Hamiltonian. 
	Indeed it comes from the standard toric action on $\C P^3$.
\end{example}

\begin{subsubsection}{Isomorphic GKM Skeletons}
\begin{definition}\label{Def: Isom skeleton}
	Let $\skeleton$ and $(\Gamma',\sigma', \operatorname{Ord}((E')^{\sigma'}), \mathbf{d'})$ be two GKM skeletons. An 
	\textbf{isomorphism} between  these GKM skeletons is an isomorphism $F: V \rightarrow V'$ between the simple graphs 
	$\Gamma=(V,E)$ and $\Gamma'=(V',E')$ (see Definition \ref{Def: isomorphism simple graphs}) such that 
	the following holds. Let  $i=1,...,\frac{1}{2}\left| E\right|$ be any index and let $j\in \{1,...,\frac{1}{2}\left| 
	E\right|\}$ be the unique index such that the $i$-th edge of $E^\sigma$ is mapped under the bijection $E\rightarrow 
	E'$ induced by $F$ (see Remark \ref{Rem: bijection edges}) to either the $j$-th edge of $(E')^{\sigma'}$ or to the 
	reversed of the  $j$-th edge of $(E')^{\sigma'}$. Then $d_i=d'_j$, where $d_i$ is the $i$-th entry of the vector 
	$\mathbf{d}$ and $d'_j$ is the $j$-th entry of the vector $\mathbf{d'}$.
\end{definition} 

\begin{corollary}\label{Cor: Isom skeleton}
	Let $\skeleton$ and $(\Gamma',\sigma', \operatorname{Ord}((E')^{\sigma'}), \mathbf{d'})$ be two isomorphic $n$-GKM 
	skeletons. Let $(\Gamma, \w)$ be an abstract $(n,d)$-GKM graph that is supported by $\skeleton$. Then there exists 
	an 
	abstract $(n,d)$-GKM graph that is supported by $(\Gamma',\sigma', \operatorname{Ord}((E')^{\sigma'}), \mathbf{d'})$
	and that is isomorphic to $(\Gamma,\w)$. 
\end{corollary} 
\begin{proof}
Let $F:V \rightarrow V'$ be an isomorphism between $\skeleton$ and $(\Gamma',\sigma', 
\operatorname{Ord}((E')^{\sigma'}), \mathbf{d'})$. Let $\w':E' \rightarrow \Z^d\setminus \{0\}$ be the map defined by
$$\w'(v',w')=\w(F^{-1}(v), F^{-1}(w')) $$
for any edge $(v',w')\in E'$. The pair $(\Gamma', \w')$ is an abstract GKM graph and $(F, \operatorname{Id})$
is an isomorphism between $(\Gamma, \w)$ and  $(\Gamma', \w')$, where $\operatorname{Id}: \Z^d \rightarrow \Z^d$
is the identity map. Moreover, by Lemma \ref{Lemma: simple properties first chern map invariant under iso and 
projection.} the abstract GKM graph $(\Gamma' \w')$ is supported by  $(\Gamma',\sigma', 
\operatorname{Ord}((E')^{\sigma'})$.
\end{proof}

\begin{remark}\label{Rem: iso skeletons}
Let $(\Gamma,\w)$ be an abstract GKM graph, let $E^\sigma$ be an orientation for the edge set, let 
$\operatorname{Ord}(E^\sigma)$ be an ordering and let $\mathbf{d}\in \Z^{\frac{1}{2}\left| E\right|} $ the vector
whose $j$-entry is $\mathcal{C}_1(e_j)$. Then $\skeleton$ is a GKM skeleton that supports $(\Gamma, \w)$.
In particular, any two GKM skeletons that both support $(\Gamma, \w)$ are isomorphic.
\end{remark}
\end{subsubsection}

\begin{subsubsection}{The Kernel Condition (K1)}

We derive a necessary condition for a GKM skeleton to support an abstract GKM graph out of item $(ii)$ of Definition 
\ref{Def: abstract GKM graph}, saying that for any abstract $(n,d)$-GKM graph $(\Gamma, \w)$ and any vertex $v$ of 
$\Gamma$ the vectors $\w(e)$ for $e\in E$ with $i(e)=v$ are pairwise linearly independent  in $\Z^d$ over $\Z$.

\begin{lemma}\label{Lem: linarly inde. Cond. for skeletons}
Let $\skeleton$ be an $n$-GKM skeleton and let $\mathbf{A}=(a_{j,k})_{j,k=1,...,\frac{1}{2}\left| E\right|}$ be its 
structure matrix. Let $\w: E \rightarrow \Z^d$ be an antisymmetric map. Then the map $\w$ satisfies 
condition $(ii)$ of Definition \ref{Def: abstract GKM graph} if and  only if  for any two indices $j,k\in\{1,\dots, 
\frac{1}{2}\left| E\right|\}$ with $a_{j,k}=1$ or $a_{j,k}=-1$ the vectors $\w(e_j)$ and $\w(e_k)$ are linearly 
independent.
\end{lemma}
\begin{proof}
Note that since the map $\w$ is antisymmetric, it satisfies condition $(ii)$ of Definition \ref{Def: abstract GKM 
graph} if and only if for any two edges $e,e' \in E^{\sigma}$  with $e\neq e'$  and such that 
one of the following conditions holds:
\begin{align}\label{EQ1:Lem: linarly inde. Cond. for skeletons}
i(e)=i(e'),\quad i(e)=t(e'), \quad t(e)=i(e')\quad \text{or}\quad t(e)=t(e'),
\end{align}
the vectors $\w(e)$ and $\w(e')$ are linearly independent. 
By the definition of the structure matrix, one of the condition in \eqref{EQ1:Lem: linarly inde. Cond. for skeletons}
is satisfied if and only if  $a_{j,k}=1$ or $a_{j,k}=-1$. 
With these considerations it is obvious that the statement of the lemma is true.
\end{proof}

\begin{definition}\label{Def: Kernel Condition (K1)}
A GKM skeleton with a positive defect satisfies the  \textbf{Kernel Condition (K1)} if for any 
fundamental 
system $f_1,\dots, f_{\frac{1}{2}\left| E\right| }$ and for each two indices $j,k\in\{1,\dots, \frac{1}{2}\left| 
E\right|\}$ with $a_{j,k}=1$ or $a_{j,k}=-1$, the vectors $f_j$ and $f_k$ are linearly independent.
\end{definition} 

That the Kernel Condition (K1) is necessary for a GKM skeleton to support an abstract GKM  is a direct consequence 
of Proposition  \ref{Pro: defect and FS} and Lemma \ref{Lem: linarly inde. Cond. for skeletons}.

\begin{corollary}\label{Cor: (K1) is a necessary condition}
A GKM skeleton that supports an abstract GKM graph has a positive defect and  satisfies the Kernel Condition (K1).
\end{corollary}
\begin{proof}
Let $(\Gamma, \sigma, \operatorname{Ord}(E^\sigma), \mathbf{d})$ be  an $n$-GKM skeleton that supports an abstract 
$(n,d)$-GKM 
graph $(\Gamma, \w)$. By Proposition \ref{Pro: defect and FS}  the defect $\delta$ of the GKM skeleton is  greater or 
equal to $d$ and $\delta$ is indeed positive. Again by Proposition \ref{Pro: defect and FS}, whenever $f_1,\dots, 
f_{\frac{1}{2}\left| E\right|}$ is a fundamental system,  there exists a $(d\times \delta)$-matrix $\mathbf{M}$ 
such that
\begin{align*}
\w(e_i)=\mathbf{M}\cdot f_i
\end{align*}
for all $i=1,\dots, \frac{1}{2}\left| E\right|$. By Lemma \ref{Lem: linarly inde. Cond. for skeletons} we have that 
if $a_{j,k}=1$ or $a_{j,k}=-1$, then $\w(e_j)$ and $\w(e_k)$ are linearly independent. So $f_j$  and $f_k$ are linearly 
independent  if $a_{j,k}=1$ or $a_{j,k}=-1$. Hence, $\skeleton$ satisfies the Kernel Condition (K1). 
\end{proof}

\end{subsubsection}

\begin{subsubsection}{The Kernel Condition (K2)}
	
Now we consider item $(iii)$ of Definition \ref{Def: abstract GKM graph},  the existence of a compatible 
connection for each edge. Let $(\Gamma, \sigma, \operatorname{Ord}(E^\sigma), \mathbf{d})$ be a GKM skeleton and 
let $V$ be the vertex set of $\Gamma$. We associate to each vertex $v \in V$  the following  set of indices

\begin{align*}
\operatorname{IND}_v=\left\lbrace k \in \{1,\dots,\dfrac{1}{2} \left| E\right|\} \, \mid \, i(e_k)=v \, \text{or} 
\,\, t(e_k)=v \right\rbrace.
\end{align*}
Recall that a connection along an edge $e_j=(v,w)\in E^\sigma$ is a bijection $\nabla_{e_j}: E_v^i \rightarrow 
E_w^i$ such that $\nabla_{e_j}(v,w)=(w,v)$. Such a connection induces a bijection
	
\begin{align*}
\widetilde{\nabla}_j \colon \operatorname{IND}_v \rightarrow \operatorname{IND}_w
\end{align*}
such that $\widetilde{\nabla}_j(j)=j$. Indeed for each $e\in E_v^i $ resp. $e'\in E_w^i $ there exists
	a unique $k\in \operatorname{IND}_v$ resp. $k'\in \operatorname{IND}_w$ such that either $e=e_k$ or $\bar{e}=e_k$
	resp. either $e'=e_{k'}$ or $\bar{e}'=e_{k'}$. So $k$ is mapped under $\widetilde{\nabla}_j$ to $k'$ if
	$\nabla_{e_j}(e)=e'$. Note that since $e_j=(v,w)$, the graph is simple,  and $\nabla_{e_j}(v,w)=(w,v)$, we have that 
	$\operatorname{IND}_v\cap\operatorname{IND}_w=\{j\}$ and $\widetilde{\nabla}_j(j)=j$.

\begin{example}\label{Exam: Maps induced by Connections}
Consider the complete graph with four vertices together with the orientation and the ordering as in Example 
\ref{Example: tori with four vertices}. Consider the edge $e_1=(v_1,v_2)$. We have 
\begin{align*}
E_{v_1}^i=\{(v_1,v_2),(v_1,v_3),(v_1,v_4)\} \quad  \text{and}\quad E_{v_2}^i=\{(v_2,v_1),(v_2,v_3),(v_2,v_4)\}.
\end{align*}
There exist exactly two connections $ E_{v_1}^i \rightarrow E_{v_2}^i$ along the edge $e_1$, namely 
\begin{align*}
\left\lbrace \begin{array}{r} (v_1,v_2) \mapsto (v_2,v_1)  \\ (v_1,v_3) \mapsto (v_2,v_3)  \\  (v_1,v_4) 
\mapsto 
(v_2,v_4)
\end{array}\right\rbrace 
\quad \text{and} \quad
\left\lbrace \begin{array}{r} (v_1,v_2) \mapsto (v_2,v_1)  \\ (v_1,v_3) \mapsto (v_2,v_4)  \\  (v_1,v_4) 
\mapsto 
(v_2,v_3)
\end{array}\right\rbrace. 
\end{align*}
Moreover, we have that $e_2=(v_1,v_2)$, $e_3=(v_1,v_3)$, $e_4=(v_2,v_3)$ and $e_5=(v_2,v_4)$.  Hence, 
\begin{align*}
\operatorname{IND}_{v_1}=\{1,2,3\} \quad \text{and} \quad \operatorname{IND}_{v_2}=\{1,4,5\}
\end{align*} 
and the bijections $\operatorname{IND}_{v_1}\rightarrow \operatorname{IND}_{v_2}$ induced by the two 
connections along $e_1$ are 
\begin{align*}
\left\lbrace \begin{array}{r} 1 \mapsto 1  \\ 2 \mapsto 4  \\  3 \mapsto 5 \end{array}\right\rbrace 
\quad \text{and} \quad
\left\lbrace \begin{array}{r} 1 \mapsto 1  \\ 2 \mapsto 5  \\  3 \mapsto 4 \end{array}\right\rbrace. 
\end{align*}
\end{example}

\begin{lemma}\label{Lem: connections for skeltons}
Let $\skeleton$ be  a GKM skeleton and let $\mathbf{A}=(a_{j,k})_{j,k=1,...,\frac{1}{2}\left| E\right|}$ be its 
structure 
matrix. Let $\w : E \rightarrow \Z^d$ be an antisymmetric map. Then the map $\w$ satisfies  condition $(iii)$ of 
Definition \ref{Def: abstract GKM graph} if and only if for all $j=1,\dots, \frac{1}{2}\left| E\right|$ there exists a 
connection $\nabla_{e_j}$ along the edge $e_j$ such that  for all $k\in \operatorname{IND}_{i(e_j)}$
\begin{align*}
a_{j,k} \cdot \w(e_k)+a_{j,\widetilde{\nabla}_j(k)} \cdot  
\w(e_{\widetilde{\nabla}_j(k)}) \quad \text{is an integer multiple of} \,\, \w(e_j),
\end{align*}
where $\widetilde{\nabla}_j \colon \operatorname{IND}_{i(e_j)} \rightarrow \operatorname{IND}_{t(e_j)}$ is the 
bijection induced by $\nabla_{e_j}$.
\end{lemma}

\begin{proof}
We first note that since the map $\w$ is antisymmetric, a connection $\nabla_e: E^i_{i(e)}\rightarrow E^i_{t(e)}$ along 
an edge $e$ 
is compatible with $\w$ if and only if its inverse $\nabla_e^{-1}: E^i_{t(e)}\rightarrow E^i_{i(e)}$ is compatible with
$\w$. Therefore, $\w$ satisfies $(iii)$ of Definition \ref{Def: abstract GKM graph} if and only if for each $j=1,\dots, \frac{1}{2}\left| E\right|$ there 	exists 
a connection $\nabla_{e_j}$ along the edge $e_j$ that is compatible with $\w$. 
We show that for $j=1,\dots, \frac{1}{2}\left| E\right|$, a connection $$\nabla_{e_j}: E_{i(e_j)}^i \rightarrow 
E_{t(e_j)}^i $$ 
 along the edge $e_j$ is compatible with $\w$, i.e., 
 \begin{align}\label{EQ1: Lem: After projection compatible connnections}
\operatorname{w}(e)- \operatorname{w} (\nabla_{e_j}(e))
\end{align}
 is an integer multiple of $\w(e_j)$ for all $e\in E_{i(e_j)}^i$ if and only if
 \begin{align}\label{EQ2: Lem: After projection compatible connnections}
a_{j,k}\cdot \operatorname{w}(e_k)+a_{j,\widetilde{\nabla}_j(k)}  \operatorname{w} 
(e_{\widetilde{\nabla}_j(k)})
\end{align}
 is an integer multiple of $\w(e_j)$ for all $k\in \operatorname{IND}_{i(e_j)}$,
 where $\widetilde{\nabla}_j \colon \operatorname{IND}_{i(e_j)} \rightarrow \operatorname{IND}_{t(e_j)}$ is the 
bijection induced by $\nabla_{e_j}$.

Indeed, for $e\in E_{i(e_j)}^i$, let $k\in \operatorname{IND}_{i(e_j)}$ be the unique index such that either $e_k=e$ or $e_k=\bar{e}$.
If $k=j$, then 
$\nabla_{e_j}(e)=\bar{e}$ and $\widetilde{\nabla}_{e_j}(j)=j$. Since $\w$ is antisymmetric the term 
\eqref{EQ1: Lem: After projection compatible connnections} is equal to $2\cdot \w(e)$. Since $a_{j,j}=2$ the term 
\eqref{EQ2: Lem: After projection compatible connnections}
is equal to $4\cdot \w(e)$.

Now assume that $k\neq j$. If $e=e_k$, then $i(e_k)=i(e_j)$ and by the definition of the structure matrix we have that 
$a_{j,k}=1$ and $\w(e)=a_{j,k} \cdot \w(e_k)$. If $\bar{e}=e_k$, then $t(e_k)=i(e_j)$ and by the definition of the 
structure matrix we have that $a_{j,k}=-1$ and $\w(e)=-\w(\bar{e})=a_{j,k} \cdot \w(e_k)$. For the same reason 
we have $-\w(\nabla_{e_j}(e))=a_{j,\widetilde{\nabla}_j(k)} \cdot \operatorname{w} (e_{\widetilde{\nabla}_j(k)})$.
So the terms \eqref{EQ1: Lem: After projection compatible connnections} and \eqref{EQ2: Lem: After projection 
compatible connnections} are equal.

\end{proof}
	
\begin{definition}\label{Def: Kernel Condition (K2)}
Assume that the defect is positive. We say that  
$(\Gamma, \sigma, \operatorname{Ord}(E^\sigma), \mathbf{d})$ satisfies the \textbf{Kernel Condition (K2)} if for any 
fundamental system $f_1,\dots, f_{\frac{1}{2}\left| E\right| }$ and for each $e_j\in E^\sigma$  there exists 
a connection $\nabla_{e_j}$ along $e_j$ such that for each $k\in\operatorname{IND}_{i(e_j)}$
\begin{align*}
a_{j,k}\cdot f_k+a_{j,\widetilde{\nabla}_j(k)} \cdot f_{\widetilde{\nabla}_j(k)}\quad \text{is an integer multiple 
of}\,\, f_j, 
\end{align*}
where $\widetilde{\nabla}_j \colon \operatorname{IND}_{i(e_j)} \rightarrow \operatorname{IND}_{t(e_j)}$ is the 
bijection induced by   $\nabla_{e_j}$ . 
\end{definition}
	
\begin{corollary}\label{Cor: low defect and (K2)}
Let $(\Gamma, \sigma, \operatorname{Ord}(E^\sigma), \mathbf{d})$ be  an $n$-GKM skeleton with a positive defect 
$\delta$ that supports an abstract $(n,\delta)$-GKM graph $(\Gamma, \w)$. Then this GKM skeleton satisfies the Kernel 
Condition (K2).
\end{corollary}
\begin{proof}
Let $f_1,\dots, f_{\frac{1}{2}\left| E\right|}$ be a fundamental system. By Proposition \ref{Pro: defect and FS}
there exists an invertible $(\delta \times \delta)$-matrix $\mathbf{M}$ such that  $\mathbf{M}\cdot 
f_j=\operatorname{w}(e_j)$ for all $j=1,\dots, \frac{1}{2}\left| E\right|$. By Lemma \ref{Lem: connections for skeltons}
for all $j=1,\dots, \frac{1}{2}\left| E\right|$ there exists a 
connection $\nabla_{e_j}$ along the edge $e_j$ such that  for all $k\in \operatorname{IND}_{i(e_j)}$
\begin{align*}
a_{j,k}\cdot \w(e_k)+a_{j,\widetilde{\nabla}_j(k)} \cdot  
\w(e_{\widetilde{\nabla}_j(k)}) \quad \text{is an integer multiple of} \, \w(e_j),
\end{align*}
where $\widetilde{\nabla}_j \colon \operatorname{IND}_{i(e_j)} \rightarrow \operatorname{IND}_{t(e_j)}$ is the 
bijection induced by $\nabla_{e_j}$ and $\mathbf{A}=(a_{j,k})_{j,k=1,...,\frac{1}{2}\left| E\right|}$ is the structure 
matrix. Since $\mathbf{M}$ is invertible, the  corollary follows. 
\end{proof}
	
\begin{remark}\label{Rem: (K2) is not a necessary condition}
Let $(\Gamma, \sigma, \operatorname{Ord}(E^\sigma), \mathbf{d})$ be  an $n$-GKM skeleton.  
Whenever the GKM skeleton supports an abstract $(n,d)$-GKM graph, then by Proposition \ref{Pro: defect and FS} its 
defect $\delta$ is greater or equal to $d$. Moreover, if $d=\delta$, then by Corollary \ref{Cor: low defect and (K2)} 
the Kernel Condition (K2) is satisfied. 
But note that if $d<\delta$, then the Kernel Condition (K2) might not be satisfied. In 
Example 
\ref{Example: (K2) not satisfied but supporting GKM graph} we consider a $3$-GKM skeleton that supports an abstract 
$(3,2)$-GKM graph but  does 
not satisfy the  Kernel Condition (K2).
\end{remark}

\begin{example}\label{Example: (K2) not satisfied but supporting GKM graph}
Let $(\Gamma, \sigma, \operatorname{Ord}(E^\sigma), \mathbf{d})$ be the $3$-GKM skeleton where  $(\Gamma, \sigma, 
\operatorname{Ord}(E^\sigma))$ is as in Example \ref{Example: tori with four vertices} and $\mathbf{d}$ is the zero 
vector. This $3$-GKM skeleton has the same structure matrix $\mathbf{A}$ as the one in Example \ref{Example: tori with 
four 
vertices}. The defect of this $3$-GKM skeleton is $3$ and a fundamental system is given by 	

\begin{align*}
&\quad\,\,\, f_1\,\,\,\, f_2\,\,\,\,\, f_3\,\,\,\,\, f_4\,\,\,\, f_5\,\,\, f_6\\
&\left( \begin{array}{rrrrrr}
0&	1&	$-1$&	0&	0&	1 \\ 
1&	0&	$-1$&	0&	1&	0\\
1&	$-1$&	0&	1&	0&	0\\
\end{array}\right).
\end{align*}	
The Kernel Condition (K1) is satisfied. But the Kernel Condition (K2) fails to be true.
Consider the $(2\times 3)$ matrix $\mathbf{M}$ given by 

\begin{align*}
\mathbf{M}=
\left( \begin{array}{rrr}
$-1$&	$1$&	$0$ \\ 
$2$&	$-1$&	$1$ \\
\end{array}\right).
\end{align*}
Let  $\operatorname{w}:E \rightarrow \Z^2\setminus\{0\}$ be the antisymmetric map given by 
\begin{align*}
\operatorname{w}(e_i)=\mathbf{M}\cdot f_i \,\,\,\, \text{for all } i=1,...,6.
\end{align*}
This map together with $\Gamma$ form an abstract $(3,2)$-GKM graph, i.e., the 
abstract GKM graph $(\Gamma, \operatorname{w})$ is supported by the $3$-GKM skeleton.
\end{example}

So far we proved that the Kernel Condition (K1) is a necessary condition for an $n$-GKM skeleton to support an abstract 
$(n,d)$-GKM graph (Corollary \ref{Cor: (K1) is a necessary condition}), and that (K2) is a necessary condition if 
$d=\delta$ (Corollary \ref{Cor: low defect and (K2)}). Now we show that  a 
sufficient condition is that  (K1) and (K2) hold. This is  part of the following proposition.

\begin{proposition}\label{Pro: (K1)+(K2) gives a unique GKM graph}
Let $(\Gamma, \sigma, \operatorname{Ord}(E^\sigma), \mathbf{d})$ be  an $n$-GKM skeleton with  a positive defect 
$\delta$ that satisfies the Kernel Conditions (K1) and (K2). Then there exists an abstract $(n,\delta)$-GKM graph 
$(\Gamma, \operatorname{w})$ that is supported by the GKM skeleton. Moreover, the GKM graph $(\Gamma, \w)$ is unique up 
to isomorphism
and projection. This means that whenever $(\Gamma, 
\operatorname{w}')$  is  an abstract $(n,d)$-GKM graph that  is also supported by this GKM skeleton, then
either $d=\delta$ and $(\Gamma, \operatorname{w})$ and $(\Gamma, \operatorname{w}')$   are isomorphic or  $d<\delta$ 
and $(\Gamma, \operatorname{w}')$ is a projection of $(\Gamma, \operatorname{w})$.
\end{proposition}
	
In order to prove this proposition, we first prove  two auxiliary lemmas.
In the following, for  vectors $f_1,\dots, f_n\in \Q^d$ and for $R=\Z$ or $R=\Q$, we denote the $R$-span
 
$$\left\lbrace \kappa_1 \cdot f_1+...+ \kappa_n \cdot f_n\,\, \vert \,\,\kappa_1,...,\kappa_n \in R \right\rbrace $$
of these vectors by 
$$\operatorname{span}_R\left\lbrace f_1,...,f_n\right\rbrace. $$	
\begin{lemma}\label{Lem: technical lemma Z span}
Let $n$ and $d$ be positive integers with $d\leq n$ and let $f_1,\dots, f_n\in \Q^d$ be vectors. Then there 
exist vectors $w_1,\dots, w_d\in \Q^d$ such that
\begin{align*}
\operatorname{span}_{\Z}\{f_1,\dots, f_n\}\, =\,\operatorname{span}_{\Z}\{w_1,\dots, w_d\}.
\end{align*}
\end{lemma}
\begin{proof}
Fix a positive integer $d$ and let $n\in \N$ such that $n \geq d$. 
We prove the claim by induction over $n-d$. The induction base, i.e., $n=d$, is 
obvious. Indeed in this case we can choose $w_i=f_i$ for all $i=1,\dots,d$. Assume that the claim holds for some 
$n$ with $n\geq d$.  Given  $f_1,\dots, f_{n+1}\in \Q^d$, by the induction assumption 
there 
exist $d$ vectors  $w_1',\dots, w_d'\in \Q^d$ such that 
\begin{align*}
\operatorname{span}_{\Z}\{f_1,\dots, f_n\}\, =\,\operatorname{span}_{\Z}\{w_1',\dots, w_d'\}.
\end{align*}
Therefore, we have 
\begin{align*}
\operatorname{span}_{\Z}\{f_1,\dots, f_{n+1}\}\, =\,\operatorname{span}_{\Z}\{w_1',\dots, w_d', f_{n+1}\}.
\end{align*}
Let $\mathbf{G}$ be the $(d\times (d+1))$-matrix given by
\begin{align*}
\mathbf{G}= \left( \begin{matrix}
\vert & &\vert& \vert  \\
w_1'   &\dots &w_d' & f_{n+1}  \\
\vert & &\vert & \vert
\end{matrix}\right) .
\end{align*}
For dimensional reasons the kernel of $\mathbf{G}$ as a map $\Q^{d+1}\rightarrow \Q^d$ is not trivial and  since this 
matrix has 
only rational entries, there exists a primitive vector $h_0\in \Z^{d+1}\setminus\{0\}$ that is an element of 
the kernel of  $\mathbf{G}$. Now let $h_1,\dots h_d\in \Z^{d+1} $ vectors such 
that  $h_0,h_1,\dots, h_d$ is a  
		$\Z$-basis 
		of $\Z^{d+1}$. So we have 
		\begin{align*}
		\operatorname{span}_{\Z}\{w_1',\dots, w_d', f_{n+1}\}=\operatorname{span}_{\Z}\{\mathbf{G}\cdot h_0, 
		\dots,\mathbf{G}\cdot 
		h_d\}. 
		\end{align*}
		We set $w_i:=\mathbf{G}\cdot h_i$ for all $i=1,..., d$. Since $\mathbf{G}\cdot h_0=0$, we conclude that
		\begin{align*}
		\operatorname{span}_{\Z}\{f_1,\dots, f_{n+1}\}\, =\operatorname{span}_{\Z}\{w_1,\dots, w_d\}.
		\end{align*} 
		This completes the induction step. 
\end{proof}
	
\begin{lemma}\label{Lemma: technical existence M}
Let $(\Gamma, \sigma, \operatorname{Ord}(E^\sigma), \mathbf{d})$ be  an $n$-GKM skeleton with  a positive defect 
$\delta$ that satisfies the Kernel Condition (K2) and let $f_1,\dots, f_{\frac{1}{2}\left| E \right| }$ be a 
fundamental system. Then $\delta \leq n$ and there exists an invertible $(\delta\times \delta)$-matrix $\mathbf{M}$ 
such 
that  $\mathbf{M}\cdot f_j\in \Z^\delta$ for all $j=1,...,\frac{1}{2}\left| E\right|$ and 
\begin{align*}
\operatorname{span}_{\Z}\{\mathbf{M}\cdot f_k\,\, \vert\,\, k\in \operatorname{IND}_v \} = \Z^\delta
\end{align*} 
for all vertices $v\in V$.
\end{lemma}
	
\begin{proof}
Note that $f_j\in \Q^\delta$ for all $j=1,...,\frac{1}{2}\left| E\right|$.  Since the defect $\delta$ is the dimension 
of the kernel of $\mathbf{A-D}$, we have that
\begin{align}\label{EQ1:Lemma: technical existence M}
\operatorname{span}_{\Q}\left\lbrace  f_k\, \vert\,  k\in \{1,\dots, \frac{1}{2}\left| E\right| \}\right\rbrace
=\Q^\delta. 
\end{align}
Next we  show that for any two vertices $v,w\in V$
\begin{align}\label{EQ2:Lemma: technical existence M}
\operatorname{span}_\Z\{ f_k\,\, \vert\,\, k\in \operatorname{IND}_v \}=\operatorname{span}_\Z\{ f_k\,\, \vert\,\, k\in 
\operatorname{IND}_w \}.
\end{align}
Note that since $\Gamma$ is connected it is enough to show this for vertices that are connected by an edge.  
Let us  fix an edge $e_j \in E^\sigma$. Since the Kernel Condition (K2) is satisfied, there exists a connection 
along $e_j$ such that for each $k\in\operatorname{IND}_{i(e_j)}$ there exists an integer $c_k$
\begin{align*}
a_{j,k}\cdot f_k+a_{j,\widetilde{\nabla}_j(k)} \cdot f_{\widetilde{\nabla}_j(k)}=c_k\cdot f_j, 
\end{align*}
where $\widetilde{\nabla}_j \colon \operatorname{IND}_{i(e_j)} \rightarrow \operatorname{IND}_{t(e_j)}$ is the 
bijection induced by the connection.  Note that $\widetilde{\nabla}_j(j)=j$. Let $k\in 
\operatorname{IND}_{i(e_j)}\setminus\{j\}$. Since $k\neq j$ and $i(e_k)=i(e_j)$ or $t(e_k)=i(e_j)$, 
by the definition of the structure matrix we have $a_{j,k}=\pm 1$. Moreover, since $\widetilde{\nabla}_j(k)\neq j$ and 
$\widetilde{\nabla}_j(k)\in \operatorname{IND}_{t(e_j)}$ we also have $a_{j,\tilde{\nabla}_j(k)}=\pm 1$. We conclude 
that \eqref{EQ2:Lemma: technical existence M} holds for vertices that are connected by an edge.\\
By combining \eqref{EQ1:Lemma: technical existence M} and \eqref{EQ2:Lemma: technical existence M} we conclude that
\begin{align}\label{EQ3:Lemma: technical existence M}
\operatorname{span}_{\Q}\{ f_k\,\, \vert\,\, k\in \operatorname{IND}_v \}=\Q^\delta
\end{align}
for all $v\in V$. Moreover, since the graph $\Gamma$ is $n$-valent, the cardinality of $\operatorname{IND}_v$ is $n$.
Therefore we have that $\delta \leq n$.
Let us fix a vertex $v_0\in V$. By Lemma \ref{Lem: technical lemma Z span} there exist $\delta$ vectors 
$w_1,\dots,w_\delta$
such that
\begin{align}\label{EQ4:Lemma: technical existence M}
\operatorname{span}_{\Z}\{ f_k\,\, \vert\,\, k\in \operatorname{IND}_{v_0} \}=\operatorname{span}_{\Z}\{ w_1,\dots, 
w_\delta \}.
\end{align}
Since \eqref{EQ3:Lemma: technical existence M} holds, the $\Q$-span of the vectors $w_1,\dots,w_\delta$ is equal to 
$\Q^d$.
Therefore,
\begin{align*}
\mathbf{M}= \left( \begin{matrix}
\vert & &\vert \\
w_1   &\dots &w_{\delta}   \\
\vert & &\vert
\end{matrix}\right) ^{-1}
\end{align*}
is a well defined  $(\delta \times \delta)$-matrix. By using \eqref{EQ2:Lemma: technical existence M} and 
that the equation \eqref{EQ4:Lemma: technical existence M} is satisfied for any two vertices $v,w\in V$, 
it is clear that the matrix $\mathbf{M}$ satisfies the desired properties.
\end{proof}

\begin{proof}[Proof of Proposition \ref{Pro: (K1)+(K2) gives a unique GKM graph}]
Let $(\Gamma, \sigma, \operatorname{Ord}(E^\sigma), \mathbf{d})$ be an $n$-GKM skeleton with positive defect 
$\delta$ that satisfies the Kernel Conditions (K1) and (K2). Let $f_1,\dots, f_{\frac{1}{2}\left| E \right| }$ be a 
fundamental system. 
By Proposition \ref{Pro: defect and FS}, in order to prove  that the GKM skeleton supports an abstract 
$(n,\delta)$-GKM   we 
need to show  that there exists a 
$(\delta\times \delta)$-matrix $\mathbf{M}$ such that the graph $\Gamma$ together with the antisymmetric
map $\operatorname{w}: E \rightarrow \Z^\delta$ given by
\begin{align}\label{EQ1:Pro: (K1)+(K2) gives a unique GKM graph}
\operatorname{w}(e_j)=\mathbf{M}\cdot f_j \text{ for all }j=1,\dots, \frac{1}{2}\left| E\right| \quad 
\end{align}
form an abstract GKM graph.\\
By Lemma \ref{Lemma: technical existence M}, since the GKM skeleton has positive defect 
$\delta$ and satisfies the Kernel Condition (K2), we 
can pick  an invertible $(\delta\times 
\delta)$-matrix $\mathbf{M}$ such 
that  $\mathbf{M}\cdot f_j\in \Z^\delta $ for all $j=1,...,\frac{1}{2}\left| E\right|$ and 
\begin{align}\label{EQ2:Pro: (K1)+(K2) gives a unique GKM graph}
\operatorname{span}_\Z \{\mathbf{M}\cdot f_k\,\, \vert\,\, k\in \operatorname{IND}_v \} = \Z^\delta
\end{align} 
for all vertices $v\in V$.
Note that Kernel condition (K1) implies that for all $j$,  the vector $f_j$ is not zero hence so is $\mathbf{M}\cdot f_j$. So for this selected matrix $\mathbf{M}$, let $\w \colon E \to \Z^d \setminus \{0\} $ be the antisymmetric map as in 
\eqref{EQ1:Pro: (K1)+(K2) gives a 
unique GKM graph}. We need to show that $\w$ satisfies conditions $(i)$, $(ii)$ and $(iii)$ of Definition 
\ref{Def: abstract GKM graph}. 

Since $\w$ is antisymmetric, we have for all vertices $v\in V$
\begin{align*}
\operatorname{span}_\Z\{\w(e)\,\, \vert\,\, e\in E_v^i \}  =
\operatorname{span}_\Z\{\mathbf{M}\cdot f_k\,\, \vert\,\, k\in \operatorname{IND}_v \}. 
\end{align*}
By \eqref{EQ2:Pro: (K1)+(K2) gives a unique GKM graph}
\begin{align*}
\operatorname{span}_\Z\{\w(e)\,\, \vert\,\, e\in E_v^i \}  =
\Z^\delta
\end{align*}
for all $v\in V$, i.e., $\w$ satisfies condition $(i)$ of Definition \ref{Def: abstract GKM graph}. 

Since $\skeleton$ satisfies the Kernel Condition (K1) and since $\mathbf{M}$ is an invertible matrix,
we have that for any two indices $j,k\in \{1,\dots, \frac{1}{2}\left| E\right| \}$ with  $a_{j,k}=1$ or 
$a_{j,k}=-1$ the vectors $\w(e_j)$ and $\w(e_k)$ are linearly independent, 
where $\mathbf{A}=(a_{j,k})_{j,k=1,...,\frac{1}{2}\left| E\right|}$ is the structure matrix. 
Hence, by Lemma \ref{Lem: linarly inde. Cond. for skeletons}, $\w$ satisfies condition $(ii)$ of Definition \ref{Def: 
abstract GKM graph}.

Since $\skeleton$ satisfies the Kernel Condition (K2), for each $e_j\in E^\sigma$  there exists 
a connection $\nabla_{e_j}$ along $e_j$ such that for each $k\in\operatorname{IND}_{i(e_j)}$
\begin{align*}
a_{j,k}\cdot f_k+a_{j,\widetilde{\nabla}_j(k)} \cdot f_{\widetilde{\nabla}_j(k)}\quad \text{is an integer multiple 
of}\, f_j, 
\end{align*}
where $\widetilde{\nabla}_j \colon \operatorname{IND}_{i(e_j)} \rightarrow \operatorname{IND}_{t(e_j)}$ is the 
bijection induced by   $\nabla_{e_j}$.  So we have 
\begin{align*}
a_{j,k}\cdot \w(e_k)+a_{j,\widetilde{\nabla}_j(k)} \cdot \w(e_{\widetilde{\nabla}_j(k)})\quad \text{is an integer 
multiple 
	of}\,\, \w(e_j), 
\end{align*}
for all $k\in\operatorname{IND}_{i(e_j)}$. Hence, by Lemma \ref{Lem: connections for skeltons}, $\w$ satisfies 
condition $(iii)$ of Definition \ref{Def: abstract GKM graph}.
So $(\Gamma, \w)$ is an abstract $(n,\delta)$-GKM graph that is supported by $\skeleton$.\\
 In order to prove that 
$(\Gamma, \w)$ is unique (up to isomorphism and projection), note that $\w(e_1),\dots, \w(e_{\frac{1}{2}\left| 
E\right| })$ is also a fundamental system. Let $(\Gamma, \w')$ be an abstract $(n,d)$-GKM graph that is supported by 
$\skeleton$. By Proposition \ref{Pro: defect and FS} we have $d\leq \delta$ and there exists a $(d\times \delta)$-
matrix $\mathbf{M'}$ such that $\w'(e_i)= \mathbf{M'}\cdot \w(e_i)$ for all $i=1,\dots, \frac{1}{2}\left| E\right|$.
Since the $\Z$-span of  $\w(e_1),\dots, \w(e_{\frac{1}{2}\left| E\right|})$ resp. $\w'(e_1),\dots, 
\w'(e_{\frac{1}{2}\left| E\right|})$ is equal to $\Z^\delta$ resp. $\Z^d$, $\mathbf{M'}$ induces a linear surjection
$\theta: \Z^\delta \rightarrow \Z^d$ such that $\theta(\w(e))=\w'(e)$ for all $e\in E$. Therefore, if $d<\delta$, then
$(\Gamma,\w')$ is a projection of $(\Gamma, \w)$. If $d=\delta$, then $\theta$ is a linear isomorphism and the pair 
$(F,\theta)$, where $F: V \rightarrow V$ is the identity map, is an isomorphism between the abstract GKM graphs 
$(\Gamma, \w)$
and $(\Gamma, \w')$.
\end{proof}

\end{subsubsection}

\begin{subsubsection}{If the Kernel Condition (K2) is not satisfied} \label{subsec:ProjectionTest}
	
Now we consider the case that the Kernel Condition (K2) is not satisfied. We focus on $3$-GKM skeletons
whose defects are equal to $3$. We first look at the general case.\newline

Let $(\Gamma, \sigma, \operatorname{Ord}(E^\sigma), \mathbf{d})$ be an $n$-GKM skeleton with a positive  defect 
$\delta$ that satisfies the Kernel Condition (K1) but not (K2). By Corollary \ref{Cor: low defect and (K2)}, this GKM 
skeleton can not support an abstract $(n,\delta)$-GKM graph, but it may support an abstract $(n,d)$-GKM graph for some 
$d<\delta$. Let $f_1,...,f_{\frac{1}{2}\left| E\right| }$ be a fundamental system. By Proposition \ref{Pro: defect and 
FS}, if the GKM skeleton supports an 
abstract $(n,d)$-GKM graph $(\Gamma, \w)$, then there exist a $(d\times\delta)$-matrix 
$\mathbf{M}$ of full rank  such that 
\begin{align*}
\mathbf{M}\cdot f_i= \operatorname{w}(e_i)
\end{align*}
for all $i=1,\dots, \frac{1}{2}\left| E\right| $. In order to find out if such a weight transformation matrix 
$\mathbf{M}$ exists, we consider 
the following ansatz. Namely, since the Kernel Condition (K2) is not satisfied, there exists at least one edge $e_j\in 
E^\sigma$ where (K2) fails to be true. This means that for any connection $\nabla_{e_j}$ along $e_j$ there exists 
an index $k\in \operatorname{IND}_{i(e_j)}$ such that 
\begin{align*}
a_{j,k}\cdot f_k+a_{j,\widetilde{\nabla}_j(k)} \cdot  
f_{\widetilde{\nabla}_j(k)}\,\, \text{is \textbf{not} an integer multiple of}\,\,   f_j.
\end{align*}

\begin{definition}\label{Def: fails by rational}
A connection $\nabla_{e_j}$ along an edge $e_j\in E^\sigma$ 
\begin{center}
\textbf{fails (K2) by a rational number}
\end{center} 
if for any  fundamental system  $f_1,...,f_{\frac{1}{2}\left| E\right|}$
there exist an index $k\in \operatorname{IND}_{i(e_j)}$ and a rational number $\nu \in \Q \setminus \Z$ such that
\begin{align*}
a_{j,k}\cdot f_k+a_{j,\widetilde{\nabla}_j(k)} \cdot  
f_{\widetilde{\nabla}_j(k)}\,=\,\nu \cdot  f_j, 
\end{align*}
where $\widetilde{\nabla}_j \colon \operatorname{IND}_{i(e_j)} \rightarrow \operatorname{IND}_{t(e_j)}$ is the 
bijection induced by   $\nabla_{e_j}$.
\end{definition}

\begin{lemma}\label{Lem: fails rational}
Let $e_j\in E^\sigma$ be an edge and let $\nabla_{e_j}$ a connection along the  edge $e_j$ that fails  (K2) 
by a rational number. 
Then there exists no abstract GKM graph $(\Gamma, \w)$ that is supported by this GKM skeleton such that the connection
$\nabla_{e_j}$ is compatible. 
\end{lemma}

\begin{proof}
We prove this lemma by contradiction. Assume by negation that $(\Gamma, \w)$ is an abstract $(n,d)$-GKM graph that is 
supported by 
$\skeleton$ and that the connection $\nabla_{e_j}$  along the edge $e_j$ is compatible.
 Since $\nabla_{e_j}$ is a 
compatible connection we have that for all $k\in \operatorname{IND}_{i(e_j)}$  
\begin{align}\label{EQ1:Lem: fails rational}
a_{j,k}\cdot \w (e_k)+a_{j,\widetilde{\nabla}_j(k)} \cdot  
\w(e_{\widetilde{\nabla}_j(k)})\quad \text{is an integer multiple of } \,\, \w(e_j), 
\end{align}
where $\widetilde{\nabla}_j \colon \operatorname{IND}_{i(e_j)} \rightarrow \operatorname{IND}_{t(e_j)}$ is the 
bijection induced by   $\nabla_{e_j}$.  
Let $f_1,...,f_{\frac{1}{2}\left| E\right|}$ be a fundamental system of $\skeleton$.
Since $\nabla_{e_j}$ fails (K2) by a rational number there exists at least one index $k\in 
\operatorname{IND}_{i(e_j)}$ such that 
\begin{align*}
a_{j,k}\cdot f_k+a_{j,\widetilde{\nabla}_j(k)} \cdot  
f_{\widetilde{\nabla}_j(k)}\,=\,\nu \cdot  f_j, 
\end{align*}
where $\nu \in \Q\setminus \Z$. By Proposition \ref{Pro: defect and FS}, there exists a matrix $\mathbf{M}$ such that 
$\mathbf{M}\cdot f_i =\w_i$ for all  $i=1,\dots, \frac{1}{2}\left| E\right|$. So we have that 
\begin{align*}
a_{j,k}\cdot \w(e_k)+a_{j,\widetilde{\nabla}_j(k)} \cdot  
\w(e_{\widetilde{\nabla}_j(k)})\,=\,\nu \cdot  \w(e_j). 
\end{align*}
Since \eqref{EQ1:Lem: fails rational} holds and $\nu \notin \Z$, we must have $\w(e_j)=0$, which is a contradiction.
\end{proof}


Since in an abstract GKM graph  $(\Gamma, \w)$, for each edge $e\in E$ there exists a compatible connection along 
$e$, a direct consequence of Lemma \ref{Lem: fails rational} is the following corollary.

\begin{corollary}\label{Cor: fails rational}
Let $\skeleton$ be an $n$-GKM skeleton with a positive defect $\delta$. Let $e_j\in E^\sigma$ be an edge 
such that any connection along the  edge $e_j$ fails (K2) by a rational number. 
Then there exists no abstract GKM graph that is supported by the GKM skeleton. 
\end{corollary}

In the following lemma we give sufficient criteria for a $3$-GKM skeleton that ensure that the GKM skeleton does not 
support an 
abstract GKM graph.  
\begin{lemma}\label{Lemma: Projection Test}
Let $\skeleton$ be a $3$-GKM skeleton whose defect is equal to $3$ that satisfies the Kernel Condition (K1), but 
does not satisfy the Kernel Condition (K2).
Let $\mathbf{A}=(a_{j,k})_{j,k=1,\dots, \frac{1}{2}\left| E\right|}$ be the structure matrix and let 
$f_1,\dots,f_{\frac{1}{2}\left| E\right|}$ be a fundamental system. Let $e_{j_1}, e_{j_2}\in E^\sigma$ be two different 
edges such 
that the following hold. 
\begin{itemize}
\item[$(i)$] There exists exactly one 
connection $\nabla_{e_{j_1}}$ 
along $e_{j_1}$  and 
exactly one connection $\nabla_{e_{j_2}}$ along $e_{j_2}$ that do not fail  (K2) by a rational number.
\item[$(ii)$] There exist indices 
\begin{align*}
k_1\in \operatorname{IND}_{i(e_{j_1})}\setminus \left\lbrace {j_1}\right\rbrace \quad \text{and}\quad k_2\in 
\operatorname{IND}_{i(e_{j_2})}\setminus \left\lbrace {j_2}\right\rbrace
\end{align*}
such that at least  one of the following four pairs of vectors is linearly dependent
\begin{align}\label{EQ3: Lemma: Projection Test}
\left( f_{j_1}, h_2 \right),\quad \left( f_{j_1}, f_{j_2} \right), \quad \left( h_1, f_{j_2} \right)\quad \text{or} 
\quad \left( 
h_1, h_2 \right),
\end{align}
where for $i=1,2$
\begin{align*}
h_i:=a_{j_{i},k_i} \cdot f_{k_i} +a_{j_{i},\widetilde{\nabla}_{{j_i}}(k_i) }\cdot 
f_{\widetilde{\nabla}_{j_i}(k_i)}
\end{align*}
and $\widetilde{\nabla}_{j_i}: \operatorname{IND}_{i(e_i)}\rightarrow \operatorname{IND}_{t(e_i)}$ is the bijection 
induced by $\nabla_{e_{j_i}}$.
\item[$(iii)$] Both vectors $h_1$ and $h_2$ are non-zero.
\item[$(iv)$] The matrix of size $(3\times 4)$ given by
\begin{align*}
 \left( \begin{matrix}
\vert & \vert & \vert &\vert \\
f_{j_1}   & f_{j_2} & h_1& h_2   \\
\vert &\vert&\vert &\vert
\end{matrix}\right) 
\end{align*}
has matrix rank $3$.
\end{itemize}
Then the GKM skeleton $\skeleton$ does not support an abstract GKM graph. 
\end{lemma}

\begin{proof}
Assume by negation that  $\skeleton$ supports an abstract $(3,d)$-GKM graph $(\Gamma, \w)$. 
Since the Kernel Condition (K2) is not satisfied we have $d<3$ and because the graph $\Gamma$ is $3$-valent we have 
that $d\geq 2$ (see Remark \ref{Rem: relation n, d abstract GKM graph}). So  $(\Gamma, \w)$ is an abstract $(3,2)$-GKM 
graph.
By Proposition \ref{Pro: defect and FS} there 
exits a $(2\times 3)$-matrix $\mathbf{M}$ of rank $2$ such that
\begin{align}\label{EQ1: Lemma: Projection Test}
\mathbf{M} \cdot f_i =\w(e_i)
\end{align}
for all $i=1,\dots, \frac{1}{2}\left| E\right|$. Now we show that assumptions $(i)$, $(ii)$, $(iii)$ and $(iv)$ imply 
that the rank of $\mathbf{M}$ is strictly smaller than $2$, which is a contradiction.

Consider the edge $e_{j_1}$. At least one connection along this edge must be compatible with the map $\w$. By Lemma 
\ref{Lem: fails rational}
this edge does not fail (K2) by a rational number. Since by $(i)$ there exists exactly one connection along 
$e_{j_1}$ that does not fail (K2) by a rational number, namely $\nabla_{e_{j_1}}$, we have that 
$\nabla_{e_{j_1}}$ is a compatible connection. This implies that for all $k\in \operatorname{IND}_{i(e_{j_1})}$
\begin{align}\label{EQ2: Lemma: Projection Test}
a_{j_{1},k} \cdot \w(e_{k}) +a_{j_{1},\widetilde{\nabla}_{{j_1}}(k) }\cdot 
\w(e_{\widetilde{\nabla}_{j_1}(k)})\quad \text{is an integer multiple of} \quad \w(e_{j_1}).
\end{align}
Let $k_1\in \operatorname{IND}_{i(e_{j_1})}$ be the index and $h_1$ be the vector as in $(ii)$. Since 
for all $i=1,\dots, \frac{1}{2}\left| E\right|$ \eqref{EQ1: 
Lemma: Projection Test} holds and  for all $k\in \operatorname{IND}_{i(e_{j_1})}$ \eqref{EQ2: Lemma: Projection Test} 
holds, we conclude that there exists an integer $A_1\in \Z$ such that 
$$\mathbf{M}\cdot h_1= A_1 \cdot\mathbf{M}\cdot f_{j_1}.$$
Now consider the edge $e_{j_2}$. For the same reason as above, there exists an integer $A_2\in \Z$ such that 
$$\mathbf{M}\cdot h_2= A_2 \cdot\mathbf{M}\cdot f_{j_2}.$$

Since the kernel Condition (K1) is satisfied, the vectors $f_{j_1}$ and $f_{j_2}$ are both non-zero. By $(iii)$ both 
vector $h_1$ and $h_2$ are non-zero. Hence, since by $(iii)$ one vector pair as in \eqref{EQ3: Lemma: Projection Test} 
is linearly dependent, we have that the matrix $\mathbf{M}$ maps the set of the vectors $f_{j_1}, f_{j_2}, h_1,$ and 
$h_2$ onto a set of pairwise linearly 
dependent vectors.
By $(iv)$  the $(3\times4)$,  whose column vectors are $f_{j_1}, f_{j_2}, h_1,$ and $h_2$ has matrix rank $3$.
So we conclude that the matrix rank of $\mathbf{M}$ is strictly smaller than $2$. This is a contradiction. 
\end{proof}

In the following example we apply Lemma \ref{Lemma: Projection Test}.

\begin{example}\label{Beispiel}

Consider the  $3$-GKM skeleton as in Figure \ref{Figure: GKMSkeletonGraph6one}. 

\begin{figure}[h]
	\begin{center}
		\includegraphics[width=14cm]{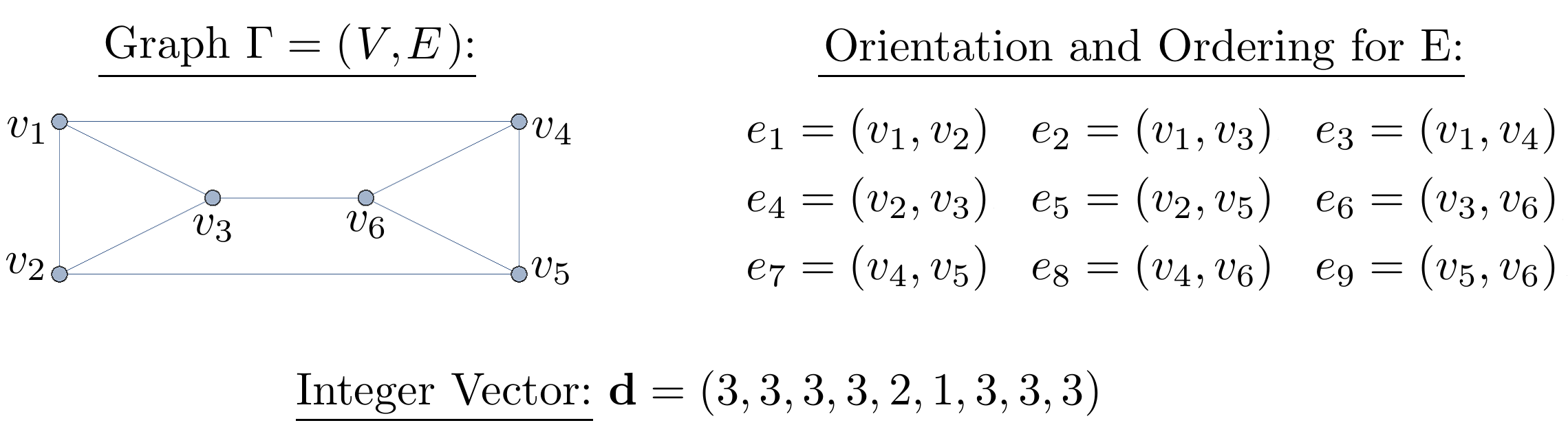}
		\caption{$3$-GKM Skeleton of Example  \ref{Beispiel} }
		\label{Figure: GKMSkeletonGraph6one}
	\end{center}
\end{figure}

The structure matrix is
\begin{align*}
\mathbf{A}=\left( \begin{array}{rrrrrrrrr}
$2$&	$1$&	$1$&	$-1$	&$-1$&	$0$&	$0$&	$0$&	$0$ \\ 
$1$&	$2$&	$1$&	$1$&	$0$&	$-1$&	$0$&	$0$&	$0$ \\
$1$&	$1$&	$2$&	$0$&	$0$&	$0$&	$-1$&	$-1$&	$0$\\
$-1$&	$1$&$0$&	$2$&	$1$&	$-1$&	$0$&	$0$&	$0$\\
$-1$&	$0$&	$0$&	$1$&	$2$&	$0$&	$1$&$	0$&	$-1$\\
$0$&	$-1$&	$0$&	$-1$&	$0$&	$2$&	$0$&	$1$&	$1$\\
$0$&	$0$&	$-1$&	$0$&	$1$&	$0$&	$2$&	$1$&	$-1$\\
$0$&	$0$&	$0$&	$0$&	$-1$&	$1$&	$-1$&	$1$&	$2$\\
\end{array}\right).
\end{align*}
The defect of this $3$-GKM skeleton is $3$ and a fundamental system is given by

\begin{align*}
&\quad\,\,\, f_1\,\,\,\, f_2\,\,\,\,\, f_3\,\,\,\, f_4\,\,\, f_5\,\,\, f_6\,\,\, f_7\,\,\, f_8\,\,\,f_9\\
&\left( \begin{array}{rrrrrrrrr}
$-1$&	$0$&	$0$&	$1$	&$0$&	$0$&	$-1$&	$0$&	$1$ \\ 
$1$&	$1$&	$0$&	$0$&	$0$&	$0$&	$1$&	$1$&	$0$ \\
$1$&	$2$&	$3$&	$1$&	$3$&	$3$&	$0$&	$0$&	$0$\\
\end{array}\right).
\end{align*}

The Kernel Condition (K1) is satisfied. But the Kernel Condition (K2) is not satisfied, it fails to be true at 
the edges $e_3$, $e_5$ and $e_6$. 
Now we use Lemma \ref{Lemma: Projection Test} to show that this $3$-GKM skeleton does not support an abstract GKM-graph.
First consider the edge   $e_3=(v_1,v_4)$.
There exist exactly two connections along $e_3$, namely 
\begin{align*}
\nabla_{e_{3}}:\,\,\left\lbrace \begin{array}{r} (v_1,v_2) \mapsto (v_4,v_6) \\ (v_1,v_3) \mapsto (v_4,v_5)  \\  
(v_1,v_4) \mapsto 
(v_4,v_1) 
\end{array}\right\rbrace 
\quad \text{and} \quad \nabla'_{e_{3}}:\,\,
\left\lbrace \begin{array}{r} (v_1,v_2) \mapsto (v_4,v_5)  \\ (v_1,v_3) \mapsto (v_4,v_6)  \\  (v_1,v_4) \mapsto 
(v_4,v_1) 
\end{array}\right\rbrace. 
\end{align*}

We have
\begin{align*}
\operatorname{IND}_{v_1}=\{1,2,3\} \quad \text{and} \quad \operatorname{IND}_{v_4}=\{3,7,8\}
\end{align*} 
and the maps $\operatorname{IND}_{v_1}\rightarrow \operatorname{IND}_{v_4}$ induced by the two connections 
along $e_3$
are given by
\begin{align}
\widetilde{\nabla}_3:\,\,\left\lbrace \begin{array}{r} 1 \mapsto 8  \\ 2 \mapsto 7  \\  3 \mapsto 3 
\end{array}\right\rbrace 
\quad \text{and} \quad
\widetilde{\nabla}'_3:\,\,\left\lbrace \begin{array}{r} 1 \mapsto 7  \\ 2 \mapsto 8  \\  3 \mapsto 3 
\end{array}\right\rbrace. 
\end{align}

Note that 
\begin{align*}
&a_{3,1}\cdot f_1+a_{3,\widetilde{\nabla}_3(1)}\cdot f_{\widetilde{\nabla}_3(1)}
=a_{3,1}\cdot f_1+a_{3,8}\cdot f_8=(-1,0,1)^\intercal \\
&a_{3,2}\cdot f_2+a_{3,\widetilde{\nabla}_3(2)}\cdot f_{\widetilde{\nabla}_3(2)}=a_{3,1}\cdot f_2+a_{3,7}\cdot 
f_7=(1,0,2)^\intercal\\
&a_{3,3}\cdot f_3+a_{3,\widetilde{\nabla}_3(3)}\cdot f_{\widetilde{\nabla}_3(3)}=a_{3,3}\cdot f_1+a_{3,3}\cdot 
f_3=(0,0,12)^\intercal.
\end{align*}
and since $f_3=(0,0,3)^\intercal$, the connection $\nabla_{e_{3}}$ does not fail (K2) by a rational number.
But since 
\begin{align*}
a_{3,1}\cdot f_1+a_{3,\widetilde{\nabla}'_3(1)}\cdot f_{\widetilde{\nabla}'_3(1)}=a_{3,1}\cdot f_1+a_{3,7}\cdot 
f_7=\frac{1}{3}\cdot f_3.
\end{align*}
the connection $\nabla_{e_{3}}'$ fails (K2) by a rational number. Hence, there exists only one connection 
along $e_3$ that does not fail (K2) by a rational number, namely $\nabla_{e_{3}}$.
Now consider the edge $e_5=(v_2, v_6)$. As above,  we can show that there exists exactly one connection 
$\nabla_{e_{5}}$ along $e_5$ that does not fail (K2) by a rational number. We have
\begin{align*}
\operatorname{IND}_{v_2}=\{1,4,5\} \quad \text{and} \quad \operatorname{IND}_{v_5}=\{5,7,9\}
\end{align*} 
and the map $\operatorname{IND}_{v_2}\rightarrow \operatorname{IND}_{v_5}$ induced by the connection $\nabla_{e_{5}}$
is given by
\begin{align*}
\widetilde{\nabla}_5:\,\,\left\lbrace \begin{array}{r} 1 \mapsto 9  \\ 4 \mapsto 7  \\  5 \mapsto 5
\end{array}\right\rbrace.
 \end{align*} 
Now fix the indices
\begin{align*}
 1 \in \operatorname{IND}_{v_1}\setminus\left\lbrace 3\right\rbrace  \quad \text{and} \quad 5
\in \operatorname{IND}_{v_2}\setminus\left\lbrace 5\right\rbrace
\end{align*}
and let 
\begin{align*}
&h_1=a_{3,1}f_1 + a_{3,8}f_8=(-1,0,1)^\intercal\\
&h_2=a_{5,1}f_5 + a_{5,9}f_9=(0,-1,-1)^\intercal.
\end{align*}
These both vector are non-zero. Moreover, we have $f_3=f_5$ and the matrix 
\begin{align*}
\left( \begin{matrix}
\vert & \vert & \vert &\vert \\
f_{3}   & f_{5} & h_1& h_2   \\
\vert &\vert&\vert &\vert
\end{matrix}\right) 
\end{align*}
has rank $3$. Hence, by Lemma \ref{Lemma: Projection Test} the $3$-GKM skeleton does not support an abstract GKM 
graph. 	
\end{example}

We formulate Lemma \ref{Lemma: Projection Test}  as a test that gives a necessary condition for a $3$-GKM skeleton 
with defect $3$ that does not 
satisfy the Kernel Condition (K2) to support an abstract $(3,2)$-GKM graph. We call this test the \textbf{Projection 
Test}. So let  $(\Gamma, \sigma, 
\operatorname{Ord}(E^\sigma), \mathbf{d})$ be such a $3$-GKM skeleton and let $f_1,\dots, f_{\frac{1}{2}\left| E\right| 
}$ be a fundamental system.

\begin{center}
\underline{The Projection Test}\\
\end{center}
\begin{itemize}
 \item[1.] First let $\operatorname{ListKern}=\left\lbrace \right\rbrace $ be an empty list. 
\item[2.] Start with $j=1$ and do the following as long as $j\leq\frac{1}{2}\left| E\right| $.
Check if the Kernel Condition (K2) at $e_j$ is satisfied. 
\begin{itemize}
\item[$\bullet$] If (K2) is satisfied at $e_j$,  set $j\rightarrow j+1$.
\item[$\bullet$] If (K2) is not  satisfied at $e_j$, then let $k_1\in \operatorname{IND}_{i(e_j)}\setminus\{j\}$
and $k_2,k_3\in \operatorname{IND}_{t(e_j)}\setminus \{j\}$.\\
Check if the conditions
\begin{align*}
&(*)\quad \quad\,\, 	a_{j,k_1}\cdot f_{k_1}+	a_{j,k_2}\cdot f_{k_2} =\nu_1\cdot f_j \,\,\,\, \text{for some}\,\, 
\nu_1 
\in 
\Q\setminus\Z\quad \text{and} \\
 &(**)\quad \quad 	a_{j,k_1}\cdot f_{k_1}+	a_{j,k_3}\cdot f_{k_3} =\nu_2\cdot f_j \,\,\,\, \text{for some}\,\, \nu_2 
 \in 
 \Q\setminus\Z
\end{align*}
are satisfied. 
\begin{itemize}
\item[$\cdot$] If $(*)$ and $(**)$ are true, set $j\rightarrow \frac{1}{2}\left| E\right| +2$.
\item[$\cdot$] If $(*)$ and $(**)$ are false, set $j\rightarrow j+1$.
\item[$\cdot$] If $(*)$ is true and $(**)$ is false, then  add the $(3\times 2)$ matrix
\begin{align*}
\left( \begin{matrix}
\vert  &\vert \\
f_j &a_{j,k_1}\cdot f_{k_1}+	a_{j,k_3}\cdot f_{k_3}  \\
\vert&\vert
\end{matrix}\right)
\end{align*}
to the list $\operatorname{ListKern}$ and set $j\rightarrow j+1$.
\item[$\cdot$] If $(*)$ is false and $(**)$ is true, then  add the $(3\times 2)$ matrix
\begin{align*}
\left( \begin{matrix}
\vert  &\vert \\
f_j &a_{j,k_1}\cdot f_{k_1}+	a_{j,k_2}\cdot f_{k_2}  \\
\vert&\vert
\end{matrix}\right)
\end{align*}
to the list $\operatorname{ListKern}$ and set $j\rightarrow j+1$.
\end{itemize}
\end{itemize}
\item[3.] 
\begin{itemize}
\item[$\bullet$] If $j=\frac{1}{2}\left| E\right|+2 $, then
$$\text{Output: The Projection Test is not satisfied.}$$
\item[$\bullet$] If $j=\frac{1}{2}\left| E\right|+1 $ and if the list $\operatorname{ListKern}$ is empty or contains 
only one element, then
$$\text{Output: The Projection Test makes no statement.}$$
\item[$\bullet$] If $j=\frac{1}{2}\left| E\right|+1 $ and if the list $\operatorname{ListKern}$ contains at least two 
elements, then check  for any two elements 
\begin{align*}
\left( \begin{matrix}
  \vert &\vert \\
 \alpha_1 & \beta_1  \\ 
\vert &\vert
\end{matrix}\right) 
\quad \text{and} \quad
\left( \begin{matrix}
\vert &\vert \\
\alpha_2 & \beta_2  \\ 
\vert &\vert
\end{matrix}\right)
\end{align*}
 of $\operatorname{ListKern}$ if one the vector pairs $(\alpha_1,\alpha_2)$, $(\alpha_1,\beta_2)$, $(\beta_1,\alpha_2)$,
 if $\beta_1$ and $\beta_2$ are both non-zero and $(\beta_1,\beta_2)$ is linearly dependent and if the matrix 

\begin{align*}
\left( \begin{matrix}
\vert & \vert & \vert &\vert \\
\alpha_1   & \alpha_1 & \beta_1& \beta_2   \\
\vert &\vert&\vert &\vert
\end{matrix}\right) 
\end{align*}
has rank $3$. If this holds, then 
$$\text{Output: The GKM skeleton does not support an abstract GKM graph.}$$ 
otherwise
$$\text{Output: The Projection Test makes no statement.}$$
\end{itemize}
\end{itemize}

\end{subsubsection}

\end{subsection}

\begin{subsection}{Positive Hamiltonian Abstract $(3,d)$-GKM Graphs}

In this subsection we formulate  necessary conditions for an abstract $(3,d)$-GKM graph $(\Gamma, \w)$ to 
be realizable by  a Hamiltonian GKM space $\tham$ of dimension six, i.e., the isomorphism class 
of  the abstract GKM graph  $(\Gamma, \w)$ is in the image of the isomorphism class of the  GKM graph of $\tham$ under the 
map $\mathcal{L}_{3,d}$ of 
\eqref{EQ: Lnd}.
\begin{definition}\label{Def: 24-Rule abstract GKM graphs}
Let $(\Gamma, \w)$ be an abstract $(3,d)$-GKM graph.
We say that $(\Gamma, \w)$ satisfies the $\mathbf{24}$-\textbf{Rule} if for any orientation
$\sigma$ of the edge set, we have 
\begin{align*}
\sum_{e\in E^\sigma} \mathcal{C}_1(e)=24.
\end{align*}
Here $\mathcal{C}_1:E \rightarrow \Z$ is the first Chern class map of $(\Gamma, \w)$.
\end{definition}

\begin{definition}\label{Def: positive abtract GKM graphs}
An abstract $(n,d)$-GKM graph $(\Gamma, \w)$ is called positive if $\mathcal{C}_1(e)>0$ for all
$e\in E$. 
\end{definition}
We deduce necessary conditions for an abstract $(3,d)$-GKM to be realizable by a 
(positive) Hamiltonian GKM space.
\begin{corollary}\label{Cor: 24-rule and V leq 16}
Let $(\Gamma, \w)$ be an abstract $(3,d)$-GKM graph.
\begin{itemize}
\item[(i)] If $(\Gamma, \w)$ is Hamiltonian, i.e., it is realizable by  a Hamiltonian GKM space of dimension six, then 
$(\Gamma, \w)$ satisfies the $24$-Rule.
\item[(ii)] If $(\Gamma, \w)$  is realizable by a positive  Hamiltonian GKM space of dimension six, then $(\Gamma, 
\w)$ is positive. In particular,  the 
number of vertices of $\Gamma$ is at most $16$.
\end{itemize}
\end{corollary}
\begin{proof}
\begin{itemize}
\item[(i)] This follows directly form Corollary \ref{Cor: Ingerdients for proof} $(ii)$ and Remark \ref{Rem: Ham vs 
abstract first Chern class map}.
\item[(ii)] That $(\Gamma, \w)$ is positive follows from Remark \ref{Rem: Ham vs abstract first Chern class map}.
That the number of vertices is at most $16$ follows from Corollary \ref{Cor: Ingerdients for proof} $(iii)$.
\end{itemize}
\end{proof}

We note that Definitions \ref{Def: 24-Rule abstract GKM graphs} and \ref{Def: positive abtract GKM graphs} 
 extend naturally to GKM skeletons.

\begin{definition}\label{Def: GKMSkeleton positive and 24-Rule}
Let $\skeleton$ be an $n$-GKM skeleton. The skeleton is called \textbf{positive} if $d_j>0$ for all $j=1,\dots, 
\frac{1}{2}\left| E\right| $, where $d_j$ is the $j$-th entry of the integer vector $\mathbf{d}$.
If $n=3$, then  $\skeleton$ satisfies the\textbf{ $24$-Rule} if 
$$d_1+...+d_{\frac{1}{2}\left| E\right|}=24.$$
\end{definition}

The following is a direct consequence of Corollary \ref{Cor: 24-rule and V leq 16}.

\begin{corollary}\label{Cor: 24Rule Skeleton}
Let $\skeleton$ be a $3$-GKM skeleton that supports a positive Hamiltonian abstract $(3,d)$-GKM graph. Then the 
skeleton is positive and satisfies the $24$-Rule. In particular, the number of the vertices of the graph $\Gamma$ is at 
most $16$. 
\end{corollary}

By Corollary \ref{Cor: 24-rule and V leq 16},  necessary  conditions for an abstract 
$(3,d)$-GKM graph to be 
realizable by a positive Hamiltonian GKM space of dimension six are the 24-Rule and the positive condition. These conditions are not sufficient, as we show in Example \ref{Beispie}.
To show this, we use the existence of the special Kirwan class. In the following, we extend the definition of a generic 
vector and an ascending path to abstract GKM graphs. We use these terms to formulate a necessary condition for a 
positive abstract $(3,d)$-GKM graph to be Hamiltonian: the Kirwan Class Test.

\begin{definition}\label{Def: genric ect. for abstract...}
Let $(\Gamma,\w)$ be an abstract  $(n,d)$-GKM graph.  A vector $\xi \in \R^d$ is 
\textbf{generic}  if for any edge $e\in E$ 
$$\left\langle \w(e), \xi\right\rangle_{\R^d} \neq 0,$$ 
where $\left\langle \cdot, \cdot \right\rangle_{\R^d} $ is the standard scalar product on $\R^d$.\\
For a vertex $v\in V$, its \textbf{index} $\lambda(v)$ with respect to $\xi$ is equal to the number of edges that satisfy
\begin{align}\label{EQ: Def: genric ect. for abstract... }
e\in E_v^i\quad \text{and}\quad\left\langle \w(e), \xi\right\rangle_{\R^d} <0.
\end{align}
For a vertex $v\in V$ with $\lambda(v)=1$, its\textbf{ negative Euler class} with respect to $\xi$ is 
\begin{align*}
\tau_v^-:=\w(e),
\end{align*}
where $e$ is the unique edge that satisfies \eqref{EQ: Def: genric ect. for abstract... }.\\

An \textbf{ascending path} from 
a  vertex $v$ to another vertex $w$ is a 
$(k+1)$-tuple $(v_0,...,v_k)$ of vertices in $V$ such that $v_0=v$, $v_k=w$ and 

$$(v_{i-1},v_i)\in E  \,\,\text{and}\,\, \left\langle \w(v_i),\xi \right\rangle_{\R^d} >0. $$
Moreover, for each vertex $v\in V$, the \textbf{stable set of $v$}, denoted by $\Xi_v$, is the set of vertices $w\in V$ 
such that there 
exists an ascending path from $v$ to $w$, including $v$ itself.
\end{definition}

\begin{lemma}\label{Lemma: special KC abstract case}
Let $(\Gamma, \w)$ be a positive and Hamiltonian abstract $(3,d)$-GKM graph. Let $\xi \in \R^d$ be a generic vector
and let $v\in V$ be a vertex with $\lambda(v)=1$.
Then there 	exists a map $\gamma_v:V \rightarrow \Z^d$ such that
\begin{itemize}
\item[(i)] $\gamma_v(w)=0$ if $\lambda(w)=0,1$ and $w\notin \Xi_v$.
\item[(ii)] $\gamma_v(w)=\tau_v^-$ if $\lambda(w)=1$ and $w\in \Xi_v$.
\item [(iii)]For each edge $e=(v_1,v_2)\in E$ the difference $\gamma_v(v_1)-\gamma_{v}(v_2)$ is an integer multiple of
$\w(e)$. 
\end{itemize}  
\end{lemma}
\begin{proof}
Since $(\Gamma, \w)$ is Hamiltonian, there exits a Hamiltonian GKM space $\tham$ of dimension $6$ with 
$\operatorname{dim}(T)=d$ and a linear isomorphism $\chi: \ell_{T}^*\rightarrow \Z^d$ such that $\Gamma_{GKM}=\Gamma$ 
and $\w=\chi \circ \eta$, where $\GKM$ is the GKM graph of $\tham$. Let $\bar{\xi}\in \mathfrak{t}$ the unique vector 
that satisfies 
$$\left\langle \eta(e), \bar{\xi}\right\rangle=\left\langle \w(e), \xi\right\rangle_{\R^d}$$ 
for all $e\in E=E_{GKM}$, where  $\left\langle \cdot , \cdot \right\rangle$ is the natural pairing between 
$\mathfrak{t}$ and  $\mathfrak{t}'$. 
Note that the vector  $\bar{\xi}$ is generic for the $T$-action, as defined in Subsection \ref{SubSec: Generic Vectors and Morse Theory}, 
and the index of $v$ as a vertex is 
equal to the index of $v$ as a fixed point with respect to 
$\phi^{\bar{\xi}}$. 
Moreover, a path in $\Gamma$ is ascending with respect to $\bar{\xi}$ as in Definition \ref{Def: paths} if and 
only if  it is ascending with respect to $\xi$. Note that since $\skeleton$ is positive, we have that $\GKM$ is 
positive. So 
by Lemma \ref{Lemma: c1positive implies weak index increasing}, $\GKM$ is weak indexing increasing with respect to 
$\bar{\xi}$. Hence, by Proposition \ref{Proposition: special Kirwan class lambda=1}, there exists a class
\begin{align*}
\widehat{\gamma_{v}}\in H_T^2(M;\Z)\subset\operatorname{Maps}(V, \ell_{T}^*)
\end{align*}
such that 	
\begin{itemize}
\item $\widehat{\gamma_v}(w)=0$ if $\lambda(w)=0,1$ and $w\notin \Xi_v$ and
\item $\widehat{\gamma_v}(w)=\Lambda_v^-$ if $\lambda(w)=1$ and $w\in \Xi_v$,
\end{itemize}
where $\Lambda_v^-\in \ell_{T}^*$ is the equivariant Euler class of the negative tangent bundle at 
$v$ with respect to $\phi^{\bar{\xi}}$. Note that since $\lambda(v)=1$, the class $\Lambda_v^-$ equals the unique weight $\alpha_{v}$ such that $\left \langle \alpha_v,\bar{\xi} \right \rangle<0$, which equals $\eta(e)$ for the unique edge such that  $\left \langle \eta(e),\bar{\xi} \right \rangle<0$. So $\Lambda_v^-$ is mapped under $\chi$ to $\tau_v^-=\w(e)$.
Therefore the map
\begin{align*}
\gamma_v=\chi \circ \widehat{\gamma_{v}}: V\rightarrow \Z^d
\end{align*}
satisfies the properties (i) and (ii). Moreover, since $\widehat{\gamma_{v}}\in H_T^2(M;\Z)$, we have that 
for each edge $e=(v_1,v_2)\in E$, the difference $\widehat{\gamma_v}(v_1)-\widehat{\gamma_{v}}(v_2)$ is an integer 
multiple of $\eta(e)$. So $\gamma_v$ satisfies also (iii). 
\end{proof}

\begin{definition}\label{Def: weight sum map}
Let $(\Gamma, \w)$ be an abstract $(n,d)$-GKM graph. 
The \textbf{weight sum map} for this abstract GKM graph is the map $\varphi: V \rightarrow \Z^d$ given by
\begin{align*}
\varphi(v)=- \sum_{e\in E_v^i}  \w(e).
\end{align*}
Moreover, for a given vector $\xi \in \R^\delta$ the $\xi$-component,  $\varphi^\xi: V \rightarrow \R$, of $\varphi$
is given by
\begin{align*}
\varphi^\xi(v)=\left\langle \varphi(v), \xi\right\rangle_{\R^d}. 
\end{align*}
\end{definition}

\begin{lemma}\label{Lemma: weight sum map positive}
Let $(\Gamma, \w)$ be an abstract $(n,d)$-GKM graph that is positive.  Let $\xi \in \R^d$ be a generic 
vector. For any edge $e=(v,w)$ with $\left\langle \w(e), \xi 
\right\rangle_{\R^d}<0 $
we have $$\varphi^\xi(v)>\varphi^\xi(w).$$ 
\end{lemma}
\begin{proof}
Let $e=(v,w)$ be an edge with $\left\langle \w(e), \xi 
\right\rangle_d<0 $ . By the construction of the weight sum map we have 
\begin{align*}
	\varphi(v)-\varphi(w)= - \mathcal{C}_1(e) \cdot {\w}(e),
	\end{align*}
	 Since $(\Gamma, \w)$ is positive, we have in 
	particular $\mathcal{C}_1(e)>0$. So we have that
	\begin{align*}
	\varphi^\xi(v)-\varphi^\xi(w)= -\mathcal{C}_1(e) \cdot \left\langle{\w}(e), \xi \right\rangle_{\R^d} >0.
	\end{align*}
\end{proof}

\begin{remark}\label{Rem: compute KC lambda=1 on higher points}
Let $(\Gamma, \w)$ be an abstract $(3,d)$-GKM graph that is positive and Hamiltonian.
Let $\xi \in \R^d$ be a generic vector.
Given $v\in V$ with $\lambda(v)=1$, let $\gamma_v: V \rightarrow \Z^d$ be a map as in Lemma \ref{Lemma: special KC 
abstract case}.  
We know $\gamma_v(w)$ for vertices with $\lambda(w)\leq 1$.
For vertices with $\lambda(w)\geq2$  it can be computed.
Let $\varphi^\xi: V \rightarrow \R$ be the $\xi$-component of the weight sum map. 
Let $w_1,\dots,w_k$ be the vertices with  index greater than one, ordered so that 
\begin{align*}
\varphi^\xi(w_1)\leq\varphi^\xi(w_2)\leq \dots \leq \varphi^\xi(w_k).
\end{align*}
Suppose we have already computed $\gamma_v(w_i)$ for  $i=1,...,j$, where $j<k$. 
Since $\lambda(w_{j+1})\geq 2$, there exist two different vertices $r_1,r_2\in V$ with 
$e:=(w_{j+1},r_1),e':=(w_{j+1},r_2)\in E$ and 
\begin{align*}
\left\langle {\w}(e), \xi\right\rangle_{\R^d} <0 \quad \text{and}\quad \left\langle {\w}(e'), \xi\right\rangle_{\R^d} 
<0.
\end{align*}
By Lemma \ref{Lemma: weight sum map positive} we have $\varphi^\xi(r_1)<\varphi^\xi(q_{j+1})$. If $\lambda(r_1)\geq2$,  
we 
have  $r_1=w_i$ for some $i=1,...,j$, so we have already  computed $\gamma_v(r_1)$.  If $\lambda(r_1)\leq 1$, we also 
know $\gamma_v(r_1)$. For the same reason we know $\gamma_v(r_2)$.
By property (iii) of $\gamma_v$, there exist integers  $A_1$ and $A_2$ such that
\begin{align*}
\gamma_v(w_{j+1})-\gamma_v(r_1)=A_1 \cdot {\w}(w_{j+1}, r_1) \quad \text{and} \quad 
\gamma_v(w_{j+1})-\gamma_v(r_2)=A_2 \cdot (w_{j+1}, r_2). 
\end{align*}
So we have
\begin{align*}
\gamma_v(r_1)-\gamma_v(r_2)=-A_1 \cdot {\w}(w_{j+1}, r_1) + A_2 \cdot {\w}(w_{j+1}, r_2).
\end{align*}
Since we know $\gamma_v(r_1)$ and $\gamma_v(r_2)$ and since ${\w}(w_{j+1}, r_1)$ and ${\w}(w_{j+1}, r_2)$ are 
linearly independent, we can compute $A_1$ and $A_2$. So we have
\begin{align*}
\gamma_v(w_{j+1})&=\gamma_v(r_1)+ A_1 \cdot {\w}(w_{j+1}, r_1)  \,\, \text{resp.}\\
\gamma_v(w_{j+1})&=\gamma_v(r_2)+ A_2 \cdot {\w}(w_{j+1}, r_2).
\end{align*}
\end{remark}

\begin{remark}\label{Rem: KC Test}
Let $(\Gamma, \w)$ be an abstract $(3,d)$-GKM graph that  is positive.
Let $\xi \in \R^d$ be a generic vector. If for each vertex $v\in V$ with $\lambda(v)=1$, there exists a class
$\gamma_v: V \rightarrow \Z^d$ as in Lemma \ref{Lemma: special KC abstract case}, then we say 
that $(\Gamma, \w)$ satisfies the\textbf{ Kirwan Class Test} with respect to $\xi$.
Note that  if $(\Gamma, \w)$ is Hamiltonian then the Kirwan Class Test is satisfied.
By Remark \ref{Rem: compute KC lambda=1 on higher points}, it is easy to check whether the Kirwan Class Test is satisfied. 
\end{remark} 

\begin{example}\label{Beispie} 
Consider the abstract $(3,2)$-GKM graph $(\Gamma, \w)$ with six vertices that we visualize in Figure \ref{Figure: 
Example 
24Rule but to Hamiltonian}.  The big dots mark the vertices resp. the images of the vertices under the weight sum map 
$\varphi$ in $\Z^2$. The small dots mark the other points of $\Z^2$ that lie in the convex hull of the images of the 
vertices under $\varphi$. For two different vertices  $v_i$ and $v_j$, the pair $(v_i,v_j)$ is an edge of the graph 
$\Gamma$ if and only if there exists a line segment
from the dot $v_i$ to the dot $v_j$.  
If $(v_i,v_j)$ is an edge and the line segment is blue, then $\w(v_i,v_j)$ is the primitive vector that points in the direction of the oriented line segment from $v_i$ to $v_j$. If the line segment is red, then $\w(v_i,v_j)$ is the 
double of this primitive vector. For example 
\begin{align*}
\w(v_1,v_2)=(-1,2)^\intercal \quad \text{and} \quad \w(v_1,v_6)=(2,0)^\intercal.
\end{align*}
\begin{figure}[h]
\begin{center}
\includegraphics[width=7cm]{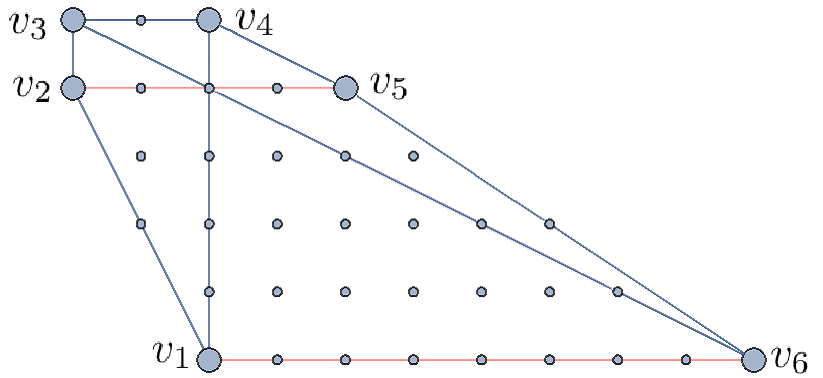}
\caption{The abstract GKM graph of Example  \ref{Beispie} }
\label{Figure: Example 24Rule but to Hamiltonian}
\end{center}
\end{figure}

Moreover, let $(v_{i},v_j)\in E$ be an edge and let $l$ be the number of points in $\Z^2$ that lie on the line 
segment between the dots  $v_i$ and $v_j$, then
\begin{align*}
	\mathcal{C}_1(v_i,v_j) = \begin{cases}
	l-1\quad\,\, \,\,\,\,\,\,\text{if the line segment is blue,} \\
	\frac{1}{2}(l-1) \quad\text{if the line segment is red.}
	\end{cases}
\end{align*}
For example, 
\begin{align*}
\mathcal{C}_1(v_1,v_2)=2 \quad \text{and} \quad \mathcal{C}_1(v_1,v_6)=4.
\end{align*}
So this abstract GKM graph is positive and satisfies the $24$-Rule.

We show by contradiction that this abstract GKM graph is not Hamiltonian. \newline
The vector  $\xi=(1,1)^\intercal\in \R^2$ is generic for this abstract GKM graph. The indices, the stable sets of 
vertices,
and the weight sum map and its $\xi$-component are given in Table \ref{table:1}.

\begin{table}[h]
\begin{center}
\begin{tabular}{|c c c c c|}
\hline
& Index $\lambda$ & $\varphi$&$\varphi^\xi$& $\Xi_{v_i}$  \\
\hline
$v_1$ &0& (-1,-3$)^\intercal$ & -4&$\{v_1,v_2,v_3,v_4,v_5,v_6\}$ \\
$v_2$ & 1&(-3,1$)^\intercal$ & -2& $\{v_2,v_3,v_4,v_5,v_6\}$ \\
$v_3$ & 1&(-3,2$)^\intercal$ & -1&$\{v_3,v_4,v_5,v_6\}$ \\
$v_4$ & 2&(-1,2$)^\intercal$& 1&$\{v_4,v_5,v_6\}$ \\
$v_5$& 2&(1,1$)^\intercal$ & 2& $\{v_5,v_6\}$ \\
$v_6$ & 3&(7,-3$)^\intercal$ & 4& $\{v_6\}$ \\
\hline
\end{tabular}
\caption{Values of $\varphi$ resp. $\varphi^\xi$ and $\Xi_{v_i}$.}
\label{table:1}
\end{center}
\end{table}
The vertex $v_2$ has index one and the unique edge $e\in E$ with $e\in E_{v_4}^i$ and $\left\langle 
\w(e),\xi\right\rangle_{\R^2}<0$ is $(v_2,v_1)$. Hence, the negative Euler class of $v_2$ with respect to
$\xi$ is 
\begin{align*}
\tau_{v_2}^-=\w(v_2,v_1)=(1,-2)^\intercal.
\end{align*} 
Now assume by negation that $(\Gamma,\w)$ is Hamiltonian. By Lemma \ref{Lemma: special KC abstract case} there exists a 
map
$\gamma_{v_2}: V \rightarrow \Z^2$ such that 
$$\gamma_{v_2}(v_1)=(0,0)^\intercal,  \quad \gamma_{v_2}(v_2)=\gamma_{v_2}(v_3)=(1,-2)^\intercal $$
and for any edge $(v_i,v_j)$ the difference $\gamma_{v_2}(v_i)-\gamma_{v_2}(v_j)$ is an integer multiple of 
$\w(v_i,v_j)$. The vertices with index greater than one are $v_4$,$v_5$ and $v_6$
and we have
$$\varphi^\xi(v_4)<\varphi^\xi(v_5)<\varphi^\xi(v_6).$$
As in Remark \ref{Rem: compute KC lambda=1 on higher points}, we can compute $\gamma_{v_2}(v_4)$,  after that 
 $\gamma_{v_2}(v_5)$, and then $\gamma_{v_2}(v_6)$.
 
 We have $(v_4,v_1), (v_4,v_3)\in E$. So there exist integers $A_1$ and $A_2$ such that 
 \begin{align*}
 \gamma_{v_2}(v_4)-\gamma_{v_2}(v_1)=A_1\cdot {\w}(v_4,v_1)\quad \text{and} \quad
 \gamma_{v_2}(v_4)-\gamma_{v_2}(v_3)=A_2\cdot {\w}(v_4,v_3). 
 \end{align*}
 This leads to
 \begin{align*}
 (-1,2)^\intercal=\gamma_{v_2}(v_1)-\gamma_{v_2}(v_3)= -A_1 \cdot (0,-1)^\intercal + A_2 \cdot (-1,0)^\intercal.
 \end{align*}
 Hence, $A_1=2$ , $A_2=1$ and $\gamma_{v_4}(v_2)=(0,-2)^\intercal$.
 
 Moreover, we have $(v_5,v_2), (v_5,v_4)\in E$. So there exist integers $B_1$ and $B_2$ such that 
 \begin{align*}
 \gamma_{v_2}(v_5)-\gamma_{v_2}(v_2)=B_1\cdot {\w}(v_5,v_2)\quad \text{and} \quad
 \gamma_{v_2}(v_5)-\gamma_{v_2}(v_4)=B_2\cdot {\w}(v_5,v_4). 
 \end{align*}
 This leads to
 \begin{align*}
 (1,0)^\intercal=\gamma_{v_2}(v_2)-\gamma_{v_2}(v_4)= -B_1 \cdot (-2,0)^\intercal + B_2 \cdot (-2,1)^\intercal.
 \end{align*}
But this implies $B_1=\frac{1}{2}\notin \Z$, which is a contradiction.
\end{example}

\end{subsection}

\end{section}

\begin{section}{About the Computer Programs}
In order to list all GKM graphs that are coming from positive Hamiltonian GKM spaces of dimension six, we need to 
consider all positive $3$-GKM skeletons that satisfy the $24$-Rule up to isomorphism. The underlying graph of such a 
skeleton is 
connected, simple and $3$-valent, i.e., a \textbf{cubic graph},  with at most $16$  
vertices; see Corollary \ref{Cor: 24Rule Skeleton}.  
The data base \cite{DataBaseCubicGraphs} provides the complete list of cubic graphs with at most $16$
vertices. Here, for $X=4,6,...,16$, we refer to the list that contains all (isomorphism classes of)
 cubic graphs with exactly $X$ vertices as $CubicGraphsX$. We refer to the graph 
 that is the $j$-th element of the list $CubicGraphsX$ as CX.j. 
 Let $\Gamma$ be a graph in the list  $CubicGraphsX$. Note that the number of edges is $3\cdot X$ and that if there exists a 
 map $\w:E\rightarrow \Z^2 \setminus \{0\}$ such that $(\Gamma,\w)$ is an abstract positive $(3,d)$-GKM graph that is 
 Hamiltonian, then for any orientation $\sigma$ of the edge set and any ordering $\operatorname{Ord}(E^\sigma)$ there 
 exists an integer vector $\mathbf{d}\in \Z_{>0}^{\frac{3}{2}X}$
 with 
 $$d_1+...+d_{\frac{3}{2}X}=24$$
 such that the GKM skeleton $\skeleton$ has defect greater or equal to $d$ and satisfies the Kernel Condition (K1), see Proposition \ref{Pro: defect and FS} and Corollary \ref{Cor: (K1) is a necessary condition}; note that since $n=3$, we have $d \geq 2$, see Remark \ref{Rem: relation n, d abstract GKM graph}.
 The following computer program computes the list $GraphsToBeConsideredX$ that contains all cubic graphs
 with $X \leq 16$ vertices that satisfy these conditions.

\begin{center}
\underline{Program: Compute the List $GraphsToBeConsideredX$ }\\
\end{center}
\begin{itemize}
\item[1.]  Compute the list $Permutations$ that contains all integer vectors $\mathbf{d}\in \Z_{>0}^{\frac{3}{2}X}$ 
that satisfy $$d_1+...+d_{\frac{3}{2}X}=24.$$
\item[2.] Let $GraphsToBeConsideredX$=$\left\lbrace \right\rbrace $ be the empty list.
\item[3.] Do the following for each graph $\Gamma$ in the list $CubicGraphsX$. 
Fix an orientation $\sigma$ of the edge set and an ordering $\operatorname{Ord}(E^\sigma)$.\newline 
Start with $j=1$ and do the following as long as $j\leq \operatorname{Length}[Permutations]$.\\
Let $\mathbf{d}$ be the $j$-th element of the list $Permutations$. 
\begin{itemize}
\item[$\bullet$] If the GKM skeleton $\skeleton$ has a defect greater than one and satisfies the Kernel Condition (K1),
then add $\Gamma$ to the list $GraphsToBeConsideredX$ and set $$j\rightarrow \operatorname{Length}[Permutations]+1. $$
\item[$\bullet$] If not, then set $j\rightarrow j+1.$
\end{itemize}
\end{itemize}	

The result of this program is given in Table \ref{table:2}.  Note  that in total there are only $14$ 
cubic graphs that belong to the list $GraphsToBeConsideredX$ for some $X=4,6,\dots,16$. 
If a $3$-GKM skeleton supports a positive and Hamiltonian abstract GKM graph, then the underlying graph is isomorphic to one of these $14$ graphs.    
\begin{table}[h]
\begin{center}
\begin{tabular}{|c| c |c|} 
\hline
X & Number of Cubic Graphs with X vertices& 
$GraphsToBeConsideredX$  \\ 
\hline
4 & 1 & C4.1 \\ 
6 & 2 & C6.1, C6.2 \\
8 & 5 & C8.2, C8.4, C8.5 \\
10 & 19& C10.11, C10.15, C10.16, C10.17 \\
12& 85 & C12.68, C12.69, C12.71, C12.74 \\
14 & 509 & empty list  \\
16 & 4060& empty list\\
\hline
&total: 4681& total: 14\\
\hline
\end{tabular}
\caption{Lists of graphs that need to be considered.}
\label{table:2}
\end{center}
\end{table}

With the following computer program, we analyze the 14 cubic graphs.
Recall that if a $3$-GKM skeleton supports an abstract $(3,d)$-GKM graph then for the defect $\delta$ of the skeleton we have $\delta \geq d \geq 2$. 
If $\delta=d$ then the Kernel Condition (K2) is a necessary condition for a $3$-GKM skeleton to support an abstract $(3,d)$-GKM graph;  see Corollary \ref{Cor: low defect and (K2)}. 
If the Kernel Conditions (K1) and (K2) both hold, the $3$-GKM skeleton supports a unique -up to isomorphism 
and projection- abstract $(3,\delta)$-GKM graph; 
see Proposition 
\ref{Pro: (K1)+(K2) gives a unique GKM graph}. 
In case $\delta=2$ and both Conditions (K1) and (K2) hold, we use the Kirwan Class Test to detect if the 
abstract GKM graph fails to be Hamiltonian; see Lemma \ref{Lemma: special KC abstract case} and Remark \ref{Rem: KC Test}. 
If  $\delta=3$ and the Kernel Condition (K2) does not hold, we use the Projection Test (see Subsection \ref{subsec:ProjectionTest}) to
detect if the $3$-GKM skeleton fails to support an abstract GKM graph.

\begin{center}
\underline{Program: Analyze the Cubic Graphs }\\
\end{center}
\begin{itemize}
\item[1.] Fix an orientation $\sigma$ of the edge set and an ordering $\operatorname{Ord}(E^\sigma)$.	
\item[2.]  Compute the list $Permutations1$ that contains all integer vectors $\mathbf{d}\in \Z_{>0}^{\frac{3}{2}X}$ 
that satisfy 
$$d_1+...+d_{\frac{3}{2}X}=24.$$
\item[3.] Let $Permutations2$ be the list that contains only the first element of  $Permutations1$.\newline
Start with $j=2$ and do the following as long as $j\leq \operatorname{Length}[Permutations1]$.\newline
Let $\mathbf{d}$ be the $j$-th element of the list $Permutations1$. 
If there exists no $j'<j$ such that $\skeleton$ and $(\Gamma, \sigma, \operatorname{Ord}(E^\sigma) \mathbf{d'})$
are isomorphic GKM skeletons, where $\mathbf{d'}$ is the $j'$-th element of the list $Permutations1$,
then add $\mathbf{d}$ to the list $Permutations2$. Set $j\rightarrow j+1.$

\item[4.] Start with $j=1$ and do the following as long as $j\leq \operatorname{Length}[Permutations2]$.\newline
Let $\mathbf{d}$ be the $j$-th element of the list $Permutations2$. 
\begin{itemize}
\item[4.1] If the GKM skeleton $\skeleton$ has a defect of two  and satisfies the Kernel 
Conditions (K1) and (K2), apply the Kirwan class test to the GKM skeleton. If the Kirwan Class test is 
satisfied, then compute a map $\w:E\rightarrow \Z^2$ such that $(\Gamma, \w)$  is an abstract $(3,2)$-GKM skeleton
supported by  $\skeleton$. 
\begin{center}
	Output: Give $(\Gamma, \w)$.
\end{center}
\item[$4.2$] If the GKM skeleton $\skeleton$ has a defect of three  and satisfies the Kernel 
Conditions (K1) and (K2), then compute a map $\w:E\rightarrow \Z^3$ such that $(\Gamma, \w)$ is an abstract $(3,3)$-GKM 
skeleton
supported by  $\skeleton$. 
\begin{center}
Output: Give $(\Gamma, \w)$.
\end{center}
\item[$4.3$] If the GKM skeleton $\skeleton$ has a defect of three  and satisfies the Kernel 
Condition (K1), but not (K2), the apply the Projection Test. If the projection makes no statement, 
then give the following output.
\begin{center}
Output: The permutation $\mathbf{d}$ need to be considered.
\end{center}
 \item[$4.4$] If the GKM skeleton $\skeleton$ has a defect greater than four  and satisfies the Kernel 
 Condition (K1), then give the following output.
 \begin{center}
 	Output: The permutation $\mathbf{d}$ need to be considered.
 \end{center}
\end{itemize}
\end{itemize}

Apply this program to the $14$ cubic graphs that belong to some list $GraphsToBeConsideredX$. We have the following 
result. 
\begin{proposition}\label{Pro: Result Program}
Let $\skeleton$ be a $3$-GKM skeleton that is positive and that satisfies the $24$-Rule. If $\skeleton$ supports an 
abstract GKM graph $(\Gamma, \w)$, then its defect $\delta$ is equal to $2$ or $3$. Moreover, the followings hold.
\begin{itemize}
\item If $\delta=2$ and if $(\Gamma, \w)$ satisfies the Kirwan class test, then   
$(\Gamma, \w)$ is isomorphic to one of the seven abstract $(3,2)$-GKM graphs that are listed in Appendix A. 
\item If $\delta=3$, then  $(\Gamma, \w)$  is a $(3,3)$-abstract GKM graph that comes from a
smooth and reflexive polytope or it is the projection of such a $(3,3)$-abstract GKM graph. 
\end{itemize} 
\end{proposition}

\begin{subsection}{Proof of Theorem \ref{ManiThm: sympetic fano GKM graphs coming from fano threefolds.}}
Recall that if $X$ is a smooth Fano variety endowed with a holomorphic action of an algebraic torus $T_{\C}$,
then polarisation of $X$ by its anticanonical line bundle induces a symplectic form $\omega_X$
on $X$ such that the induced $T$-action on $X$ is Hamiltonian with respect to $\omega_X$ and the induced almost complex structure $J$ is compatible with $\omega_X$. Here, $T_{\C}$ is the 
complexification of $T$.

\begin{proof}[Proof of Theorem \ref{ManiThm: sympetic fano GKM graphs coming from fano threefolds.}]
Let $\tham$ be a positive six-dimensional Hamiltonian GKM space of complexity one and let $\GKM$ be its GKM graph.
By Proposition \ref{Pro: Result Program}, $\GKM$ is either isomorphic to one of the seven GKM graphs that are listed 
in Appendix A or it is the projection of a GKM graph that comes from a smooth and reflexive polytope of dimension 
three. \\
The seven GKM graphs in Appendix A are isomorphic to GKM graphs that are coming from holomorphic GKM actions on  
smooth Fano varieties, as specified in the appendix.\\ 
If $\GKM$ is the projection of a GKM graph that comes from a smooth and reflexive polytope $\Delta$ of dimension three,
then by Delzant's classification, Theorem \ref{Thm: Delzant}, there exists a compact symplectic toric manifold 
$(\widetilde{M}, \widetilde{\omega}, T \times S^1, \widetilde{\phi})$ of dimension six such that 
$\widetilde{\phi}(\widetilde{M})=\Delta$ and an integrable almost complex structure $J$ on $\widetilde{M}$ that is 
$(T\times S^1)$-invariant and compatible with $\widetilde{\omega}$.  By Proposition \ref{Pro: toricReflMono}, since the moment map
polytope is reflexive, $(\widetilde{M},\widetilde{\omega})$ is positive 
monotone. Hence, due to the Kodaira Embedding Theorem (see \cite[Sect. 14.4]{McDuffSalamon}), $\widetilde{M}$ with the 
complex atlas induced by $J$ is a smooth 
Fano threefold. By Lemma \ref{Lemma: Projections abstract 
GKM graphs}, there exists a two-dimensional subtorus $H$ of $T\times S^1$ such that the $H$-action is Hamiltonian and 
GKM  and, moreover, $\GKM$ is isomorphic to the GKM graph of the $H$-action on $\widetilde{M}$.
\end{proof}
\end{subsection}
\end{section}

\appendix

\begin{section}{Positive Hamiltonian GKM Graphs in Dimension Six}

In this appendix we list all positive Hamiltonian GKM graphs in dimension six that are not projections of GKM graphs 
that are coming from a smooth and 
reflexive polytope. Up to isomorphism, there exist exactly seven such graphs.
We visualize these graphs as in Example \ref{Beispie}. Namely, the big dots mark the vertices. For two different 
vertices  $v_i$ and $v_j$, the pair $(v_i,v_j)$ is an edge of the graph $\Gamma$ if and only if there exists a line 
segment connecting the big dots that correspond to dots $v_i$ and $v_j$.  
If $(v_i,v_j)$ is an edge and the line segment is blue, then $\w(v_i,v_j)$ is the primitive vector that points in the 
direction of the oriented line segment from $v_i$ to $v_j$. If the line segment is red, then $\w(v_i,v_j)$ is the 
double of this primitive vector. All of these GKM graphs are coming from a holomorphic GKM on smooth Fano varieties.
Namely,  S\"uss  classified complexity one actions of algebraic tori on smooth Fano threefold \cite{suss}.
Recall that such an action induces a holomorphic and Hamiltonian action of a compact torus $T=(S^1)^2$, where here 
Hamiltonian 
means with respect to the pullback of the Fubini study form.  In  \cite{suss} the  corresponding Duistermaat-Heckman 
measure \cite{DuistermaatHeckman} is given. Note that the  Duistermaat-Heckman measure contains the information if the 
action is GKM; moreover if the action is GKM then the GKM graph can be recovered from the  Duistermaat-Heckman measure.
Fano threefolds are classified by Mori and Mukai \cite{MoriMukai}. There exists 105 families of Fano threefolds.\\
In our list of GKM graphs, by ID we refer to the family of Fano threefold; according to \cite{MoriMukai}.\\
\vspace{0.5cm}

\noindent\framebox{ID 1.16:\hspace{1cm} Hypersurface of degree $2$ in $\C P^4$/ Quadric threefold\hspace{4cm}}\\
\begin{minipage}[h]{0.3\textwidth}
\includegraphics[width=0.7\textwidth]{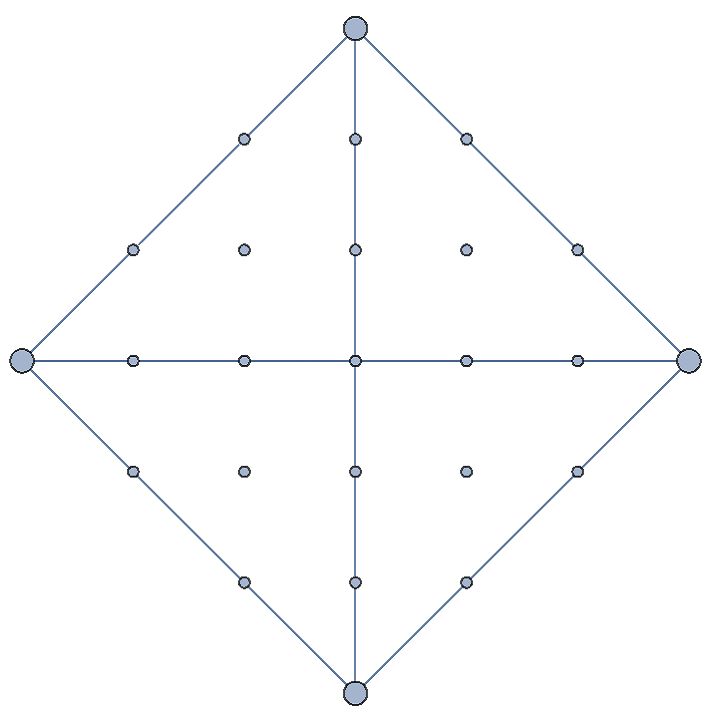}
\end{minipage}
\begin{minipage}[h]{0.15\textwidth}
\end{minipage}
\begin{minipage}[h]{0.5\textwidth}
\begin{itemize}
\item Second Betti Number: 1
\item Index: 3
\item $\int_M (c_1(M))^3=54$
\item Underlying Graph: C4.1
\end{itemize}
\end{minipage}\\
\newpage
\noindent\framebox{ID 2.24:\hspace{1cm} Divisor on $\C P^2 \times \C P^2$ of bidegree $(1,2)$\hspace{5.8cm}}\\
\begin{minipage}[h]{0.3\textwidth}
\includegraphics[width=0.7\textwidth]{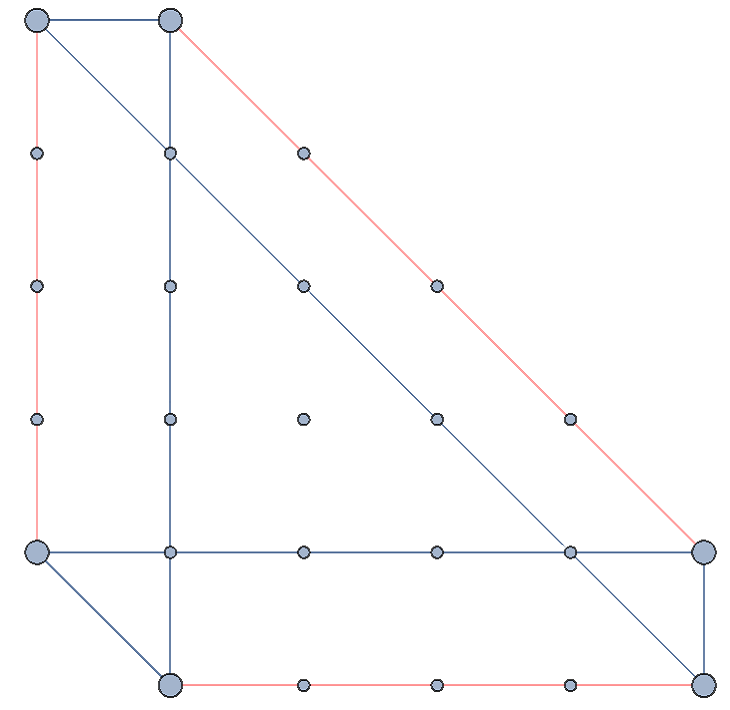}
\end{minipage}
\begin{minipage}[h]{0.15\textwidth}
\end{minipage}
\begin{minipage}[h]{0.5\textwidth}
\begin{itemize}
\item Second Betti Number: 2
\item Index: 1
\item $\int_M (c_1(M))^3=30$
\item Underlying Graph: C6.2
\end{itemize}
\end{minipage}\\
\framebox{ID 2.29:\hspace{1cm} Blowup of 1.16 in a conic\hspace{8.1cm}}\\
\begin{minipage}[h]{0.3\textwidth}
	\includegraphics[width=0.7\textwidth]{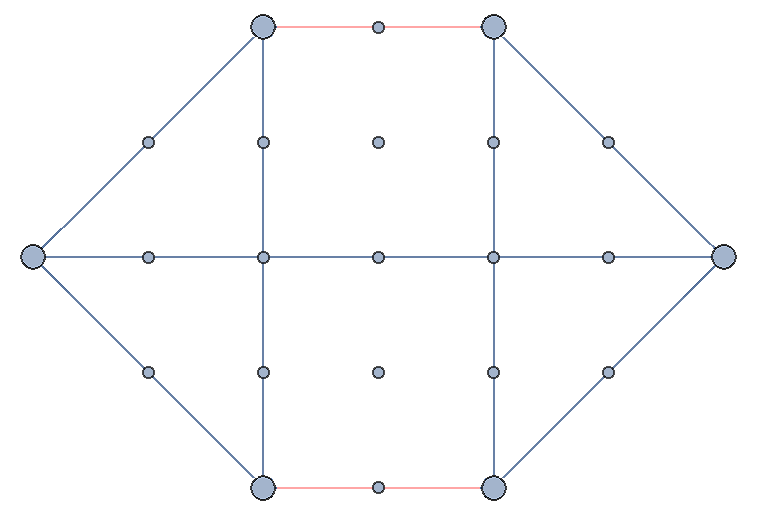}
\end{minipage}
\begin{minipage}[h]{0.15\textwidth}
\end{minipage}
\begin{minipage}[h]{0.5\textwidth}
\hspace{0.2cm}
\begin{itemize}
\item Second Betti Number: 2
\item Index: 1
\item $\int_M (c_1(M))^3=40$
\item Underlying Graph: C6.1
\end{itemize}
\end{minipage}\\
\framebox{ID 2.31:\hspace{1cm} Blowup of 1.16 in a line\hspace{8.35cm}}\\
\begin{minipage}[h]{0.3\textwidth}
\includegraphics[width=0.7\textwidth]{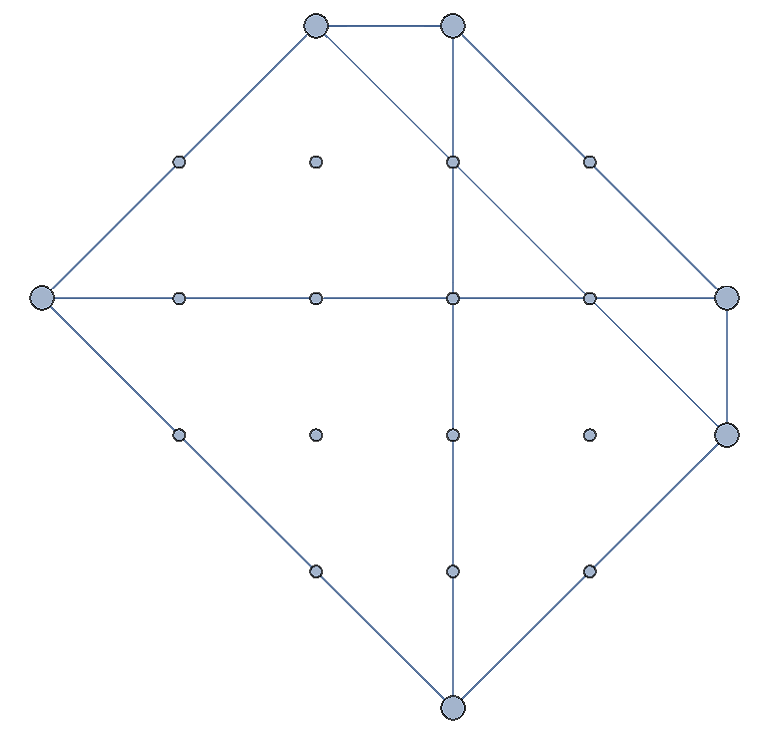}
\end{minipage}
\begin{minipage}[h]{0.15\textwidth}
\end{minipage}
\begin{minipage}[h]{0.5\textwidth}
\begin{itemize}
\item Second Betti Number: 2
\item Index: 1
\item $\int_M (c_1(M))^3=46$
\item Underlying Graph: C6.2
\end{itemize}
\end{minipage}\\
\framebox{ID 2.32:\hspace{1cm} Divisor on $\C P^2\times \C P^2$ of bidegree $(1,1)$\hspace{5.8cm}}\\
\begin{minipage}[h]{0.3\textwidth}
	\includegraphics[width=0.7\textwidth]{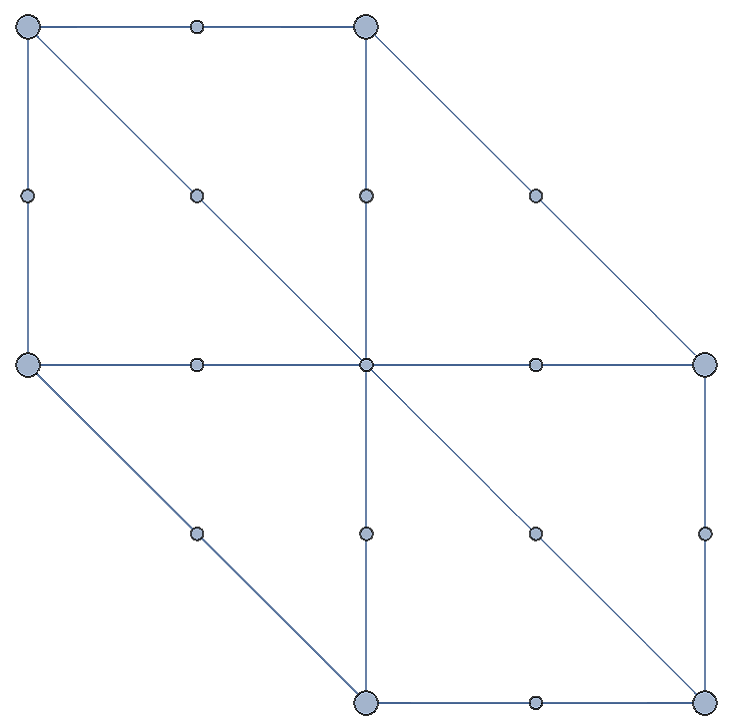}
\end{minipage}
\begin{minipage}[h]{0.15\textwidth}
\end{minipage}
\begin{minipage}[h]{0.5\textwidth}
\begin{itemize}
\item Second Betti Number: 2
\item Index: 2
\item $\int_M (c_1(M))^3=48$
\item Underlying Graph: C6.2
\end{itemize}
\end{minipage}\\

\begin{flushleft} \end{flushleft}
\framebox{ID 3.10:\hspace{1cm} Blowup of 1.16 in the disjoint union of $2$ conics\hspace{4.75cm}}\\
\begin{minipage}[h]{0.3\textwidth}
	\includegraphics[width=0.7\textwidth]{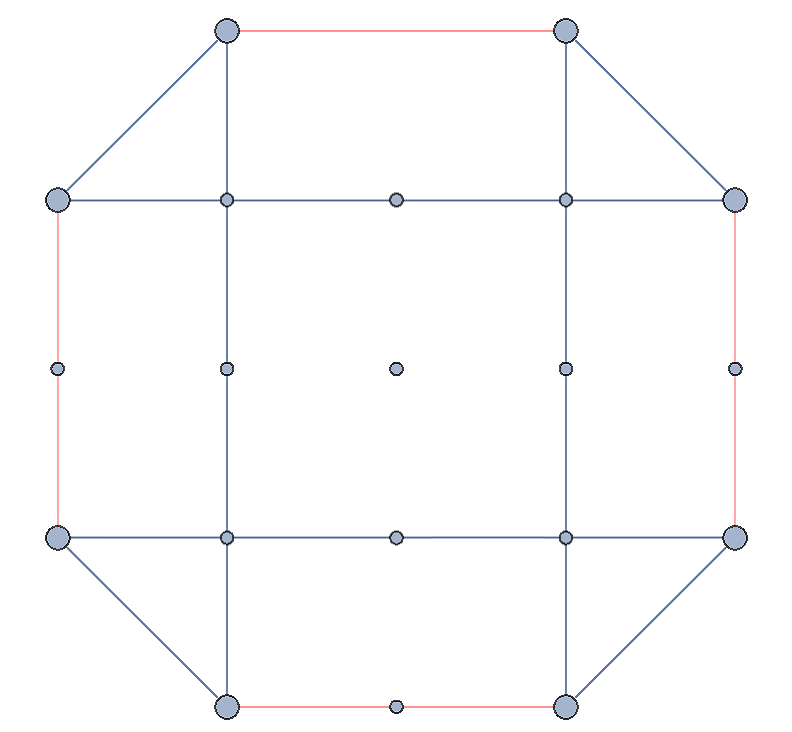}
\end{minipage}
\begin{minipage}[h]{0.15\textwidth}
\end{minipage}
\begin{minipage}[h]{0.5\textwidth}
\begin{itemize}
\item Second Betti Number: 3
\item Index: 1
\item $\int_M (c_1(M))^3=26$
\item Underlying Graph: C8.4
\end{itemize}
\end{minipage}\\
\framebox{ID 3.20:\hspace{1cm} Blowup of 1.16 in the disjoint union of $2$  lines\hspace{5.0cm}}\\
\begin{minipage}[h]{0.3\textwidth}
\includegraphics[width=0.7\textwidth]{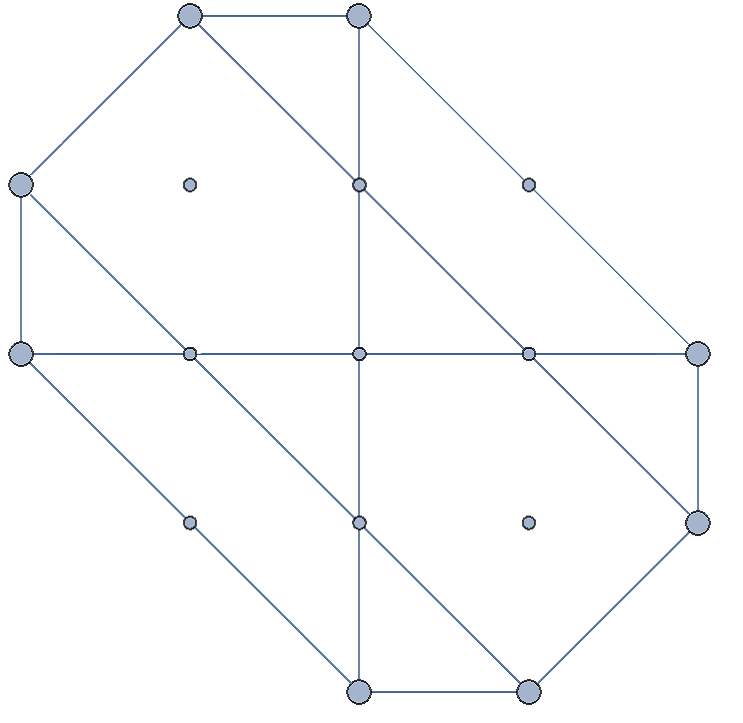}
\end{minipage}
\begin{minipage}[h]{0.15\textwidth}
\end{minipage}
\begin{minipage}[h]{0.5\textwidth}
\begin{itemize}
\item Second Betti Number: 3
\item Index: 1
\item $\int_M (c_1(M))^3=38$
\item Underlying Graph: C8.5
\end{itemize}
\end{minipage}
\end{section}

\begin{section}{Alternative Proof of Theorem \ref{ManiThm: coprime weights GKM graph determines eq. 
cohomology}}\label{sec: alt proof}
In this section we give an alternative proof for Theorem \ref{ManiThm: coprime weights GKM graph determines eq. 
cohomology} by using Kirwan classes.
As in subsection \ref{sec: kirwan}, we consider the equivariant cohomology $H_T^*(M;\Z)$ of a Hamiltonian GKM space as 
a subset of 
$$\operatorname{Maps}\left( M^T, H^*(BT;\Z)\right).$$

\begin{lemma}\label{Lemma: divisble condition}
Let $\tham$ be a Hamiltonian GKM space and let $\GKM$ be its GKM graph. 
Let $\alpha$ be a class in $ H^*_T(M;\Z)$. Then for each edge $e=(p,q)\in 
E_{GKM}$ the difference $\alpha(p)-\alpha(q)$ is divisible by $\eta(p,q)$ in $H^*(BT;\Z)$. 
\end{lemma}

\begin{proof}
Let $e=(p,q)\in 
E_{GKM}$ be an edge and let $S_e$ be the unique two-sphere with $S_e\cap M^T=\{p,q\}$ that is fixed by a 
codimensional one subtorus of $T$. The sphere $S_e$ is $T$-invariant and the weight of the $T$-representation of  
$T_pS_e$ is $\eta(p,q)$ and the weight of the $T$-representation of  $T_qS_e$ is $\eta(q,p)=-\eta(p,q)$. The 
equivariant Euler class of the normal bundle of $p$ resp. $q$ in $S_e$ is $\eta(p,q)$ resp. $-\eta(p,q)$. Let $\beta$ 
be a class in $H^*_T(S_e;\Z)$. Due to the ABBV localization formula (Theorem \ref{Thm:ABBV}), we have
\begin{align*}
\int_{S_e} \beta \,=\, \dfrac{\beta(p)-\beta(q)}{\eta(p,q)}.
\end{align*}
The left-hand side of this equation is an element in $H^*(BT;\Z)$. Therefore, the difference $\beta(p)-\beta(q)$
is divisible by $\eta(p,q)$ in $H^*(BT;\Z)$.\newline
The commutative diagram of inclusion maps 
\begin{center}
\begin{tikzpicture}
		\node (A) {$\lbrace  p,q\rbrace $};
		\node (B) [right=of A] {$S_e$};
		\node (C) [below=of B] {$M$};
		
		\draw[-stealth] (A)-- node[left] {\small } (B);
		\draw[-stealth] (B)-- node [below] {\small} (C);
		\draw[-stealth] (A)-- node [above] {\small} (C);
\end{tikzpicture}
\end{center}
induces a commutative diagram of ring homomorphisms
\begin{center}
\begin{tikzpicture}
\node (A) {$H^*_T(M;\Z)$};
\node (B) [right=of A] {$H^*(S_e;\Z)$};
\node (C) [below=of B] {$H^*_T(\{p,q\};\Z)$};
\draw[-stealth] (A)-- node[left] {\small } (B);
\draw[-stealth] (B)-- node [below] {\small} (C);
\draw[-stealth] (A)-- node [above] {\small} (C);
\end{tikzpicture}.
\end{center}
Let $\alpha \in H_T^*(M;\Z)$ and let $\beta \in H^*_T(S_e; \Z)$ be its image under the map $H^*_T(M;\Z)\rightarrow 
H^*_T(S_e;\Z)$. Since the images of $\alpha$ and $\beta$ under $H^*_T(M;\Z)\rightarrow H^*_T(\{p,q\};\Z)$ resp. 
$H^*_T(S_e;\Z)\rightarrow H^*_T(\{p,q\};\Z)$ are equal, we have $\alpha (p)=\beta(p)$ and $\alpha (q)=\beta(q)$.
Hence, 
$$\alpha(p)-\alpha(q)=\beta(p)-\beta(q)$$
is divisible by $\eta(p,q)$ in $H^*(BT;\Z)$.
\end{proof}

This lemma and and its converse hold for rational coefficients by a work of Goresky, Kottwitz and MacPherson \cite{GKM}.

\begin{theorem}
Let $\tham$ be a Hamiltonian GKM space and let $\GKM$ be its GKM graph.  Let  $\alpha :M^T \rightarrow H^*(BT;\Q)$ be 
a map. Then $\alpha$ lies in $H_T^*(M;\Q)$ if and only if for each edge $(p,q)\in E_{GKM}$ the difference 
$\alpha(p)-\alpha(q)$ is divisible by $\eta(p,q)$ in $H^*(BT;\Q)$. 
\end{theorem}

In the following, we prove an integer version of this theorem under the assumption that for each fixed point $p$ the 
weights of the $T$-representation on $T_pM$ are pairwise coprime (see Definition \ref{Def: coprime}). 

\begin{proposition}\label{Prop: GKM coprime}
Let $\tham$ be a Hamiltonian GKM space such that for each fixed point $p\in M^T$ the weights of the 
$T$-representation on $T_pM$ are pairwise coprime. 
Let $\GKM$ be its GKM graph. Given a map $\alpha\colon M^T \rightarrow H^*(BT;\Z)$, the following conditions are 
equivalent.
\begin{itemize}
\item[(i)] The map $\alpha$ belongs to $H_T^*(M;\Z)$.
\item[(ii)] For each edge $(p,q)\in E_{GKM}$, the difference $\alpha(p)-\alpha(q)$ is divisible by 
$\eta(p,q)$ in $H^*(BT; \Z)$. 
\end{itemize} 
\end{proposition}
\begin{proof}
That condition $(i)$ implies condition $(ii)$ follows directly from Lemma \ref{Lemma: divisble condition}.\newline 
In order to prove the converse, let $\alpha: M^T \rightarrow H^*(BT;\Z)$ be a map that satisfies condition 
	$(ii)$. Let $\xi \in \mathfrak{t}$ be a generic vector and let $p_1,\dots, p_{\left| M^T\right| } \in M^T$ be the 
	fixed 
	points, ordered such that
	\begin{align}\label{EQ:Prop: GKM coprime1}
	\phi^\xi(p_1)<\phi^\xi(p_2)\leq \dots  \leq \phi^\xi(p_{\left| M^T\right| -1}) < \phi^\xi(p_{\left| M^T\right|}).
	\end{align}
	We show by induction that for all $i=1,\dots \left| M^T\right| $, there exists a class $\beta_i\in H_T^*(M;\Z)$
	such that 
	\begin{align}\label{EQ:Prop: GKM coprime2}
	\beta_i(p_j)=\alpha(p_j) \text{  for all } j=1,\dots,i.
	\end{align}
	Note that $p_1$ is the unique fixed point on which $\phi^\xi$ attains its minimum. The unique Kirwan class 
	$\gamma_{p_1}\in 
	H^0_T(M;\Z)$ that belongs to $p_1$ is the constant map $M^T\rightarrow H^*(BT;\Z)$ that maps each fixed point to 
	the multiplicative identity of $H^*(BT;\Z)$. Hence, $\beta_1:=\alpha(p_1)\cdot \gamma_{p_1}$ is a class in  
	$H_T^*(M;\Z)$ that satisfies \eqref{EQ:Prop: GKM coprime2}. For the induction step, suppose that for a fixed 
	$i\in\left\lbrace 
	1,\dots , \left| M^T\right|-1\right\rbrace $, there exists a 
	class  $\beta_i$ in  $H_T^*(M;\Z)$ that satisfies \eqref{EQ:Prop: GKM coprime2}. Consider the fixed point 
	$p_{i+1}$. 
	Since $\phi^\xi$ does not attain its minimum at $p_{i+1}$, we have $\lambda(p_{i+1})\geq 1$. Let $r_1,\dots , 
	r_{\lambda(p_{i+1})}$ be the unique fixed points such that for all $k=1,\dots, \lambda(p_{i+1})$
	\begin{align*}
	(p_{i+1},r_k) \in E_{GKM} \, \text{ and } \phi^\xi(r_k)< \phi^\xi(p_{i+1}).
	\end{align*}   
	Since the fixed points are ordered such that \eqref{EQ:Prop: GKM coprime1} holds, we have that  $r_1,\dots , 
	r_{\lambda(p_{i+1})}\in \{p_1,\dots, p_i\}$. Consider the map 
	\begin{align*}
	\tau:=\alpha-\beta_i: M^T \rightarrow H^*(BT;\Z).
	\end{align*}
	Note that $\tau(r_k)=0$ for all $k=1,\dots, \lambda(p_{i+1})$. Since $\beta_i\in H_T^*(M;\Z)$, by Lemma 
	\ref{Lemma: divisble condition}
	\begin{align*}
	\beta_i(p_{i+1})-\beta_i(r_k)
	\end{align*}
	is divisible by $\eta(p_{i+1},r_k)$ for all $k=1,\dots, \lambda(p_{i+1})$.
	Hence, since $\alpha$ satisfies condition $(ii)$, we have 
	\begin{align*}
	\tau(p_{i+1})-\tau(r_k)=\tau(p_{i+1})
	\end{align*}
	is divisible by $\eta(p_{i+1},r_k)$ for all $k=1,\dots, \lambda(p_{i+1})$.
	Since the weights $\eta(p_{i+1},r_k)$ are pairwise coprime in $H^*(BT;\Z)$,
	there exists an element $f\in H^*(BT;\Z)$ such that 
	\begin{align*}
	\tau(p_{i+1})=f\cdot \prod_{k=1}^{\lambda(p_{i+1})}\eta(p_{i+1},r_k)= f\cdot \Lambda_{p_{i+1}}^-.
	\end{align*}
	Now let $\gamma_{p_{i+1}}$ be a Kirwan class at $p_{i+1}$, then
	\begin{align*}
	\beta_{i+1}:=\beta_i+ f\cdot \gamma_{p_{i+1}}
	\end{align*}
	is a class in $H_T^*(M;\Z)$ that satisfies \eqref{EQ:Prop: GKM coprime2}.
\end{proof}	

\begin{definition}\label{Def: isomorphism of GKM graphs induces isomorphism Maps() to Maps()}
Given two Hamiltonian GKM spaces $(M_1,\omega_1,T,\phi_1)$ and $(M_2,\omega_2,T,\phi_2)$, let $(F,\theta)$ be an 
isomorphism between the GKM graphs $(\Gamma_{1,GKM},\eta_1)$ and $(\Gamma_{2,GKM},\eta_2)$  of these spaces. Then the 
isomorphism induces a ring isomorphism
\begin{align*}
\Psi_{(F,\theta)} \, :\,\operatorname{Maps}(M_1^T \rightarrow H^*(BT;\Z)) \longrightarrow \operatorname{Maps}(M_2^T
\rightarrow  H^*(BT;\Z)), 
\end{align*}
as follows. Since $\ell_T^*=H^2(BT;\Z)$, the linear isomorphism $\theta: \ell_T^*\rightarrow \ell_T^*$ extends in a 
canonical way to a ring isomorphism $\widetilde{\theta}: H^*(BT;\Z) \longrightarrow H^*(BT;\Z)$.
The image of a map $\alpha: M_1^T \rightarrow H^*(BT;\Z)$ under $\Psi_{(F,\theta)}$ is the map 
$\Psi_{(F,\theta)}(\alpha):M_2^T \rightarrow H^*(BT;\Z)$ given by
	
\begin{align*}
\Psi_{(F,\theta)}(\alpha)(q)= \widetilde{\theta}\left( \alpha( F^{-1}(q))\right) \quad \text{for all }q\in M_2^T.
\end{align*}
\end{definition}

\begin{remark}\label{Rem:Def: isomorphism of GKM graphs induces isomorphism Maps() to Maps()}
It is clear that the map $\Psi_{(F,\theta)}$ is a ring homomorphism. Moreover, if $(F^{-1}, \theta^{-1})$ is the 
inverse isomorphism to $(F,\theta)$, then 
\begin{align*}
\Psi_{(F^{-1},\theta^{-1})} \, :\,\operatorname{Maps}(M_2^T \rightarrow H^*(BT;\Z)) \longrightarrow 
\operatorname{Maps}(M_1^T
\rightarrow  H^*(BT;\Z)), 
\end{align*}
is the inverse of $\Psi_{(F,\theta)}$. Hence, $\Psi_{(F,\theta)}$ is indeed a ring isomorphism.
\end{remark}

\begin{lemma}\label{Lemma: isomorphism of GKM grpahs induces isomorphism Maps() to Maps()}
	Given two Hamiltonian GKM spaces $(M_1,\omega_1,T,\phi_1)$ and $(M_2,\omega_2,T,\phi_2)$, let $(F,\theta)$ be an 
	isomorphism between the GKM graphs $(\Gamma_{1,GKM},\eta_1)$ and $(\Gamma_{2,GKM},\eta_2)$  of these spaces.
	The following holds.
	\begin{itemize}
		\item [(i)] For a map $\alpha: M_1^T \rightarrow H^*(BT;\Z)$, the following two conditions are equivalent.
		\begin{itemize}
			\item[(a)] For all $(p,q)\in E_{1,GKM}$, the difference $\alpha(p)-\alpha(q)$ is divisible by $\eta_1(p,q)$ 
			in 
			$H^*(BT;\Z)$.
			\item[(b)] For all $(p,q)\in E_{2,GKM}$, the difference $\Psi_{(F,\theta)}\left( \alpha\right) 
			(p)-\Psi_{(F,\theta)}\left( 
			\alpha\right) (q) $ is divisible by $\eta_2(p,q)$ in $H^*(BT;\Z)$. 
		\end{itemize}
		\item[(ii)] $\Psi_{(F,\theta)}$ maps the equivariant Chern classes of $(M_1,\omega_1,T,\phi_1)$ to the ones of 
		$(M_2,\omega_2,T,\phi_2)$.
	\end{itemize}
\end{lemma}
\begin{proof}
	By definition, $F:M_1^T\rightarrow M_2^T$ induces a bijection 
	\begin{align}\label{EQ1:Lemma: isomorphism of GKM grpahs induces isomorphism Maps() to Maps()}
	E_{1,GKM}\ni (p,q) \longmapsto (F(p),F(q))\in E_{2,GKM}
	\end{align}
	such that 
	\begin{align*}
	\eta_2(F(p),F(q))= \theta (\eta_1 (p,q)).
	\end{align*}
	\begin{itemize}
		\item[(i)] Let $\alpha: M_1^T \rightarrow H^*(BT;\Z)$ be a map and $(p,q)\in E_{1,GKM}$ be an edge such that
		\begin{align*}
		\alpha(p)-\alpha(q)= \eta_1(p,q) \cdot f
		\end{align*}
		for some $f \in H^*(BT;\Z)$. We have 
		\begin{align*}
		\Psi_{(F,\theta)}\left( \alpha\right) 
		(F(p))-\Psi_{(F,\theta)}\left( \alpha\right) 
		(F(q))&= \widetilde{\theta}\left( \alpha( p)\right)- \widetilde{\theta}\left( \alpha( 
		q)\right)\\
		&=\widetilde{\theta}\left( \alpha(p)-\alpha(q)\right)  \\
		&=\widetilde{\theta}\left( \eta_1(p,q)\cdot f\right)\\
		&= \theta\left( \eta_1(p,q)\right)\cdot \widetilde{\theta}\left( f\right) \\ 
		&=  \eta_2(F(p),F(q))\cdot \widetilde{\theta}\left( f\right),
		\end{align*}
		where $\widetilde{\theta}:H^*(BT;\Z)\rightarrow H^*(BT;\Z)$ is the ring isomorphism induced by $\theta: 
		\ell_{T}^* 
		\rightarrow \ell_{T}^*$. Since \eqref{EQ1:Lemma: isomorphism of GKM grpahs induces isomorphism Maps() to 
		Maps()} is a 
		bijection, we conclude that condition (a) implies (b). By applying the same argument to 
		$\Psi_{(F^{-1},\theta^{-1})}$, 
		we have that condition (b) implies (a).
		\item[(ii)] Fix a non-negative integer $k$ and let $c_k^T(M_1)$ resp. $c_k^T(M_2)$ be the $k$-th equivariant 
		Chern class of $(M_1,\omega_1,T,\phi_1)$ resp. $(M_2,\omega_2,T,\phi_2)$. Considered as a map $M_1^T\rightarrow 
		H^{2k}(BT;\Z)$, the class $c_k^T(M_1)$ is given by 
		\begin{align*}
		M_1^T \ni p \longmapsto \sigma_{n,k}(\eta_1(p,p_1),\dots, \eta_1(p,p_n)),
		\end{align*}
		where $\sigma_{n,k}$  is the  elementary symmetric polynomial in $n$ variables of degree $k$ and $p_1,...,p_n$
		are the fixed points with $(p,p_i)\in E_{1,GKM}$. The same holds for $c_k^T(M_2)$.
		Let $q\in M_2^T$ be the fixed point with $F(p)=q$. Then $q_1=F(p_1),..., q_n=F(p_n)$ are the fixed points with 
		$(q,q_i)\in E_{2,GKM}$. So we have
		\begin{align*}
		\Psi_{(F,\theta)}\left( c_k^T(M_1)\right) (q)&= \widetilde{\theta} \left( \sigma_{n,k}(\eta_1(p,p_1),\dots, 
		\eta_1(p,p_n)) 
		\right) \\
		&= \sigma_{n,k} \left(   \widetilde{\theta}(\eta_1(p,p_1)),\dots,\widetilde{\theta}( \eta_2(p,p_n))\right) \\
		&=\sigma_{n,k}(\eta_2(F(p),F(p_1)),\dots, \eta_2(F(p),F(p_n)))\\
		&=\sigma_{n,k}(\eta_2(q,q_1),\dots, \eta_2(q,q_n))\\
		&= c_k^T(M_2)(q).
		\end{align*}
		Hence, $\Psi_{(F,\theta)}\left( c_k^T(M_1)\right)=c_k^T(M_2)$ holds.

	\end{itemize}
	
\end{proof}

\begin{proof}[Alternative Proof of Theorem \ref{ManiThm: coprime weights GKM graph determines eq. cohomology}]
Let  $(M_1,\omega_1,T,\phi_1)$ and $(M_2,\omega_2,T,\phi_2)$ be two Hamiltonian GKM spaces such that the complexity of 
both spaces is one or zero. By Lemma \ref{Lemma: complexity one, zero dim six implies weights pairwise coprime}
for each fixed point $p\in M_1^T$  resp.\  $q\in M_2^T$ the weights of the $T$-representation on $T_pM_1$ resp.
	on $T_qM_2$ are pairwise coprime.   Let $(F,\theta)$ be an 
	isomorphism between the GKM graphs $(\Gamma_{1,GKM},\eta_1)$ and $(\Gamma_{2,GKM},\eta_1)$  of these spaces.\newline
	\textbf{(a):}  
	By Proposition \ref{Prop: GKM coprime}, we have that a map $\alpha: M_1^T\rightarrow H^*(BT;\Z)$ belongs to 
	$H_T^*(M_1;\Z)$ if and only if for each edge $(p,q)\in E_{1,GKM}$ the difference $\alpha(p)-\alpha(q)$ is divisible
	by $\eta_1(p,q)$ in $H^*(BT;\Z)$. The same holds for   $H_T^*(M_2;\Z)$. Let 
	\begin{align*}
	\Psi_{(F,\theta)} \, :\,\operatorname{Maps}(M_1^T \rightarrow H^*(BT;\Z)) \longrightarrow \operatorname{Maps}(M_2^T
	\rightarrow  H^*(BT;\Z)) 
	\end{align*}
	be the ring isomorphism that is induced by $(F,\theta)$. By Lemma \ref{Lemma: isomorphism of GKM grpahs induces 
		isomorphism Maps() to Maps()}, the restriction of $\Psi_{(F,\theta)}$ to $H_T^*(M_1;\Z)$ gives 
	a ring isomorphism from $H_T^*(M_1;\Z)$ to $H_T^*(M_2;\Z)$ that maps the equivariant Chern classes of  
	$(M_1,\omega_1,T,\phi_1)$ to the ones of $(M_2,\omega_2,T,\phi_2)$.
	We denote this ring isomorphism by 
	$$\mathbf{\varphi}: H_T^*(M_1;\Z) \rightarrow H_T^*(M_2;\Z). $$
	\textbf{(b):} Let $\alpha\in H^*_T(M_1;\Z)$  and  let $\beta\in H^*_T(M_2;\Z)$ be the image of $\alpha$ under 
	$\varphi$.
	Let $\widetilde{\theta}:H^*(BT;\Z)\rightarrow H^*(BT;\Z)$ be the ring isomorphism that is induced  by  $\theta: 
	\ell_{T}^* \rightarrow \ell_{T}^*$. Since $\varphi$ is the restriction of  $\Psi_{(F,\theta)}$ to $H_T^*(M_1;\Z)$,  
	for any 
	$f\in H^*(BT;\Z)$ the following conditions are equivalent. 
	\begin{itemize}
		\item[(1)] For each fixed $p\in M_1^T$, $\alpha(p)$ is divisible by $f$.
		\item[(2)] For each fixed $p\in M_2^T$, $\beta(p)$ is divisible by $\widetilde{\theta}(f)$.
	\end{itemize}
	Let 
	$$r_1^*: H_T^*(M_1;\Z) \rightarrow H^*(M_1;\Z)\quad \text{and} \quad r_2^*: H_T^*(M_2;\Z) 
	\rightarrow H^*(M_2;\Z)$$
	be the restriction maps. Due to the Kirwan Surjectivity Theorem \cite{Kirwan}, the kernel of 
	$r_1^*$ 
	resp. $r_2^*$ is the ideal in $H_T^*(M_1;\Z)$ resp. $H_T^*(M_2;\Z)$  generated by $H^2(BT;\Z)$. 
	This means that a class $\alpha\in H_T^*(M_1;\Z)$ belongs to the kernel of $r_1^*$ if and only if 
	there exists an element $f\in H^{2k}(BT;\Z)$ for some $k>0$ such that  for any $p\in M_1^T$, $\alpha(p)$
	is divisible by $f$. The same holds for the kernel of $r_2^*$. Since the conditions (1) and (2) are equivalent,
	we conclude that $\varphi$ gives a bijection from the kernel of $r_1^*$ to the one of $r_2^*$.
	
	So there exists a ring isomorphism $\widetilde{\varphi}: H^*(M_1;\Z)\rightarrow H^*(M_2;\Z)$ such that the 
	following 
	diagram commutes.
	\begin{center}
		\begin{tikzpicture}
		\node (A) {$H^*_T(M_1;\Z)$};
		\node (B) [below=of A] {$H^*(M_1;\Z)$};
		\node (C) [right=of A] {$H^*_T(M_2;\Z)$};
		\node (D) [right=of B] {$H^*(M_2;\Z)$};
		\draw[-stealth] (A)-- node[left] {\small $r_1^*$} (B);
		\draw[-stealth] (B)-- node [below] {\small $\widetilde{\varphi}$} (D);
		\draw[-stealth] (A)-- node [above] {\small $\varphi$} (C);
		\draw[-stealth] (C)-- node [right] {\small $r_2^*$} (D);
		\end{tikzpicture}
	\end{center}
	Note that the restriction maps  $r_1^*$ and $r_2^*$ map the equivariant Chern classes to the ordinary Chern classes.
	Hence, since $\varphi$ maps equivariant Chern classes to equivariant Chern classes, we have that 
	$\widetilde{\varphi}$
	maps the Chern classes of $(M_1,\omega_1)$ to the ones of $(M_2,\omega_2)$.
\end{proof}

\end{section}

\vspace{0.5cm}
\noindent{\Large\textbf{Declaration}}\\

\noindent \textbf{Conflict of Interest} The authors have no conflict of interest to declare that are relevant to this 
article.

\end{document}